\documentclass{amsart}
\usepackage[T1]{fontenc}

\usepackage{amsmath,amsthm,latexsym,amssymb,amsfonts,enumerate}
\usepackage{csquotes}
\usepackage[font=itshape]{quoting}
\usepackage{faktor}
\usepackage[xindy, toc]{glossaries}
\usepackage{multicol}
\usepackage[all,cmtip]{xypic} 
\usepackage{xypic,calc}
\usepackage{bm}

\usepackage{tikz}
\usepackage{tikz-cd}
\usepackage{xcolor}
\usetikzlibrary{shapes,arrows}
\usetikzlibrary{calc,
	arrows,decorations.pathmorphing,
	backgrounds,fit,positioning,shapes.symbols,chains}
\usetikzlibrary{decorations.pathmorphing}
\usetikzlibrary{decorations.markings}
\usetikzlibrary{decorations.pathreplacing,angles,quotes}
\usetikzlibrary{positioning,shadows,arrows,trees,shapes,fit}
\usetikzlibrary{mindmap}
\usepackage{standalone}
\usepackage{color}
\usepackage{float}
\usepackage{subfigure}
\usepackage{ragged2e}

\usepackage{mdwlist}
\usepackage{hyperref}
\usepackage[noabbrev, capitalize]{cleveref}
\usepackage{faktor}
\usepackage{xfrac} 
\usepackage{scalerel}
\allowdisplaybreaks

\restylefloat*{figure}

 \numberwithin{equation}{section}
 
\newtheorem{thmx}[equation]{Theorem}
\newtheorem{corx}[equation]{Corollary}
\newtheorem{lemx}[equation]{Lemma}
\numberwithin{equation}{section}

\newtheorem{thm}[equation]{Theorem}
\newtheorem{lem}[equation]{Lemma}
\newtheorem{prop}[equation]{Proposition}
\newtheorem{cor}[equation]{Corollary}

\theoremstyle{Definition}
\newtheorem{defn}[equation]{Definition}

\theoremstyle{remark}
\newtheorem{rmk}[equation]{Remark}

\newtheorem{ex}[equation]{Example}

\renewcommand{\emptyset}{\font\cmsy = cmsy10 at 10pt
	\hbox{\cmsy \char 59}
}

\newcommand\N{\mathbb{N}}

\newcommand{\ccc}{\mathbf{c}} 
\newcommand{\ddd}{\mathbf{d}}

\DeclareMathAlphabet\EuR{U}{eur}{m}{n}
\SetMathAlphabet\EuR{bold}{U}{eur}{b}{n}

\mathchardef\mathdash="2D

\newcommand\GS{\mathsf {GS}}
\newcommand\CSM{\mathsf {MO}}
\newcommand{\OOO}{\mathbb {O}}
\newcommand\Klgr{\Xi}
\newcommand{\pr}[1]{\mathsf{psh}(#1)}
\newcommand{\sh}[1]{\mathsf{sh}(#1)}
\newcommand{\TT}{\mathbb T}
\newcommand{\DD}{\mathbb D}
\newcommand{\TTp}{\mathbb T_*}
\newcommand\Set{\mathsf{Set}}
\newcommand\Cat{\mathsf {Cat}} 
\newcommand\fin{\mathsf{Set_f}}
\newcommand\sSet{\mathsf{sSet}}
\newcommand{\mm}[1][m]{\mathbf{#1}}
\newcommand{\nul}{\mathbf{0}}
\newcommand{\one}{\mathbf{1}}
\newcommand{\two}{\mathbf{2}}
\newcommand{\nn}{\mathbf{n}}

\newcommand{\CCat}{ \mathsf{C}}
\newcommand{\DCat}[1][D]{ \mathsf{#1}}
\newcommand{\ElP}[2]{\ensuremath{\mathsf{el}_{#2}(#1)}}
\newcommand{\ov }{/}
\newcommand{\defeq}{\overset{\textup{def}}{=}}

\newcommand{\fiso}{{\mathbf{B}}}
\newcommand{\fisinv}{{\fiso}^{\scriptstyle{\bm \S}}}
\newcommand{\fisinvdi}{{(\fiso \times \fiso^{\mathrm{op}})^{ \scriptstyle{\bm \downarrow}}}}
\newcommand{\Bifiso}{{\fiso^{\scriptstyle{\bm \downarrow}}}}

\newcommand{\Comm}{{K}}
\newcommand{\CComm}[2][(\CCC,\omega)]{{{\Comm}^{#1}_{#2}}}

\newcommand\CCC{\mathfrak{C}}
\newcommand\DDD[1][D]{\mathfrak{#1}}

\newcommand{\CGS}[1][(\CCC,\omega)]{{\GS^{#1}}}

\newcommand\Di{\mathfrak{Di}}
\newcommand{\Disig}{\sigma_{\Di}}
\newcommand{\Dipal}{(\Di, \sigma_{\Di})}%
\newcommand{\In}{{\mathrm{ in }}}
\newcommand{\Out}{{\mathrm{ out }}}

\newcommand{\Grbig}{\mathsf{Gr}\GrShape}
\newcommand{\Gret}{\mathsf{Gr_{et}}}

\newcommand\Gr{\mathsf{C}\Gret}
\newcommand{\Griso}{\mathdash\mathsf{CGr_{iso}}}

\newcommand{\elG}[1][\mathcal{G}]{\ensuremath{\mathsf{el}{(#1)}}}
\newcommand{\ElS}[1][S]{\ensuremath{\mathsf{el}{(#1)}}}
\newcommand{\ovP}[2]{\ensuremath{#2}\ov{#1}}
\newcommand{\Grets}[1][S]{\ensuremath{\Gret \ov {#1}}}
\newcommand{\Grs}[1][S]{\ensuremath{\Gr \ov {#1}}}
\tikzset{->-/.style={decoration={
			markings,
			mark=at position #1 with {\arrow{>}}},postaction={decorate}}}
\tikzset{-<-/.style={decoration={
			markings,
			mark=at position #1 with {\arrow{<}}},postaction={decorate}}}
	\newcommand\pto{\rightharpoonup}
	
	\newcommand\nuCSM{\mathsf {MO}^-}

		\newcommand{\comCinv}{\mathsf{Comp_{inv}}}
	
	\newcommand\CO{\mathsf {CO}}
	
	\newcommand\E{\mathcal{E}}
	
\newcommand{\MM}{\mathbb {M}}
\newcommand{\MMop}{\MM_{\mathsf{Op}}}
\newcommand{\Mop}{ M_{\mathsf{Op}}}
\newcommand{\Fop}{\text{free}}
\newcommand{\Uop}{\text{forget}}

\newcommand{\Fcat}{\text{free}}
\newcommand{\Ucat}{\text{forget}}

\newcommand{\boffcat}[1][\MM, \DCat]{ \Theta_{#1}}

\newcommand{\Op}{ \mathsf{Op}}

\newcommand{\diag}{\mathsf{\mathtt D}}
\newcommand{\GrShape}{\mathsf{psh_f}(\mathtt D)}

\newcommand{\Tr}[1][T]{\mathfrak{#1}}

\newcommand{\Grp}{\ensuremath{\mathsf{CGr}_*}}
\newcommand{\GSp}{\ensuremath{\mathsf{\GS}_*}}

\newcommand{\fisinvp}{\ensuremath{\fisinv_*}}

\newcommand\G{\mathcal{G}}
\renewcommand\H{\mathcal{H}}
\newcommand\W{\mathcal W}
\newcommand\I{\mathcal{I}}
\newcommand\C{\mathcal{C}}
\newcommand{\Lk}[1][k]{\ensuremath{\mathcal L^{#1}}}
\newcommand{\Wl}[1][m]{\ensuremath{\mathcal W^{#1}}}
\newcommand{\Wm}[1][m]{\ensuremath{\mathcal W^{#1}}}

\newcommand{\tauG}[1][\shortmid]{\ensuremath{\tau_{#1}}}

\newcommand\CX[1][X]{{\mathcal C_{#1}}}

\newcommand{\Fgraph}{ \xymatrix{
		E \ar@(lu,ld)[]_\tau&& H \ar[ll]_s \ar[rr]^t&& V}}
\newcommand{\Fgraphdash}{\xymatrix{
		E \ar@(lu,ld)[]_{\tau } &&H \ar[ll]_{s } \ar[rr]^{t }& &V }}

\newcommand{\Fgraphvar}[6]{\xymatrix{
		*[r] {#1}\ar@(ul,dl)[]_{#6} && {#2} \ar[ll]_-{#4} \ar[rr]^-{#5}&& {#3}}}

\newcommand\Cv[1][v]{{\mathcal C_{\mathbf{#1}}}}

\newcommand{\EI}{\ensuremath{{E_\bullet}}}

\newcommand{\vH}[1][v]{\sfrac{H}{#1}}
\newcommand{\vE}[1][v]{\sfrac{E}{#1}}
\newcommand{\nV}[1][n]{V_{#1}}
\newcommand{\nH}[1][n]{H_{#1}}
\newcommand{\nE}[1][n]{E_{#1}}

\newcommand\shorte[1][\tilde e]{{\shortmid_{#1}}}
\newcommand{\esG}[1][\mathcal{G}]{\ensuremath{\mathsf{es}{(#1)}}}

\newcommand{\esv}[1][v]{\ensuremath{\iota_{{#1}}}}
\newcommand{\ese}[1][\tilde e]{\ensuremath{\iota_{{#1}}}}
\newcommand{\esh}[1][h]{\ensuremath{\delta_{{#1}}}} 

\newcommand{\yet}{\Upsilon}
\newcommand{\yetp}{{\yet_*}}

\newcommand\X{\mathcal X}

\newcommand\Gg{\mathbf \Gamma}
\newcommand\Gdg{\mathbf \Lambda}
\newcommand\Gid[1][\G]{\mathbf{I}^{#1}}
\newcommand\coGg[1][\G]{{\Gg}({#1})}
\newcommand\GSg[1][S]{\mathbf \Gamma_{#1}}
\newcommand{\GrSG}[1][\G]{\ensuremath{(\ovP{S}{\Gr})^{(#1)}}}

\newcommand{\TJK}{T^{{\scriptscriptstyle{\mathrm{ds}}}}}
\newcommand{\muJK}{\mu^{\scriptscriptstyle
		{\mathrm{ds}}}}
\newcommand{\etaJK}{\eta^{\scriptscriptstyle{\mathrm{ds}}}}
\newcommand{\XGrJK}[1][X]{{{#1}\mathdash\mathsf{CGr}_{\mathsf{iso}}^{\scriptscriptstyle{\mathrm{ds}}}}}

\newcommand{\elpG}[1][\mathcal{G}]{\ensuremath{\mathsf{el}_*{(#1)}}}

\newcommand{\Gnov}[1][W]{\ensuremath{\mathcal G_{\setminus #1}}}
\newcommand{\Enov}[1][W]{\ensuremath{E_{\setminus #1}}}
\newcommand{\Hnov}[1][W]{\ensuremath{H_{\setminus #1}}}
\newcommand{\Vnov}[1][W]{\ensuremath{V_{\setminus #1}}}
\newcommand{\snov}[1][W]{\ensuremath{s_{\setminus #1}}}
\newcommand{\tnov}[1][W]{\ensuremath{t_{\setminus #1}}}
\newcommand{\taunov}[1][W]{\ensuremath{\tau_{\setminus #1}}}
\newcommand{\delW}[1][W]{\mathsf{del}_{\setminus #1}}

\newcommand\Grsimp{\mathsf{CGr}_{\mathsf{sim}}}
\newcommand\XGrsimp{\mathdash\mathsf{CGr}_{\mathsf{sim}}}

\newcommand{\mup}{\ensuremath{\mu^{\TTp}}}
\newcommand{\etap}{\ensuremath{\eta^{\TTp}}}
\newcommand{\Tp}{\ensuremath{T_*}}

\newcommand{\factcat}[1][\beta]{\ensuremath{\mathsf{fact}_*(#1)}}
\newcommand{\factcatup}[1][\beta]{\ensuremath{\mathsf{fact}(#1)}}
\newcommand{\prs}[1]{{\mathsf{psh}_\sSet}(#1)}

\title{Graphical combinatorics and a distributive law for modular operads}
\author{Sophie Raynor}

\thanks{The author acknowledges the support of Australian
	Research Council grants DP160101519 and FT160100393.}

\begin{document}

	\begin{abstract}
		This work presents a detailed analysis of the combinatorics of modular operads. These are operad-like structures that admit a contraction operation as well as an operadic multiplication. Their combinatorics are governed by graphs that admit cycles, and are known for their complexity. In 2011, Joyal and Kock introduced a powerful graphical formalism for modular operads. This paper extends that work. A monad for modular operads is constructed and a corresponding nerve theorem is proved, using Weber's abstract nerve theory, in the terms originally stated by Joyal and Kock. This is achieved using a distributive law that sheds new light on the combinatorics of modular operads.
		
%
		
	\end{abstract}

\maketitle
	
	

	\section*{Introduction}

Modular operads, introduced in \cite{GK98} to study moduli spaces of Riemann surfaces, are a
``{`higher genus' analogue of operads \dots in which graphs replace trees in the definition.}'' {\cite[Abstract]{GK98}.

	Roughly speaking, modular operads are $\N$-graded objects $P = \{P(n)\}_{n \in \N}$ that, alongside an operadic {multiplication} (or {composition}) $\circ\colon P(n) \times P(m) \to P({m+n -2})$ for $m,n \geq 1$, admit a {contraction} operation $\zeta \colon P(n) \to P(n-2)$, $n \geq 2$. For example, as in \cref{fig. gluing}, we may multiply two oriented surfaces by gluing them along chosen boundary components, or contract a single surface by gluing together two distinct boundary components.
	\begin{figure}[htb!]
		\includegraphics[width=.9\textwidth]{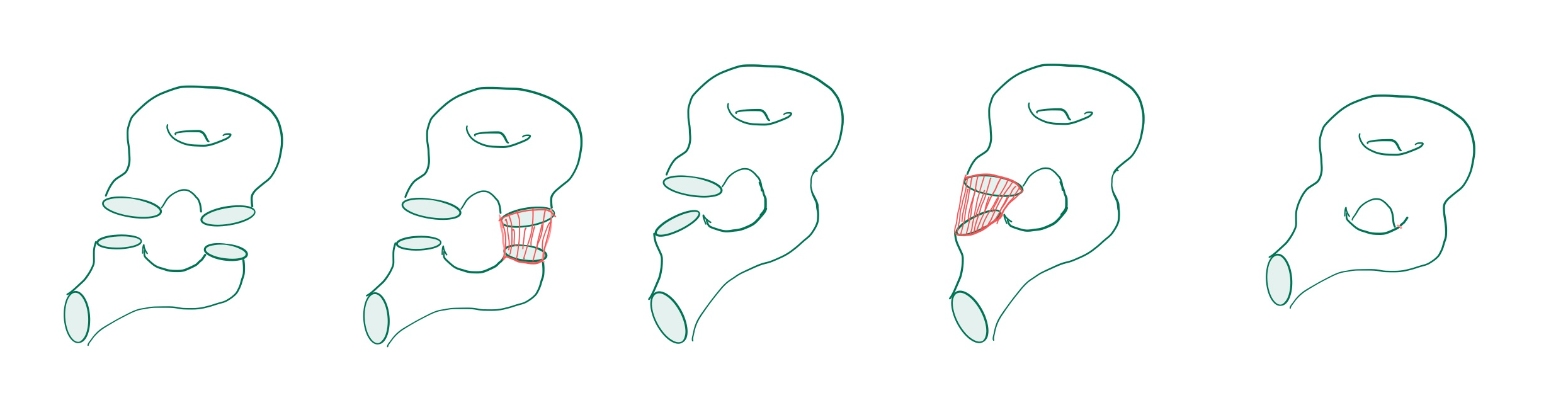}
		\caption{Gluing (multiplication) and self-gluing (contraction) of surfaces along boundary components. Moduli of geometric structures -- such as Riemann surfaces -- provide many examples of modular operads.}\label{fig. gluing}
	\end{figure}


This work considers a notion of modular operads due to Joyal and Kock \cite{JK11},
	\footnote{Joyal and Kock used the term {`compact symmetric multicategories (CSMs)'} in \cite{JK11} to refer to what are here called `modular operads'. Indeed, I adopted their terminology in \cite{Ray18} and in a previous version of this paper.} 
	that incorporates a broad compass of related structures, including modular operads in the original sense of \cite{GK98} (see \cref{ex: gk mod op}) and their coloured counterparts \cite{Gia13}, but also 
	wheeled properads \cite{HRY15, JY15} (see \cref{ex: wheeled prop}). 
	More generally, compact closed categories \cite{KL80} 
	provide examples of modular operads \cite{Ray21b} (see \cref{ex: compact closed}). These are closely related to circuit algebras that are used in the study of finite-type knot invariants \cite{BND17, DHR20} (see \cref{ex: circuits}). As such, modular operads have applications across a range of disciplines.
	
	However, the combinatorics of modular operads are complex. In modular operads equipped with a multiplicative unit, contracting this unit leads to an exceptional `loop', that can obstruct the proof of general results. This paper undertakes a detailed investigation into the graphical combinatorics of modular operads, and provides a new understanding of these loops.

	

	In \cite{JK11}, which forms the inspiration for this work, Joyal and Kock construct modular operads as algebras for an endofunctor on a category $\GS$ of coloured collections called `graphical species'. 
	Their machinery is significant in its simplicity. It relies only on minimal data and basic categorical constructions, that lend it considerable formal and expressive power.
	

	However, the presence of exceptional loops means that their modular operad endofunctor 
	does not extend to a monad 
	on $\GS$. As a consequence, it does not lead to a precise description of the relationship between modular operads and their graphical combinatorics. (See \cref{degenerate} for details.)
	
	This paper contains proofs of the following statements that first appeared in \cite{JK11} (and were proved -- by similar, though slightly less general methods than those presented here -- in my PhD thesis \cite{Ray18}):
	
	\begin{thmx}
		[Monad existence \cref{CSM monad DT}]
		\label{thm: main intro}
		The category $\CSM$ of modular operads is isomorphic to the Eilenberg-Moore category of algebras for a monad $\OOO$ on the category $\GS$ of graphical species.
	\end{thmx}
	In particular, $\OOO$ is the \textit{algebraically free monad} \cite{Kel80} on the endofunctor of \cite{JK11}. \cref{thm: nerve intro} -- the `nerve theorem' -- characterises modular operads in terms of presheaves on a category $\Klgr$ of graphs.
	
	\begin{thmx}[Nerve \cref{nerve theorem}]
		\label{thm: nerve intro}
		The category $\CSM$ has a full subcategory $\Klgr$ whose objects are graphs. 
		The induced (nerve) functor $N$ from $\CSM$ to the category $\pr{\Klgr}$ of presheaves on $\Klgr$ is fully faithful. 
		
		There is a canonical (restriction) functor $R^*\colon \pr{\Klgr} \to \GS$, and the essential image of $N$ consists of precisely those presheaves $P \colon \Klgr^{\mathrm{op}} \to \Set$ that satisfy the so-called `Segal condition': \newline $P$ is in the essential image of $N$ if and only if it is completely determined by the graphical species $R^*P$. 
	\end{thmx}


	An obvious motivation for establishing such results is provided by the study of weak, or $(\infty, 1)$-modular operads, by weakening the Segal condition of \cref{thm: nerve intro}. 
  To this end, 
	Hackney, Robertson and Yau have also recently proved versions of Theorems \ref{thm: main intro} and \ref{thm: nerve intro}, by different methods, and used them to obtain a model of $(\infty,1)$-modular operads that are characterised in terms of a weak Segal condition \cite{HRY19a, HRY19b}. A number of potential applications of such structures are discussed in the introduction to \cite{HRY19a}.

	The aim of this work is to prove Theorems \ref{thm: main intro} and \ref{thm: nerve intro} in the manner originally proposed by \cite{JK11} -- using the abstract nerve machinery introduced by Weber \cite{Web07,BMW12} (see \cref{sec: Weber}) -- and to use these proofs as a route to a full understanding of the underlying combinatorics, and the contraction of multiplicative units in particular. This method places strict requirements on the relationship between the modular operad monad $\OOO$ and the graphical category $\Klgr$. In fact, to apply the results of \cite{BMW12}, 
	the category $\Klgr$ must -- in a sense that will be made precise in \cref{sec: Weber} -- arise naturally from the definition of $\OOO$. 

	Neither the construction of the monad $\OOO$ for modular operads, nor the proof of \cref{thm: nerve intro} is entirely straightforward. First, the method of \cite{JK11}, which is closely related to analogous constructions for operads (Examples \ref{ex: operad endo}, \ref{ex: operad units}, c.f.~\cite{HRY15,Koc16, MMS09,MW07}) does not lead to a well-defined monad. Second, as a consequence of the contracted units, 
	the desired monad, once obtained, does not satisfy the conditions for applying the machinery of \cite{BMW12}. 
	To prove Theorems \ref{thm: main intro} and \ref{thm: nerve intro}, it is therefore necessary to break the problem into smaller pieces, 
	thereby rendering the graphical combinatorics of modular operads completely explicit.
	
	Since the obstruction to obtaining a monad in 
	\cite{JK11} arises from the combination of the modular operadic contraction operation and the multiplicative units (see \cref{degenerate}), the approach of this work is to first treat these structures separately -- via a monad $\TT $ 
	on $ \GS$ whose algebras are non-unital modular operads, and a monad $\DD $ 
	on $\GS$ that adjoins distinguished `unit' elements -- and then combine them, using the theory of distributive laws \cite{Bec69}.
	
	 	\cref{thm: main intro} is then a corollary of:
	\begin{thmx}[\cref{monads distribute} {\&} \cref{CSM monad DT}]\label{thm: composite intro}
		There is a distributive law $\lambda$ for $ \TT$ over $\DD$ such that the resulting composite monad $\DD\TT$ on $\GS$ is precisely the modular operad monad $\OOO$ of \cref{thm: main intro}.
	\end{thmx}
	The graphical category $\Klgr$, used to define the modular operadic nerve, arises canonically via the unique fully faithful--bijective on objects factorisation of a functor used in the construction of $\OOO$. Therefore, if the monad $\OOO$ satisfies certain formal conditions -- if it `has arities' (see \cite{BMW12}) -- then \cref{thm: nerve intro} follows from \cite[Section~1]{BMW12}. 
	
	Though the monad $\OOO$ on $\GS$ does not have arities, 
	the distributive law in \cref{thm: composite intro} implies that there is a monad $\TTp$, on the category $\GSp$ of $\DD$-algebras, whose algebras are modular operads. Moreover, \cref{thm: nerve intro} follows from: 
	
	\begin{lemx}[\cref{connected}]\label{prop: arities intro}
		The monad $\TTp$ on $\GSp$ has arities, and hence satisfies the conditions of 
		\cite[Theorem~1.10]{BMW12}.
	\end{lemx}
	
	I conclude this introduction by briefly mentioning three (related) benefits of this abstract approach.
	
	In the first place, the results obtained by this method provide a clear overview of how modular operads fit into the wider framework of operadic structures, and how other general results may be modified to this setting. For example, by \cref{prop: arities intro}, $\TTp$ and $\Klgr$ satisfy the Assumptions 7.9 of \cite{CH15}, which leads to a suitable notion of weak modular operad via the following corollary:
	
	\begin{corx}[\cref{cor: weak}]\label{cor: model}
		There is a model structure on the category of presheaves in simplicial sets on $\Klgr$. The fibrant objects are precisely those presheaves that satisfy a weak Segal condition.
	\end{corx}
	
	%
	
	Second, since this work makes the combinatorics of modular operads -- including the tricky bits -- completely explicit, it provides a clear road map for working with and extending the theory. 
	
	One fruitful direction for extending this work is to use iterated distributive laws \cite{Che11} to generalise constructions presented here. In \cite{Ray21a}, an iterated distributive law is used to construct {circuit operads} -- {modular operads with an extra product operation, closely related to small compact closed categories -- as algebras for a composite monad on $\GS$ (\cref{ex: circuits}). Once again, the distributive laws play an important role in describing the corresponding nerve. The approach of \cite[Section~3]{Che11} 
		may also be used to construct higher (or $(n, k)$-) modular operads. 
		This can be used to give a modular operadic description of {extended cobordism categories}. 
	
	Finally, the complexities of the combinatorics of contractions can provide new insights into the structures they are intended to model. In current work, also together with L.\ Bonatto, S.\ Chettih, A.\ Linton, M.\ Robertson, N.\ Wahl, I am using these ideas to explore singular curves in the compactification of moduli spaces of algebraic curves. (See also \cref{ex: gk mod op}, and c.f.~\cite{BCL22} for the genus 0 case.)
	
	\medspace

	{This work owes its existence to the ideas of A. Joyal and J. Kock and I thank Joachim for taking time to speak with me about it. P. Hackney, M. Robertson and D. Yau's work has been an invaluable resource. Conversations with Marcy have been particularly helpful. I gratefully acknowledge the anonymous reviewer whose insights have not only improved the paper, but also increased my appreciation of the mathematics.
		
		This article builds on my PhD research at the University of Aberdeen, UK and funded by the EPFL, Switzerland, and I thank my supervisors R. Levi and K. Hess. 
		Thanks to the members of the Centre for Australian Category Theory at Macquarie University for providing the ideal mathematical home for these results to mature, and to R. Garner and J. Power in particular, for their reading of this work. 
		
		\begin{rmk}\label{rmk correction}
			The following errors appear in the published version \cite{Ray20} and are corrected here:
			
			In \cite[Section~4.1]{Ray20}, graph embeddings (\cref{def embedding}) were mistakenly identified with graph monomorphisms (\cite[Proposition~4.8]{Ray20}) and the terminology of `monomorphisms' was used throughout the paper. The incorrect \cite[Proposition~4.8]{Ray20} -- which served only to establish terminology in \cite{Ray20} -- has been deleted and \cite[Lemma~4.7]{Ray20} has been replaced with \cref{def embedding}, of graph `embeddings', and \cref{lem: mono}. The examples in \cref{ssec embeddings} have also been modified accordingly. The terminology of graph embeddings is due to \cite[Section~1.3]{HRY19a}, and replaces the incorrect use of the term (graphical) `monomorphism' in \cite{Ray20}. 
			
			On \cite[page~61]{Ray20}, there is a sentence that begins ``\textit{But then $\elG[\bigcirc]\cong \elG[\shortmid]$, and hence $S(\bigcirc)\cong S(\shortmid)$ \dots }". This should simply read ``\textit{But this would imply that $S(\bigcirc)\cong S(\shortmid)$ \dots }", and the rest of \cref{degenerate} is unchanged.

		\end{rmk}
		\subsection*{Overview and key points}
		The opening two sections provide context and background for the rest of the work. An axiomatic definition of modular operads is given in \cref{sec: definitions}. \cref{sec: Weber} gives a brief review of Weber's abstract nerve theory, that provides a framework for the later sections. Both these introductory sections include a number of examples to motivate the constructions that follow. 
		
\cref{s. graphs} is a detailed introduction to the (Feynman) graphs of \cite{JK11}, and 
		\cref{sec: topology} focuses on their \'etale morphisms. The monad $ \TT$ for {non-unital modular operads} is constructed in \cref{a free monad}. 
		
		
		\cref{degenerate} acts as a short intermezzo in which the appearance of exceptional loops in the theory, and why they are problematic in the construction of \cite{JK11}, is explained.

		
		The construction of the monad $\OOO$ for modular operads happens in \cref{s. Unital}. This is the longest and most important section of the work, and contains most of the new contributions. Finally, \cref{s. nerve} contains the proof of the Nerve \cref{thm: nerve intro}, as well as a 
		short discussion on weak modular operads. 

			There have been many other approaches to the issue of loops, some of which are mentioned in Remarks \ref{rmk: standard solution} and \ref{rmk: HRY solution}. But the graphical construction presented in this paper 
			is unique, as far as I am aware, in that it \textit{does not} incorporate some version of the exceptional loop into the graphical calculus, in order to model contractions of units. (See \cref{rmk: construction unique}.)

		In other approaches, the contraction of units is described by 
adjoining a formal colimit of a diagram of graphs, 
 resulting in the exceptional loop object (see \cref{deg loop}). By contrast, we will see in \cref{s. Unital} that the definition of modular operads (\cref{defn: Modular operad}) implies that the contracted units are, in fact, described in terms of 
a formal limit of the very same diagram. This is illustrated in \cref{fig. lim and colim}.

Moreover, this construction leads to a graphical description of the unit contraction, not by an exceptional loop, but as the singularity of a `double cone' of wheel-shaped graphs (see \cref{subs. similar} and \cref{fig: contracting units}).
\begin{figure}
	[htb!]\begin{tikzpicture}[scale = .45]
	\draw [gray] (-5,-2.2)--(5,-2.2)--(5,2.3)--(-5,2.3)--(-5,-2.2);
		\draw [red,decorate,decoration={brace,amplitude=5pt, raise=4pt},yshift=0pt]
	(-5,2.5) -- (14,2.5) node [black,midway,yshift=0.5cm] {\footnotesize
	{Formal colimit: glue endpoints of edge to form a loop object.}}
	;
		\draw [blue,decorate,decoration={brace,amplitude=5pt,mirror,raise=4pt},yshift=0pt]
	(-14,-2.4) -- (5,-2.4)node [black,midway,yshift=-0.5cm] {\footnotesize
			{Formal limit: pick out midpoint of edge.}}
		
	;
	
	\draw[->] (-10,0)--(-6,0);
	\node at (-8,.3){\small{$z$}};
		\draw[->] (6,0)--(10,0);
	
	\node at (-12,0){
		\begin{tikzpicture}[scale = .5]
		\filldraw (0,0) circle (3pt);
	\end{tikzpicture}};
	\node at (12.5,0){
	\begin{tikzpicture}[scale = .6]
	\draw [ultra thick] (0,0) circle (.8cm);
	\end{tikzpicture}};
	\node at (0,0){
		\begin{tikzpicture}[scale = .6]
		\draw[ ultra thick]
		(0,0) -- (0,2)
		(6,0) -- (6,2);
		\draw[->]
		(1,.5)--(5,.5);
			\draw[->]
		(1,1.5)--(5,1.5);
		\node at (3, 2.2){ \begin{tikzpicture}
			[scale = .3] 	\draw[thick] 
			(0,0) -- (0,2);
				\draw [thin,gray, <->] (-.1,1.9) arc [start angle=110, 
			end angle=250,
			y radius= .9 cm, 
			x radius= 1 cm];
					\node at (-3.5 ,.7){ \small { flip edge } };
					\end{tikzpicture}};
					\node at (3, -.2){ \begin{tikzpicture}
					\node at (-3.5 ,.7){ \small { id } };
					\end{tikzpicture}};
		\end{tikzpicture}
	};
	\end{tikzpicture}
\caption{An edge graph with no vertices may be flipped or left unchanged. The exceptional loop that `glues the edge ends together' arises as the formal colimit of these endomorphisms. In \cref{s. Unital}, the graph category of Sections \ref{s. graphs}-\ref{sec: non-unital} (and \cite{JK11}) is enlarged to include the morphism $z$ that `picks out the midpoint' of the edge graph with no vertices. } \label{fig. lim and colim} 
\end{figure}

\section{Definitions and examples}

\label{sec: definitions}

The goal of this section is to give an axiomatic definition of modular operads (\cref{defn: Modular operad}), and to provide some motivating examples. As mentioned in the introduction, the term `modular operad' refers here to what are called `compact symmetric multicategories (CSMs)' in \cite{JK11}.

\subsection{Graphical species}
\label{ssec: graphical species}
After establishing some basic notional conventions, we discuss 
Joyal and Kock's 
graphical species \cite{JK11} that generalise various notions of coloured collection used in the study of operads.

Let $\Set$ be the category of sets and all morphisms between them. A \emph{presheaf} on a category $\CCat$ is a functor $ P\colon \CCat^\mathrm{op} \to \Set$. The corresponding functor category 
is denoted $\pr{\CCat}$. 

\begin{defn}
	\label{defn: general element} 
	Objects of the category $\ElP{P}{\CCat}$ of \textit{elements of a presheaf $P \colon \CCat^{\mathrm{op}} \to \Set$} are pairs $(c, x)$ -- called \emph{elements of $P$} -- where $c$ is an object of $\CCat$ and $x \in P(c)$. Morphisms $(c,x) \to (d,y)$ in $\ElP{P}{\CCat}$ are given by morphisms $f \in \CCat (c,d)$ such that $P(f)(y) = x$. 

\end{defn}
%
%
%
%
%
If a presheaf $P $, on an essentially small category $\CCat$, is of the form $ \CCat(-, c)$, for some $c \in \CCat$, then $\ElP{P}{\CCat}$ is the \textit{slice category} $\CCat\ov c$ whose objects are pairs $(d,f)$ where $f \in \CCat (d,c)$, and morphisms $(d,f) \to (d', f')$ are given by 
 by commuting triangles in $\CCat$:
\[ \xymatrix{ d \ar[rr]^-g \ar[dr]_{f} && d' \ar[dl]^-{f'}\\&c.&}\] 

Given a functor $\iota \colon \DCat \to \CCat$, let $\iota^* \CCat(-, c)$, $d \mapsto \CCat(\iota (d),c)$ be the induced pullback on presheaves. For all $c \in \CCat$, the \textit{slice category of $\DCat$ over $c$} is defined by $\DCat \ov c\defeq\ElP{\iota^*\CCat(-, c) }{\DCat}$. (This involves a small abuse of notation, and $ \DCat \ov c$ is more accurately denoted by $\iota \ov c$.)

In particular, the Yoneda embedding $\CCat \to \pr{\CCat}$ induces a canonical isomorphism $ \ElP{P}{\CCat} \cong \CCat \ov P$ for all presheaves $P$ on $\CCat$, and these categories will be identified in this work. 


\medspace

 The groupoid of finite sets and bijections is denoted by $\fiso$. For $n \in \N$, the set $ \{1,\dots, n\}$ is denoted by $\nn$. So $\nul = \emptyset$ is the empty set.


\begin{rmk}\label{rmk skelet}
	
Let $\Sigma\subset \fiso$ denote the skeletal subgroupoid on the objects $\nn$, for $n \in \N$. A presheaf $P\colon \fiso^\mathrm{op} \to \Set$ on $\fiso$, also called a \emph{(monochrome} or \emph{single-sorted) species} \cite{Joy81}, determines a presheaf on $\Sigma$ by restriction. 
	 Conversely, a $\Sigma$-presheaf $Q$ may always be extended to a $\fiso$-presheaf $Q_{\fiso}$, by setting
	\[ Q_{\fiso}(X) \defeq \mathrm{lim}_{(\nn,f) \in \Sigma \ov X} Q(\nn) \ \text{ for all } n \in \N.\]
 \end{rmk}

%
%


Graphical species, defined in \cite[Section~4]{JK11}, are a \textit{coloured} or \textit{multi-sorted} version of species. 

Let the category $\fisinv$ be obtained from $\fiso $ by adjoining a distinguished object $\S$ that satisfies 
\begin{itemize}
	\item $\fisinv (\S, \S) = \{ id, \tau \} $ with $\tau^2 = id$,
	\item for each finite set $X$ and each element $x \in X$, there is a morphism $ch_x \in \fisinv (\S, X)$ that `chooses' $x$, and $\fisinv(\S, X)= \{ch_x\ ,\ ch_x \circ \tau \}_{x \in X}$, 
	\item $\fisinv(X,Y) = \fiso (X,Y)$ for all finite sets $X$ and $Y$, and morphisms are equivariant with respect to the action of $\fiso$. That is, $ ch_{f(x)} = f \circ ch_x \in \fisinv (\S, Y)$ for all $ x \in X$ and all bijections $f\colon X \xrightarrow {\cong} Y$.
\end{itemize}

\begin{defn}\label{defn: graphical species}
 A \emph{graphical species} is a presheaf $S\colon {\fisinv}^\mathrm{op} \to \Set$. 
 
 The element category of a graphical species $S$ is denoted by $\elG[S] \defeq \ElP{S}{\fisinv}$, and the category of graphical species by $\GS\defeq\pr{\fisinv}$. 

\end{defn}


Hence, a graphical species $S$ is described by a species $(S_X)_{X \in \fiso}$, and a set $S_\S$ with involution $S_\tau \colon S_\S \to S_\S$, 
together with, for each finite set $X$, and $x \in X$ a $\fiso$-equivariant {projection} $S(ch_x)\colon S_X \to S_\S$.



\begin{defn}\label{c arity}
	Given a graphical species $S$, the pair $(S_\S, S_\tau)$ is called the \emph{(involutive) palette} of $S$ and elements $c \in S_\S$ are \textit{colours of $S$}. If $S_\S$ is trivial then $S$ is a \emph{monochrome graphical species}.
%

	For each element $ \underline c = (c_x)_{x \in X} \in {S_\S}^{X}$, the \emph{$\underline c$-(coloured) arity $ {S_{\underline c}}$} is the fibre above $ \underline c \in {S_\S}^{X}$ of the map $(S(ch_x))_{x \in X}\colon {S_X} \to{S_\S}^{X}$.

	
\end{defn}




\begin{rmk} \label{rmk: involution} 
The involution $\tau$ on $\S$ is responsible for much of the heavy lifting in the constructions that follow. Initially however, its role may seem obscure. I mention two key features here. First, the involution provides the expressive power necessary to describe composition rules involving colours, such as particle spin, that may have an orientation, or direction. 
(Directed graphical species are discussed in \cref{ex:direction}.)

The second is more fundamental. 
As will be explained in \cref{ex: embedding fin}, $\fisinv$ embeds in a certain category of graphs. Under this embedding, the distinguished object $\S$ is represented as the exceptional edge with no vertices, and the involution $\tau$ as the `flip' map that swaps its ends (see \cref{fig. lim and colim}). 
This enables us to encode formal compositions in graphical species -- described in terms of graphs -- as categorical limits, and thereby derive the results of this paper 
by purely abstract methods. For example, the involution underlies a well-defined notion of graph nesting, or substitution, in terms of diagram colimits, without the need to specify extra data (see Sections \ref{sec: non-unital} and \ref{degenerate}, and compare with, e.g.\ \cite{JY15,HRY15}).

\end{rmk}

\begin{ex}\label{Comm}
	The terminal graphical species $\Comm$ has trivial palette and $ \Comm_X = \{*\}$ for all finite sets $X$.
\end{ex}

\begin{defn}\label{def: palette preserving]}

	A morphism $ \gamma \in \GS(S, S')$ is \emph{palette-preserving} if its component $\gamma_\S$ at $\S$ is the identity on ${S_\S}$. For a fixed palette $(\CCC, \omega)$, $\CGS$ is the subcategory of $\GS$ on the $(\CCC, \omega)$-coloured graphical species and palette-preserving morphisms.
	
	\end{defn}


\begin{ex}\label{Comm c}
	For any palette $(\CCC, \omega)$, the terminal $(\CCC, \omega)$-coloured graphical species $\CComm{}$ in $\CGS$ is described by $\CComm{X}= \CCC^X$ with $ \CComm{\underline c} = \{*\}$ for all finite sets $X$ and all $ \underline c \in \CCC^X.$

	In particular, let $\Disig$ be the unique non-identity involution on the set $\Di \defeq \{\In, \Out\}$. A \textit{monochrome directed graphical species} is a graphical species with palette $\Dipal$. The terminal monochrome directed graphical species is denoted by $Di \defeq\CComm[\Dipal]{}$. See also \cref{ex:direction}. 
\end{ex}

\begin{rmk} \label{rmk: visualise} 
	In the graphical representation of the category $\fisinv$, mentioned in \cref{rmk: involution}, a finite set $X$ is represented by a \textit{corolla} or \textit{star graph} $\CX$ with legs in bijection with $X$ (\cref{fig. species vis} left side). 
	
	An element $ \phi \in S_X$ of a graphical species $S$ is represented as a labelling or \textit{decoration} of the unique vertex of $\CX$, and a \textit{colouring} of the legs of $\CX$ by $S_\S$ according to $S(ch_x)$ for $x \in X$ (\cref{fig. species vis} right side). %
%
%
	
	\begin{figure}[htb!]
	\[	\begin{tikzpicture}[scale = 0.55]
	\node at(0,0){\begin{tikzpicture}[scale = 0.55]
		\draw [thick] (-8.5,3.5) --(-6.5,3.5);
	\node at (-8.3, 3.1) {\tiny{ $1$}};
	\node at (-6.7, 3.1) {\tiny{ $2$}};

	\node at (-7.5,1.7){$\S$};
	\node at (0,1.7){$X = \{x,y,z\}$};
	
	\draw[->] (-5.5,3.5)--(-3.5,3.5);
	\node at (-4.5,3.1){\scriptsize{$1 \mapsto x$}};
		\node at (-4.5,3.9){\scriptsize{$ch_x$}};
	\draw [ ultra thick] (0,3.5) -- (-2.25,3.5);
	\draw (0,3.5)--(1.8,4.5)
	(0,3.5) --(1.8,2.5);
\filldraw (0,3.5) circle (3pt);	
	\node at (-1, 3.2) {};
	\node at (0.3, 4.5) {};
	\node at (-2.2, 3) {\scriptsize {$x$}};
	\node at (1.7, 4.05) {\tiny {$y$}};
	\node at (1.7, 2.75) {\tiny{ $z$}};
	\node at (0.3,2.5) {}; 
	
		\draw [gray, decorate,decoration={brace,amplitude=5pt, raise=4pt},yshift=0pt]
	(-9,2) -- (-9,5) 
	;
	\draw [gray,decorate,decoration={brace,amplitude=5pt,mirror,raise=4pt},yshift=0pt]
(2.5,2) -- (2.5,5)
	;
	
	\end{tikzpicture}};
	
	\draw[gray, thick, ->] (7.5,0)--(9.5,0);
	\node at (8.5,.5) {$S$};
	\node at(17, 0){\begin{tikzpicture}[scale = 0.55]

	\draw [cyan, thick] (-8.5,0) --(-6.5,0);
	\node at (-8.3, 0.3) {\tiny{ $c_x$}};
	\node at (-6.7, 0.3) {\tiny{ $\omega c_x$}};

	\draw[<-] (-5.5,0)--(-3.5,0);
	\node at (-4.5,.4){\scriptsize{$S(ch_{x})$}};
	
	\draw [ cyan, ultra thick] (0,0) -- (-2.25,0);
	\draw (0,0)--(1.8,1)
	(0,0) --(1.8,-1);
	
	\node at (-1, -0.3) {};
	\node at (0.3, 1) {};
	\node at (-2.2, -0.5) {\tiny {$(x)$}};
	\node at (-2.2, 0.3) {\tiny{ $c_{x}$}};
		\node at (1.7, .55) {\tiny {$(y)$}};
	\node at (1.7, 1.2) {\tiny{ $c_{y}$}};
		\node at (1.7, -.6) {\tiny {$c_{z}$}};
	\node at (1.7, -1.25) {\tiny{ $(z)$}};
	\node at (0.3,-1) {};
	
	\draw [ draw=red, fill=white]
	(0,0) circle (15pt);
	
	\node at(0,0) {\small{$\phi $}};
		\draw [gray, decorate,decoration={brace,amplitude=5pt, raise=4pt},yshift=0pt]
	(-9,-1.5) -- (-9,1.5) 
	;
	\draw [gray, decorate,decoration={brace,amplitude=5pt,mirror,raise=4pt},yshift=0pt]
	(2.5,-1.5) -- (2.5,1.5)
	;
		\node at (-7.5,-1.8){\small{${\color{cyan}{c_x}} \in S_\S$}};
	\node at (-.6,-1.8){\small{$\phi \in S_{X}$}};
	\end{tikzpicture}
		};
		\end{tikzpicture}
	\]
	\caption{
	Graphical species may be represented graphically: 
		$\phi \in S_{X}$ is represented as a $X$-corolla $\CX$ with vertex decorated by $\phi$ and $x$-leg coloured by ${c_x = S(ch_{x})}$.}\label{fig. species vis}
		\end{figure}
\end{rmk}

\begin{ex}\label{ex:direction}

The graphical species $Di$ was defined in \cref{Comm c}. 
For each finite set $X$, $Di_X =\{\In, \Out\}^X$ is the set of partitions $X = X_{\In}\amalg X_{\Out}$ of $X$ into \textit{input} and \textit{output} sets, with blockwise action of the partition-preserving isomorphisms in $\ElS[Di]$. 

In other words, $\ElS[Di]$ is equivalent to the category $\fisinvdi$, obtained from $\fiso\times \fiso^{\mathrm{op}} $ by adjoining a distinguished object $(\downarrow)$ (see \cref{fig: directed elements}(a)) with
trivial endomorphism group, 
and -- for all pairs $(X, Y)$ of finite sets -- \textit{input} morphisms $i_x\colon(\downarrow) \to (X,Y)$ for all $x \in X$, and \textit{output} morphisms $o_y\colon(\downarrow) \to (X,Y)$ for all $y \in Y$, that are compatible with the action of $\fiso\times \fiso^{\mathrm{op}} $ (see \cref{fig: directed elements}(d)).

%

The objects $(X,Y)$ of $\fisinvdi$ may be represented, 
as in \cref{fig: directed elements}(b), as \textit{directed corollas} 
and the distinguished object $(\downarrow) $ as a 
\textit{directed exceptional edge} 
 (\cref{fig: directed elements}(a)). 
If $Y = \{*\}$ is a singleton, then $(X, \{*\})$ describes a 
 \emph{rooted corolla} 
 as in \cref{fig: directed elements}(c). 
 
Hence $\GS \ov \Di$ is equivalent to the category 
 $\pr{\fisinvdi}$ of \textit{directed graphical species}. The subcategory $\CGS[\Dipal] \ov Di $ of \textit{monochrome directed graphical species} is equivalent to $ \pr{\fiso \times \fiso^{\mathrm{op}}}$. 

%

\begin{figure}[htb!]
	\[	\begin{tikzpicture}
	
	\node at (1.4,-2.3) {(a)};
	\node at (1,0){\begin{tikzpicture}[scale = .6]
		
		\draw[thick]
		(0,1)--(0,3);

		\node at (-.5, 1.3) {\tiny{$\Out$}};
		\node at (-.5, 2.7) {\tiny{$\In$}};

		\end{tikzpicture}};

	\node at (1.8,0){\begin{tikzpicture}[scale = .6]
		
		\draw[ -<-=.5]
		(0,1)--(0,3);
		
		\end{tikzpicture}};
	
			\node at (5.3, -2.3) {(b)};
	\node at (4,0){\begin{tikzpicture}[scale = .55]
		\draw[thick]
		(-1.5,0)--(0,2)
		(-1,0)--(0,2)
		(1,0)--(0,2)
		(1.5,0)--(0,2)
		(0,2)--(-1,4)
		(0,2)--(-.5,4)
		(0,2)--(1,4);
		
		\draw[dashed]
		(-. 8, 0)--(0.8,0)
		(-.3,4)--(.8,4);
		
		\draw[decoration={brace, raise=5pt},decorate]
		(-1,4) -- node[above=6pt] {\small{$X$}} (1,4);
		\draw[decoration={brace,mirror, raise=5pt},decorate]
		(-1.5,0) -- node[below=6pt] {\small{$Y$}} (1.5,0);
		
		\node at (-1.9, .2) {\tiny{$\Out$}};
		\node at (1.9, .2) {\tiny{$\Out$}};
		\node at (.8, 1.6) {\tiny{$(\In)$}};
		\node at (-1.4,3.8) {\tiny{$\In$}};
		\node at (-.8,2.4) {\tiny{$(\Out)$}};
		\node at (1.4, 3.8) {\tiny{$\In$}};

		
		\draw [ draw=black, fill=white]
		(0,2)circle (8pt);
		
		
		\end{tikzpicture}};
	
	\node at (6.6,0){\begin{tikzpicture}[scale = .55]
		\draw[-<-=.3]
		(-1.5,0)--(0,2);
		\draw[-<-=.3]	(-1,0)--(0,2);
		\draw[-<-=.3]	(1,0)--(0,2);
		\draw[-<-=.3]	(1.5,0)--(0,2);
		\draw[-<-=.7](0,2)--(-1,4);
		\draw[-<-=.7](0,2)--(-.5,4);
		\draw[-<-=.7]	(0,2)--(1,4);
		
		\draw[dashed]
		(-. 8, 0)--(0.8,0)
		(-.3,4)--(.8,4);
		
		\draw[decoration={brace, raise=5pt},decorate]
		(-1,4) -- node[above=6pt] {\small{$X$}} (1,4);
		\draw[decoration={brace,mirror, raise=5pt},decorate]
		(-1.5,0) -- node[below=6pt] {\small{$Y$}} (1.5,0);

		\draw [ draw=black, fill=white]
		(0,2)circle (8pt);
		
		\end{tikzpicture}};
		\node at (10,0.1){	\begin{tikzpicture}
		
		\node at (2,0){\begin{tikzpicture}[scale = .6]
			\draw[ultra thick]
			(0,0) --(0,2);
			\draw
			(0,2)--(-1,4)
			(0,2)--(-.5,4)
			(0,2)--(1,4);
			
			\draw[dashed]
			(-.3,4)--(.8,4);
			
			\draw[decoration={brace, raise=5pt},decorate]
			(-1,4) -- node[above=6pt] {\small{$X$}} (1,4);
			
			\draw [ draw=black, fill=white]
			(0,2)circle (8pt);
			\node at (0, -.3){$*$};
			
			\end{tikzpicture}};
		\end{tikzpicture}
	};
		\node at (10,-2.3) {(c)};
	
	\node at (14.5,0){\begin{tikzpicture}[scale = .6]
		\draw[red,thick, -<-=.5]
		(-2.5, 2)--(-2.5,4);
		\draw [red, ->]
		(-1.8,3)--(-1,3);
		\node at (-1.4,3.3) {\tiny{$i_x$}};
		\draw[cyan,thick, ->-=.5]
		(3, 2)--(3,0);
		\draw [cyan, ->]
		(2.5,1)--(1.7,1);
		\node at (2.2,.7) {\tiny{$o_y$}};
		\draw[-<-=.3]
		(-1.5,0)--(0,2);
		\draw[ -<-=.3]	(-1,0)--(0,2);
		\draw[cyan, thick,-<-=.3]	(1,0)--(0,2);
		\node at (.55, .3){\scriptsize{$y$}};
		\draw[ -<-=.3]	(1.5,0)--(0,2);

		\draw[-<-=.7](0,2)--(-1,4);
		\draw[red, thick,-<-=.7](0,2)--(-.5,4);
		
		\node at (-.2,3.6){\scriptsize{$x$}};
		\draw[-<-=.7]	(0,2)--(1,4);
		
		\draw[dashed]
		(-. 8, 0)--(0.8,0)
		(-.3,4)--(.8,4);
		
		\draw[decoration={brace, raise=5pt},decorate]
		(-1,4) -- node[above=6pt] {\small{$X$}} (1,4);
		\draw[decoration={brace,mirror, raise=5pt},decorate]
		(-1.5,0) -- node[below=6pt] {\small{$Y$}} (1.5,0);

		\draw [ draw=black, fill=white]
		(0,2)circle (8pt);
		
		\end{tikzpicture}};

		\node at (14.5,-2.3) {(d)};

	\end{tikzpicture}\]
	\caption{(a) The directed exceptional edge $(\downarrow)$; (b) the pair $(X,Y) \in \fisinvdi$ describes a directed corolla; (c) $(X, \{*\})$ describes a rooted corolla $t_X$; (d) input and output morphisms in $\fisinvdi$.}
	\label{fig: directed elements}
\end{figure}

A \textit{PROP} \cite{MacL65} is a strict symmetric monoidal category $(\DCat[E], + , 0)$ whose objects are natural numbers and whose monoidal product $+$ is addition on objects. 
More generally, for any set $\DDD$, a $\DDD$-\textit{coloured PROP} $(\DCat[E]^{\DDD}, \oplus , \emptyset)$ is a strict symmetric monoidal category whose monoid of objects is freely generated by $\DDD$. 
By \cref{rmk skelet}, this is equivalently a presheaf $P^{\DDD}$ on $\fisinvdi$ with $P^{\DDD}(\downarrow) = \DDD$ and, 
\[P^{\DDD}(X;Y) = \mathrm{lim}_{{(\mm,f)\in \Sigma \ov X}\atop {(\nn,g)\in \Sigma \ov Y}} \left(\coprod_{(\ccc, \ddd) \in \DDD^{\mm} \times \DDD^{\nn} } \DCat[E]^{\DDD}(\ccc, \ddd)\right) \text{ for all pairs $(X,Y)$ of finite sets,} \]
together with composition and monoidal product maps, and an injection $P^{\DDD}(\downarrow) \hookrightarrow P^{\DDD}(\one;\one)$ that induces the identities for composition. In particular, PROPs may be described in terms of graphical species. 

\end{ex}

\subsection{Multiplication and contraction on graphical species}

Intuitively, a multiplication $\diamond$ on a graphical species $S$ is a rule for combining (gluing) distinct elements of $S$ along pairs of legs (called `ports') with dual colouring as in \cref{multiplication}: 
 \begin{figure}[htb!]
	
	\begin{tikzpicture}
	\node at (0,0){\begin{tikzpicture}[scale = 0.5]
		\foreach \angle in {-0,90,180,270} 
		{
			\draw(\angle:0cm) -- (\angle:1.5cm);

		}
		\draw [ thick] (0,0) -- (1.5,0);
		
		\node at (-1, -0.3) {};
		\node at (0.3, 1) {};
		\node at (1, -0.5) {\tiny {$(x)$}};
		\node at (1, 0.3) {\tiny c};
		\node at (0.3,-1) {};
		
		\draw [ draw=red, fill=white]
		(0,0) circle (15pt);
		
		\node at(0,0) {\small{$\phi $}};
		\end{tikzpicture}
	};

	\node at (2,0){\begin{tikzpicture}[scale = 0.5]
		\foreach \angle in {60,180, 300} 
		{
			\draw(\angle:0cm) -- (\angle:1.5cm);

		}
		\draw [ thick] (0,0) -- (-1.5,0);
		\node at (-1, 0.3) {\tiny {$\omega c$}};
		\node at (-1, -0.5) { \tiny{$(y)$} };
		\node at (1, 0.8) {};
		\node at (1, -0.8) {};
		\draw [ draw=red, fill=white]
		(0,0) circle (15pt);
		
		\node at(0,0) {\small{$\psi $}};

		\end{tikzpicture}};
	
	
	\node at (4,0){\large $\longrightarrow$};
	
	\node at (6,0) 
	{\begin{tikzpicture}[scale = 0.5]
		\foreach \angle in {-0,90,180,270} 
		{
			\draw(\angle:0cm) -- (\angle:1.5cm);
			\draw[thick]
			(0,0)--(2.1,0);
			\draw [ draw=red, fill=white]
			(0,0) circle (15pt);
			
			\node at(0,0) {\small{$\phi $}};

		}
		
		\end{tikzpicture}};
	
	\node at (7,0){\begin{tikzpicture}[scale = 0.5]
		\foreach \angle in {60,180, 300} 
		{
			\draw(\angle:0cm) -- (\angle:1.5cm);

		}
		\draw [ thick] (0,0) -- (-1.5,0);
		\draw [ draw=red, fill=white]
		(0,0) circle (15pt);
		
		\node at(0,0) {\small{$ \psi $}};
		
		\end{tikzpicture}};
	
	\draw [ draw=blue]
	(6.5,0) ellipse (30pt and 14 pt);
	
	\node at (6.5,-1.5){ $\phi \diamond_{x,y} \psi$};
	
\end{tikzpicture}

\caption{Multiplication}
\label{multiplication} 	
\end{figure}

The notation `$\pto$' denotes a partial map of sets. So $f \colon A \pto B$ is given by a subset $A ' \subset A$ and a function $A' \to B$. 
 
\begin{defn}\label{defn: multiplication}\label{coloured mult cont} 
	A \emph{multiplication} $\diamond$ on a graphical species $S$ is given by a family of partial maps 
	\begin{equation} \label{eq: mult def} - \diamond^{X,Y}_{x,y} - \colon S_{X \amalg \{x\}} \times S_{Y \amalg \{y\}} \pto S_{X \amalg Y},\end{equation} 
	defined (for all $X,Y$ and $x,y$) whenever 
	$\phi \in S_{X \amalg \{x\}}, \psi \in S_{Y \amalg \{y\}}$ satisfy $S(ch_x)(\phi) = S(ch_y \circ \tau)(\psi)$. 

The multiplication $\diamond$ satisfies the following conditions:
 \begin{enumerate}[(m1)]
 	\item \emph{(Commutativity axiom.)}\\
 	Wherever $\diamond^{X,Y}_{x,y}$ is defined, 
 	\[\psi \diamond^{Y,X}_{y,x} \phi = \phi \diamond^{X,Y}_{x,y} \psi \] 
 	\item \emph{(Equivariance axiom.)}\\
 	For all bijections $ \hat \sigma\colon X \xrightarrow{\cong} W$ and $\hat \rho\colon Y \xrightarrow{\cong} Z$ that extend to bijections $ \sigma \colon X \amalg \{x\} \xrightarrow{\cong} W \amalg \{w\}$ and $ \rho \colon Y \amalg \{y\} \xrightarrow{\cong} Z \amalg \{z\}$,
 		\[ S(\hat \sigma \sqcup \hat \rho)(\phi \diamond ^{W,Z}_{w,z} \psi) = S(\sigma)(\phi) \diamond^{X,Y}_{x,y} S(\rho)(\psi),\]
 	(where $ \hat \sigma \sqcup \hat \rho\colon X \amalg Y \xrightarrow{\cong} W \amalg Z$ is the block permutation).
 	
 \end{enumerate}
 
 A \emph{unit for the multiplication $\diamond$} is a map 
 $\epsilon\colon S_\S \to S_\two $, $c \mapsto \epsilon_c$ such that, for all $X$ and all $\phi \in S_{X \amalg \{x\}}$ with $S(ch_x) = c \in S_\S$, 
 \[ \phi \diamond^{X, \{1\}}_{x,2} \epsilon_c = \epsilon_c \diamond^{\{1\},X}_{2,x} \phi = \phi.\]
A multiplication $\diamond$ is called \emph{unital} if it has a unit $\epsilon$. In this case 
$ \epsilon _c$ is a \emph{$c$-coloured unit for $\diamond$}. 

\end{defn}

If $(\diamond,\epsilon\colon \CCC \to S_\two)$ is a unital multiplication on a $(\CCC, \omega)$-coloured graphical species $S$, then $\epsilon_c \in S_{(c, \omega c)}$ for all $c \in \CCC$. 
Let $\sigma_\two\in \fisinv (\two, \two) $ be the unique non-identity endomorphism. 

\begin{lem} \label{lem: unique and equivariant}If $\diamond$ admits a unit $\epsilon: \CCC \to S_\two$, it is unique. Moreover, $\epsilon$ is compatible with the involutions $\omega = S_\tau$ on $\CCC$ and $S(\sigma_\two)$ in that 
	\begin{equation}
	\label{eq: unit compatible with involution}
	\epsilon \circ \omega = S(\sigma_\two) \circ \epsilon\colon \ \CCC \to S_\two.
	\end{equation}
	%
\end{lem}
\begin{proof}
	
	If $\epsilon \colon S_\S \to S_{\two}$ is a unit for $\diamond$ then, by definition $S(\sigma_\two)\epsilon_c = (S(\sigma_\two)\epsilon_c)\diamond^{\{2\},\{1\}}_{1,2}\epsilon_{\omega c}$ for all $c \in \CCC$. By equivariance $(S(\sigma_\two)\epsilon_c)\diamond^{\{2\},\{1\}}_{1,2}\epsilon_{\omega c} = \epsilon_c \diamond^{\{1\}, \{1\}}_{2,2} \epsilon _{\omega c}$, so 
	\[ S(\sigma_\two)\epsilon_c = (S(\sigma_\two)\epsilon_c)\diamond^{\{2\},\{1\}}_{1,2}\epsilon_{\omega c} = \epsilon_c \diamond^{\{1\}, \{1\}}_{2,2} \epsilon _{\omega c} = \epsilon_{\omega c},\]
	whereby the second statement is proved.
	
	Now, let $\lambda\colon \CCC \to S_\two$, $c \mapsto \lambda_c$ be another unit for $\diamond$. Then, for all $c \in \CCC$,
	\[ \epsilon_c = \epsilon_c \diamond^{\{2\},\{1\}}_{1,2} \lambda_c =\left(S(\sigma_\two) \epsilon_c\right) \diamond^{\{1\},\{1\}}_{2,2} \lambda_c = \epsilon_{\omega c} \diamond^{\{1\},\{1\}}_{2,2} \lambda_c  = \lambda_c. \] Hence multiplicative units are unique. 
\end{proof}

\begin{rmk}	\label{rmk: mult notation}
Equivalently, a multiplication $\diamond$ on $ (\CCC, \omega)$-coloured graphical species $S $ is a family of maps 
 	\begin{equation}\label{eq. coloured mult}- \diamond^{\underline c, \underline d}_{c} - \colon S_{(\underline c, c) } \times S_{(\underline d, \omega c)} \to S_{(\underline c \underline d)}, \ \text{ for } c \in \CCC,\ \underline c \in \CCC^X, \ \underline d \in \CCC^ Y,\end{equation}
Both (\ref{eq: mult def}) and (\ref{eq. coloured mult}) are used in what follows. Where the context is clear, the superscripts may be dropped altogether.
 

\end{rmk}

As one would expect, a multiplication $\diamond$ on a graphical species $S$ is called `associative' if the result of several consecutive multiplications does not depend on their order. This is stated precisely in condition (M1) of \cref{defn: Modular operad}, and visualised in the figure therein.

\begin{ex}\label{cyclic}
	
	A graphical species $O$ equipped with a unital, associative multiplication $(\diamond , \epsilon)$ is a cyclic operad in the sense of \cite{DCH19}. When the involution is trivial, these are the \textit{entries-only} cyclic operads of \cite{CO20} (see there for a comparison with cyclic operads as introduced in \cite{GK95}).

		Some advantages of the involutive, graphical species approach to cyclic operads are discussed in \cite{DCH19} and \cite[Introduction]{HRY19b}.
\end{ex}

\begin{ex}
	\label{ex: operad}
	Operads (see e.g.~\cite{BM07}) admit a description as graphical species with unital multiplication:
	
	Recall, from Examples \ref{Comm c} and \ref{ex:direction}, the graphical species $Di$, and the category $\fisinvdi\simeq \elG[Di]$ whose objects are either the exceptional directed edge $(\downarrow)$, or pairs $(X, Y)$ of finite sets. 

If $Y \cong \{*\}$ is a singleton, then $(X, \{*\})$ is called a rooted corolla, and denoted by $t_X$ (\cref{fig: directed elements}(c)). Let $\Bifiso \subset \fisinvdi$ be the full subcategory on $(\downarrow)$ and all rooted corollas $t_X$. 


%

 Presheaves $O \colon{ \Bifiso }^{\mathrm{op}} \to \Set$ are described by a set $\DDD = O(\downarrow)$ and sets $O(\underline c;d)$, defined for all $d \in \DDD$ and $\underline c \in \DDD^X$ (for all $X$), and such that the action of $\fiso$ on $O$ induces isomorphisms $O((c_x)_x; d) \cong O((c_{f(x)})_x;d)$ for all $f \colon X \xrightarrow {\cong} Y$. Hence, a $\DDD$- coloured operad is a $\Bifiso$ presheaf $O$, together with an operadic composition, and a $d$-coloured unit $1_d \in O(d;d)$ for each $d \in \DDD$. 

The graphical species $RC \subset Di$ is given by $ \elG[RC] \xrightarrow \simeq \Bifiso$ under the restriction of the equivalence $\elG[Di] \xrightarrow \simeq \fisinvdi$. So, $RC_\S = Di_\S = \{\In, \Out\}$, $RC_\nul = \emptyset$, and $RC_{X \amalg \{*\}}$ consists of those $\phi \in Di_{X \amalg \{*\}}$ such that 
\[ Di (ch_x)(\phi) = (\In) \text { for all } x \in X, \text { and } Di (ch_{*}) (\phi) =(\Out). \] 
%

Clearly, $RC$ inherits the trivial unital multiplication from $Di$. Moreover, a presheaf $O \colon {\Bifiso}^{\mathrm{op}} \to \Set$ has the structure of an operad precisely when the corresponding graphical species $O^{\GS} \in \GS \ov RC$ is equipped with an associative unital multiplication. 
Hence, the category $\Op$ of (symmetric) operads 
 is equivalent to the category whose objects are objects of $\GS \ov RC $ 
with an associative unital multiplication,
and whose morphisms are morphisms in $ \GS \ov RC$ that preserve the multiplication and units.

\end{ex}

\begin{rmk}
	Examples \ref{cyclic} and \ref{ex: operad} highlight the expressive power of graphical species. The involution $\tau$ on $\S$ means that (undirected) cyclic operads and (directed) operads may be expressed in terms of presheaves on the same underlying category. (See also Examples \ref{ex:direction} and \ref{ex: wheeled prop}.)
\end{rmk}


Intuitively, a contraction $\zeta$ on a graphical species $S$ may be thought of as a rule for `self-gluing' single elements of $S$ along pairs of ports with dual colouring (\cref{contraction}). The presence of a contraction operation enables modular operads to encode algebraic structures -- such as those involving trace -- that ordinary operads cannot \cite{MMS09, Mer10}.
\begin{figure}[htb!] 
	
	\begin{tikzpicture}
	
	\node at (0,0){\begin{tikzpicture}[scale = 0.5]
		\foreach \angle in {-0,90,180,270} 
		{
			\draw(\angle:0cm) -- (\angle:1.5cm);

		}
		\draw [ thick] (0,0) -- (1.5,0);
		\draw [ thick] (0,0) -- (-1.5,0);
		\node at (-1, 0.3) {\tiny c};
		\node at (-1, -0.3) {\tiny (x)};
		\node at (1, -0.3) {\tiny (y)};
		\node at (1, 0.3) {\tiny {$\omega c$}};
		
		\draw [ draw=red, fill=white]
		(0,0) circle (15pt);
		
		\node at(0,0) {\small{$ \phi $}};
		\end{tikzpicture}};

	\node at (2.5,0){\large $\longrightarrow$};
	
	\node at (5,0) 
	{\begin{tikzpicture}[scale = 0.5]
		\draw (0,-1.5)--(0,3.5);
		\draw[ thick] 
		(0,0)..controls (3,3) and (-3,3)..(0,0);
		\node at (-1.5,1) {\tiny {$(y)$}};
		\node at (1.5, 1) {\tiny {$(x)$}};

		\draw [ draw=red, fill=white]
		(0,0) circle (15pt);
		
		\node at(0,0) {\tiny{$\phi $}};
		\draw [ draw=blue]
		(0,1) ellipse (30pt and 60 pt); 
		\end{tikzpicture}};

	\node at (5,-2){ $ \zeta_{x,y} (\phi)$};
	
\end{tikzpicture}
\caption{ Contraction}\label{contraction}

\end{figure}

\begin{defn}\label{defn: contraction}

A \emph{contraction} $\zeta$ on $ S$ is given by a family of partial maps 
\begin{equation}\label{eq: cont def} \zeta^X_{x,y}\colon S_{X \amalg \{x,y\}} \pto S_{X}\end{equation}
defined for all finite sets $X$ and all $\phi \in S_{X \amalg \{x, y\}},$ such that $ S(ch_x)(\phi) = S(ch_y\circ \tau)(\phi)$, and {equivariant} with respect to the action of $\fiso $ on $S$: 
If $\sigma\colon X \amalg \{x,y\} \xrightarrow{\cong} Z \amalg\{w,z\}$ extends the bijection $ \hat \sigma\colon X \xrightarrow{\cong} Z$ by $\sigma(x) = w,\sigma(y) = z$, then for any $\phi \in S_{Z \amalg \{w,z\}}$, we have
\[ S(\hat \sigma) \left(\zeta^Z_{w,z}(\phi) \right) = \zeta^X_{x,y}\left (S(\sigma)(\phi) \right).\]
 \end{defn}

If $\zeta$ is a contraction on $S$, then by, equivariance, $\zeta^X_{x,y}(\phi)= \zeta^X_{y,x} (\phi)$ wherever defined. 

\begin{rmk}\label{rmk: cont notation}
		 A contraction $\zeta$ on a $(\CCC, \omega)$-coloured graphical species $S$ is equivalently a family of maps 
		\[ \zeta ^{\underline c}_c\colon S_{(\underline c, c, \omega c)} \to S_{\underline c }\]
for $c \in \CCC$, and $\underline c \in \CCC^X$. Depending on context, both $\zeta ^{\underline c}_c$ (and even $\zeta_{c}$
) and (\ref{eq: cont def}) will be used. 

\end{rmk}

Let $S$ be a $(\CCC, \omega)$-coloured graphical species equipped with a unital multiplication $(\diamond, \epsilon)$ and contraction $\zeta$. By \cref{lem: unique and equivariant}, there is a \emph{contracted unit} map
\begin{equation}\label{eq: unit contraction} o \defeq \zeta \epsilon: \CCC \to S_\nul, \text{ satisfying } \zeta_c (\epsilon_c) = \zeta_{\omega c} (\epsilon_{\omega c}) \text{ for all } c \in \CCC. \end{equation}

As will be explained in Sections \ref{degenerate} and \ref{s. Unital}, the contracted units $o\colon {S_\S}\to S_\nul$ present the main challenge for describing the combinatorics of modular operads.

\subsection{Modular operads: definition and examples}
Modular operads are graphical species with multiplication and contraction operations that satisfy the nicest possible (mutual) coherence axioms.

\begin{defn}\label{defn: Modular operad}
A \emph{modular operad} is a graphical species $S$, with palette $(\CCC, \omega)$, say, together with a unital multiplication $(\diamond,\epsilon)$, 
 and a contraction $\zeta$, that together satisfy the following four \emph{coherence axioms} 
governing their composition:

%
\begin{minipage}{.55\textwidth}
\emph{(M1)}	\emph{Multiplication is associative.}\\
For all $\underline b \in \CCC^{X_1}, \underline c \in \CCC^{X_2}, \underline d\in \CCC^{X_3}$ and all $c, d \in \CCC$, the following square commutes:
\[
\xymatrix{
	S_{(\underline b, c)} \times S_{(\underline c, \omega c, d)} \times S_{(\underline d, \omega d)} 
	\ar[rr]^-{\diamond_c \times id}
	\ar[dd]_{id \times \diamond_d} &&
	S_{(\underline b \underline c, d)}\times S_{(\underline d, \omega d)} 
	\ar[dd]^{\diamond_d }\\
	&{}&\\
	S_{(\underline b, c)}\times S_{(\underline c, \omega c, \underline d)}
	\ar[rr]_-{\diamond_c} &&
	S_{\underline b\underline c\underline d}. }
\]

\end{minipage}
\begin{minipage}{.45\textwidth}
	\begin{figure}[H]
	\includestandalone[width = \textwidth]{standalones/axiom1}
\end{figure}
\end{minipage}

\begin{minipage}{.55\textwidth}
	\emph{(M2)} \emph{Order of contraction does not matter.} \\For all $\underline c \in \CCC^X$ and $c, d \in \CCC$, the following square commutes:
	 	 \[\xymatrix{
	S_{(\underline c, c, \omega c, d, \omega d)} \ar[rr]^-{\zeta_c }\ar[dd]_{\zeta_d} && 
	S_{(\underline c, d, \omega d) }\ar[dd]^{\zeta_d}\\
	&{}&\\\
	S_{(\underline c, c, \omega c)}\ar[rr]_-{\zeta_c} &&S_{\underline c}.}\]

\end{minipage}
\begin{minipage}{.45\textwidth}
	\begin{figure}[H]
		\includestandalone[width = \textwidth]{standalones/axiom2}
	\end{figure}
\end{minipage}

\begin{minipage}{.5\textwidth}
	\emph{(M3)}\emph{ Multiplication and contraction commute.} \\For all $\underline c \in \CCC^{X_1}$, $\underline d \in \CCC^{X_2}$ and $c, d \in \CCC$, the following square commutes.
	\[
	\xymatrix{
	S_{(\underline c, c, \omega c, d) } \times S_{(\underline d, \omega d)}
	\ar[rr]^-{\zeta_c \times id}
	\ar[dd]_ {\diamond_d} &&
	S_{(\underline c, d)} \times S_{(\underline d, \omega d)} 
	\ar[dd]^{\diamond_d}\\
	&{}&\\
	S_{(\underline c,c,\omega c,\underline d)}
	\ar[rr]_-{\zeta_c} &&
	S_{\underline c\underline d}}\]
\end{minipage}
\begin{minipage}{.48\textwidth}
	\begin{figure}[H]
		\includestandalone[width = \textwidth]{standalones/axiom3}
	\end{figure}
\end{minipage}

\begin{minipage}{.55\textwidth}
	\emph{(M4)} \emph{`Parallel multiplication' of pairs.}\\For all $\underline c \in \CCC^{X_1}$, $\underline d \in \CCC^{X_2}$, and $c, d \in \CCC$, the following square commutes:
	\[
	\xymatrix{
	S_{(\underline c, c, d)} \times S_{(\underline d, \omega c, \omega d)}
	\ar[rr]^-{ \diamond_c}
	\ar[dd]_{ \diamond_d}&&
	S_{(\underline c, d, \underline d, \omega d)}
	\ar[dd]^{\zeta_d}\\&{}&\\
	S_{(\underline c,c, \underline d, \omega c)}\ar[rr]_-{\zeta_c}&&
	S_{\underline c \underline d}}\]

\end{minipage}
\begin{minipage}{.45\textwidth}
	\begin{figure}[H]
		\includestandalone[width = \textwidth]{standalones/axiom4}
	\end{figure}
\end{minipage}

Modular operads form a category $\CSM$ whose morphisms are morphisms of the underlying graphical species that preserve multiplication, contraction and multiplicative units. 

\end{defn}

 Informally, the multiplication and contraction operations describe rules for \textit{collapsing} edges of graphs that represent formal compositions of elements. 
 The coherence axioms (M1)-(M4) say that this is independent of the order in which the edges are collapsed. 

%
%
%
%

\begin{rmk} \label{rmk:non-unital}
 A \emph{non-unital modular operad} $(S, \diamond, \zeta)$ is a graphical species $S$ equipped with a multiplication $\diamond$ and contraction $\zeta$ satisfying (M1)-(M4), but without the requirement of a multiplicative unit. These form a category $\nuCSM$ whose morphisms are morphisms in $\GS$ that preserve the multiplication and contraction operations. Non-unital modular operads are the subject of \cref{sec: non-unital}.
\end{rmk}
%
%
%
%
%
 To provide context and motivation for the constructions that follow, the remainder of this section is devoted to examples. 
\begin{ex}\label{ex: gk mod op}
\label{cob} {\textbf{Getzler-Kapranov modular operads.}}
The monochrome graphical species $M$ given by $M_{X} = \N$ for all $X \in \fiso$, admits a 
unital multiplication $(+, 0 \in M_\two)$ induced by addition in $\N$, and a contraction $t$ induced by the successor operation:
\[ +\colon M_{\mm} \times M_\nn \to M_{\mathbf{m + n - 2}}, \ (g_m, g_n) \mapsto g_m+g_n \ (m,n \geq 1); \] 
	\[ t\colon M_\nn \to M_{\mathbf{n - 2}} , \ g_n \mapsto g_n+1 \ ( n \geq 2).\]

Since a compact oriented surface with boundary is determined, up to homeomorphism, by its genus and number of boundary components, the combinatorics of $(M, + , s)$ describe gluing of topological surfaces along boundary components (see \cref{fig. gluing}). A monochrome object $(S, \gamma)$ of the slice category $\CSM\ov M$ describes a bigraded set $(S^\gamma(g,n))_{g,n}$ with operations
\[ +^S \colon S^\gamma(g_1, n_1) \times S^\gamma(g_2, n_2) \to S^\gamma (g_1 + g_2, n_1 + n_2 -2) \text { for } n_1, n_2 \geq 1, \] \[
t^S\colon S^\gamma(g, n) \to S^\gamma(g+1, n-2) \text { for } n \geq 2,\]
and may encode (moduli spaces) of geometric structures on surfaces. For example, 
the Deligne-Mumford compactification $\overline {\mathcal M}_{g,n}$ of the moduli space of genus $g$ smooth algebraic curves with $n$ marked points, may be described, via Belyi's Theorem, in terms of the space of genus $g$ Riemann surfaces with $n$ nodes, and the spaces $\overline {\mathcal M}_n \defeq \coprod_{g \in \N} \overline {\mathcal M}_{g,n}$ form a monochrome modular operad (c.f.~\cite[Example~6.2]{GK98}). 
%
%
%

Getzler and Kapranov originally defined modular operads \cite{GK98} in terms of the restriction to the \emph{stable part} $M^{st} \subset M$ of the graphical species $M$, bigraded by pairs $(g,n)$ such that $2g + n -2 > 0$. So $M^{st}_\nn = M_\nn$ for $n >2$ but $M^{st}_\nul = \{ 2, 3, 4, \dots\}$ and $M^{st}_{\mathbf{1}} = M^{st}_{\two} = \{1, 2, 3, 4, \dots\}.$

In particular, since $0 \not \in M^{st}_{\two}$, 
modular operads in the original sense of \cite{GK98} are non-unital. 

These ideas may be extended to many-coloured cases: for example, one can describe a 2-coloured modular operad for gluing surfaces along 
{open and closed} subsets of their boundaries. (See, e.g.~\cite{Gia13}.)
\end{ex}

\begin{ex}\label{ex: compact closed} {\textbf{Compact closed categories}}, introduced in \cite{Kel72}, are symmetric monoidal categories $(\CCat, \otimes, e)$ for which every object $c \in \CCat$ has a symmetric \textit{categorical dual} (see \cite{BG99,KL80}): there is an object $c^* \in \CCat$, and morphisms $\cap_c \colon e \to c^* \otimes c $ and $ \cup_c\colon c \otimes c^* \to e $ such that
	\[(\cup_c \otimes id_c) \circ(id_c \otimes \cap_c) = id_c = (\cap_{c^*}\otimes id_c)\circ (id_c \otimes \cup_{c^*}).\]
	
	Examples of compact closed categories include categories of finite dimensional vector spaces over a given field, or, more generally, finite dimensional projective modules over a commutative ring. Cobordism categories provide other important examples. 
	
There is a canonical monadic adjunction $\CSM \leftrightarrows \comCinv$, where $\comCinv$ is the category of \emph{involutive compact closed categories}, whose objects are small compact closed categories $\CCat$ 
such that $c = c^{**}$ for all $c \in \CCat$. 
	The right adjoint takes an involutive compact closed category $(\CCat, \otimes, e,*)$ with object set $\CCat_0$ to a $(\CCat_0, *)$-coloured modular operad $S^\CCat$ with coloured arities
\[ S^\CCat_{(d_1, \dots,d_n, c_m^*,\dots, c_1^*)} =\CCat(c_1 \otimes \dots \otimes c_m, d_1 \otimes \dots \otimes d_n).\] 
The modular operad structure on $S^\CCat$ is induced by composition in $\CCat$ together with $\cup$ and $\cap$. The left adjoint $\CSM \to \comCinv$ is induced by the free monoid functor on palettes and arities. 

These observations underly the proof, in \cite{Ray21b}, of an `operadic' nerve theorem for compact closed categories in the style of \cref{s. nerve}. 

\end{ex}

\begin{ex}
	\label{ex: circuits} \textbf{Circuit algebras} -- so named because of their resemblance to electronic circuits -- are a symmetric version of Jones's planar algebras, introduced to study finite-type invariants in low-dimensional topology \cite{BND17, DHR20}. 
	
The category of $\Set$-valued circuit algebras is equivalent to a category $\CO$ of \textit{circuit operads} \cite{Ray21a} whose objects are modular operads equipped, via a monadic adjunction $\CSM \leftrightarrows \CO$, with an extra `external product' operation. Moreover, the adjunction between modular operads and involutive compact closed categories in \cref{ex: compact closed} factors through the adjunction $\CSM \leftrightarrows \CO$. 
	

%
	
	This formal perspective on modular operads, circuit algebras, and compact closed categories leads to interesting questions in a number of directions. For example, we can study the analogous relationships if the definition of modular operads is relaxed by replacing the symmetric action with a braiding, or by considering higher dimensional versions. Related ideas are being explored by Dansco, Halacheva and Robertson in their work on algebraic and categorical structures in low-dimensional topology \cite{DHR21}.
\end{ex}

\begin{ex}\label{ex: wheeled prop}
	\label{wheeled properads intro} {\textbf{Wheeled properads.}} 
		Wheeled properads have been studied extensively in \cite{HRY15} and \cite{JY15}. They describe the \textit{connected part} (c.f.~\cite[Introduction]{Val07}) of wheeled PROPs (i.e.~coloured PROPs with a contraction) that have applications in geometry, deformation theory, and other areas \cite{MMS09,Mer10}. 
	
	The category $\mathsf{WP}$ of ($\Set$-valued) wheeled properads is canonically equivalent to the slice category $\CSM\ov Di$ of \textit{directed modular operads}. This is well-defined since the terminal directed graphical species $Di$ trivially admits the structure of a modular operad (see \cref{ex:direction}). 
An equivalence between wheeled PROPs in linear categories and directed circuit algebras is established in \cite{DHR20}. 
	

\end{ex}

\section{Abstract nerve theorems and distributive laws}\label{sec: Weber}

The purpose of this largely formal section is to review some basic theory of distributive laws, and provide an overview of Weber's abstract nerve theory. The simplicial nerve for categories, and the dendroidal nerve for operads provide motivating examples for the latter.

For an overview of monads and their Eilenberg--Moore (EM) categories of algebras, see for example {\cite[Chapter~VI]{Mac98}.}

\subsection{Monads with arities and abstract nerve theory}

Given an essentially small category $\CCat$, a functor $F\colon \DCat \to \CCat$ induces a \emph{nerve functor}
 $N_{\DCat}\colon \CCat \to \pr{\DCat}$ by $N_{\DCat}(c)(d) = \CCat(F d, c) $ for all $c \in \CCat$, $ d \in \DCat$. If $N_{\DCat}$ is fully faithful, and $F$ and $\DCat$ are suitably nice, then $N_{\DCat}$ provides a useful tool for studying $\CCat$.

 In the crudest sense, monads with arities are monads whose EM category of algebras may be characterised in terms of a fully faithful nerve, the construction of which is entirely abstract. 
The aim of this section is to explain, without proofs, the key points of this abstract nerve theory (details may be found in \cite[Sections~1-3]{BMW12}). This motivates the framework of this paper, and underlies the proof of the nerve theorem for modular operads, \cref{nerve theorem} in \cref{s. nerve}.

\vskip 2ex

 Recall that every functor 
admits an (up to isomorphism) unique \textit{bo-ff factorisation} as a bijective on objects functor followed by a fully faithful functor. For example, if $\MM$ is a monad on a category $\CCat$, and $\CCat^{\MM}$ is the EM category of algebras for $\MM$, then the free functor $ \CCat \to \CCat^\MM$ has bo-ff factorisation $\CCat \rightarrow \CCat_{\MM} \rightarrow \CCat^{\MM}$, where $\CCat_{\MM}$ is the {Kleisli category} of free $\MM$-algebras (see e.g.\ \cite[Section~VI.5]{Mac98}). 

Hence, for any subcategory $\DCat $ of $ \CCat$, the bo-ff factorisation of the canonical functor $\DCat \hookrightarrow \CCat \xrightarrow {\text{ free}} \CCat^\MM$ factors through the full subcategory $\Theta_{\MM, \DCat}$ of $\CCat_\MM $ with objects from $\DCat$. 

By construction, the defining functor $\Theta_{\MM, \DCat} \to \CCat^{\MM}$ is fully faithful. It is natural to ask if there are conditions on $\DCat$ and $\MM$ that ensure that the induced nerve $N_{\MM, \DCat} \colon \CCat^{\MM} \to \pr{\Theta_{\MM, \DCat}}$ is also fully faithful. This is the motivation for describing monads with arities. 


 

%
 	 	
\begin{defn}\label{def. replete}
The \emph{essential image} $im^{es}(F)$ of a functor $F\colon \DCat[E] \to \CCat$ is the smallest subcategory of $\CCat$ that contains the image $im (F)$ of $F$ in $\CCat$ and is closed under isomorphisms in $\CCat$. 

A subcategory $\iota \colon \DCat \hookrightarrow \CCat$ is a \emph{dense subcategory} (and $\iota$ is a \emph{dense functor}) if the induced nerve $N_{\DCat }\colon \CCat \to \pr{\DCat}$ is full and faithful.
\end{defn}

Once again, let $\MM= (M, \mu^\MM, \eta^\MM)$ be a monad on $\CCat$. Let $\iota \colon \DCat \to \CCat$ be the inclusion of a dense subcategory, and let $\boffcat$ be obtained in the bo-ff factorisation of $\DCat \to \CCat^\MM$. There is an induced diagram of functors 
\begin{equation} \label{eq: arities}
\xymatrix{ 
	\boffcat\ar[rr]^{\text{f.f.}}						&& 	\CCat^\mathbb{M}\ar@<2pt>[d]^-{\text{forget}}
	\ar [rr]^-{N_{\mathbb{M}, \DCat}}				&&\pr{	\boffcat}\ar[d]^{j^*}	\\
	\DCat\ar@{^{(}->} [rr]^-{\text{dense}}_-{\text{}}	 \ar[u]_{j}^{\text{b.o.}}		&&\CCat \ar [rr]_{\text{f.f.}}^-{N_{\DCat}}	 \ar@<2pt>[u]^-{\text{free}}	
	&&\pr{\DCat}.
}
\end{equation} where $j^*$ is the pullback of the bijective on objects functor $j \colon \DCat \to \Theta_{\MM, \DCat} $. The left square of (\ref{eq: arities}) commutes by definition, and the right square commutes up to natural isomorphism. 

By \cite[Proposition~5.1]{Lam66}, 
the inclusion 
$\iota \colon \DCat \to \CCat$ is dense if and only if every object $c$ of $\CCat$ is given canonically by the colimit of the functor $\DCat \ov c \to \CCat$, $(d,f) \mapsto \iota(d)$. 

The monad \textit{$\MM$ has arities $\DCat$} if $N_{\DCat } \circ M $ takes the canonical colimit cocones $\DCat \ov c$ in $\CCat$ to colimit cocones in $\pr{\DCat}$. In this case, by \cite[Section~4]{Web07}, the full inclusion $\boffcat \to \CCat^{\MM}$ is dense, and the essential image of the induced fully faithful nerve $N_{\MM, \DCat}\colon \CCat^{\MM} \to \pr{\boffcat}$ is the full subcategory of $\pr{ \boffcat}$ on those presheaves $P$ whose restriction $j^* P$ to $\DCat$ is in the essential image of $N_{\DCat} \colon \CCat \to \pr{\DCat}$. 
\begin{rmk}
	\label{ex: arities not necessary}
	The condition that $\mathbb{M}$ has arities $\DCat \hookrightarrow \CCat$ is sufficient, but not necessary, for the induced nerve $\CCat^\MM \to \pr{\boffcat}$ to be fully faithful. 
	
In fact, by \cref{nerve theorem} and \cref{rmk: Gr no arities}, the modular operad monad $\OOO$ on the category of graphical species, together with the full dense subcategory $\Gr\hookrightarrow \GS$ of connected graphs and \'etale morphisms (described in \cref{sec: topology}), provides an example of a monad that does not have arities, but for which the nerve theorem holds. 
	
	Necessary conditions on $\MM$ and $\DCat \hookrightarrow \CCat$, for the induced nerve to be fully faithful 
	are	described in \cite{BG19}. 

\end{rmk}

\begin{ex} \label{ex: classical nerve}
	
	Recall that directed graphs $ \mathbf G = \left(\mathfrak s, \mathfrak t\colon E \rightrightarrows V\right)$ are presheaves over the small diagram category 
	$\E \defeq \bullet \rightrightarrows\bullet$, and that the canonical forgetful functor from $\Cat$ to $\pr{\E}$ -- that assigns to a small category $\CCat$, the directed graph $\mathbf G_\CCat$ with vertex set $V_\CCat$ indexed by objects of $\CCat$, and edge set $E_\CCat$ indexed by morphisms of $\CCat$ -- is monadic. So, every directed graph freely generates a small category.
	
	
For $n \in \N$, the finite ordinal $[n]$ may be viewed as a directed linear graph: \begin{equation}\label{eq: cat Segal}[n] = \overset{0}{\bullet} \longrightarrow \overset{1}{\bullet} \longrightarrow \dots \longrightarrow \overset{n}{\bullet}. \end{equation} 

The free category on $[n]$ is the $n$-simplex $\Delta (n)$, and $\Delta$ is the \textit{simplex category} of \textit{simplices} $\Delta(n)$, $n \in \N$, and functors between them. 
	 The category of $\Delta$-presheaves, or \emph{simplicial sets}, is denoted by $\sSet$.
	
	The \textit{classical nerve theorem} states that the induced nerve functor $N_{\Delta} \colon \Cat \to \sSet$ is fully faithful. Moreover, its essential image consists of precisely those $P \in \sSet$ that satisfy the classical Segal condition, originally formulated in \cite{Seg68}: a simplicial set $P$ is the nerve of a small category if and only if, for $n >1$, the set $P_n$ of $n$-simplices is isomorphic to the $n$-fold fibred product 
	\begin{equation}\label{eq. classical Segal} P_n \cong \underbrace{P_1 \times _{P_0} \dots \times_{P_0} P_1}_{n \text{ times}}.\end{equation} 
	
	The nerve theorem and Segal condition (\ref{eq. classical Segal}) may be derived using abstract nerve theory: 
		
	
	

Let $ \Delta_0 \subset \pr{\E}$ be the full subcategory on the directed linear graphs $[n]$ whose morphisms $f \colon [m] \to [n]$ 
satisfy $f(i+1) = f(i)+ 1$ for all $0 \leq i <m$. In particular, $\E$ embeds in $\Delta_0$ as the full subcategory on the objects $[0]$ and $[1]$, and the full inclusion $\Delta_0 \hookrightarrow \pr{\E}$ is precisely the nerve induced by the inclusion $\E\hookrightarrow \Delta_0$. Hence $\E$ is dense in $\Delta_0$. Since $\E \hookrightarrow \Delta_0$ is fully faithful, so is $N_{\Delta_0}$ (by \cite[Section~VII.2]{MM94}), so $\Delta_0$ is also dense in $\pr{\E}$.


%
%
%
%
%

Since $\Delta$ is the category obtained in the bo-ff factorisation of $\Delta_0 \to \pr{\E} \to \Cat$, we consider the following diagram of functors 
	\begin{equation} \label{eq: cat weber}
\xymatrix{ 
		&&
	\Delta\ar@{^{(}->} [rr]_-{\text{f.f.}}				&& 
	\Cat \ar@<2pt>[d]^{\Ucat} \ar[rr]^-{N_{\Delta}}	&&
\sSet
	\ar[d]^{j^*}\\ 
	\E \ar@{^{(}->} [rr]^-{\text{dense}}_-{\text{ f.f.}} 	&&
	\Delta_0 \ar@{^{(}->}[rr]^-{\text{dense}}_-{\text{ f.f.}} 	 \ar[u]^{j}_{\text{b.o.}}		&&
	\pr{\E} \ar@{^{(}->}[rr] \ar@<2pt>[u]^{\Fcat}\ar[rr]^{N_{\Delta_0}}_-{\text{ f.f.}} 	&&
	\pr{\Delta_0}. }
\end{equation}	

It is straightforward to prove -- using for example \cite[Sections~1~\&~2]{BMW12} -- that the category monad on $\pr{\E}$ has arities $\Delta_0$. Hence $N_\Delta \colon \Cat \to \sSet$ is fully faithful, and a simplicial set $P$ is in its essential image if and only if $j^* P$ is in the essential image of $N_{\Delta_0}$. Segal's condition (\ref{eq. classical Segal}) follows from the fact that $\E$ is dense in $\Delta_0$. 


\end{ex}

\begin{rmk}
	The notion of graph in \cref{ex: classical nerve} is different and, in a suitable sense, dual to the one used in \cref{ex: dendroidal}, and in the rest of this paper (from \cref{s. graphs}), where edges function as `objects' and connections between them as `morphisms'. 
\end{rmk}

The classical Segal condition (\ref{eq. classical Segal}) may be generalised as follows: 

As before, let $\DCat \subset \CCat$ be a dense subcategory, and, as in \cref{ex: classical nerve}, let $\CCat = \pr{\E} $ be the category of presheaves on a dense subcategory $ \E$ of $\DCat$. So, the dense inclusion $\DCat \hookrightarrow \CCat$ is also full. If $\DCat$ provides arities for a monad $\MM$ on $\CCat$, then by \cite[Lemma~3.6]{BMW12}, 
a presheaf $P \colon \boffcat ^{\mathrm{op}} \to \Set$ is in the essential image of $N_{\MM, \DCat}$ if and only if 
\begin{equation}
\label{eq: sheaf condition}
P(jd) = \mathrm{lim}_{(e,f) \in \E \ov d } \ j^*(P)(e) \ \text{ for all } d \in \DCat .
\end{equation}
Equation (\ref{eq: sheaf condition}) 
is called the \textit{Segal condition} for the nerve functor $N_{\MM, \DCat}$.

\begin{ex}\label{ex: dendroidal}

The category $\Bifiso$ -- whose objects are the directed exceptional edge $(\downarrow)$, and the rooted corollas $t_X$ (for all finite set $X$) -- was describe in  \cref{ex: operad}. Recall that an operad is a presheaf $O$ on $\Bifiso$, 
	together with a unital composition operation satisfying certain axioms. The forgetful functor $\Op \to \pr{\Bifiso}$ is monadic, so every presheaf $O$ on $\Bifiso$ freely generates an operad. Let $\MMop$ be the induced monad.
	
	
	
	 \textit{Rooted trees} $\Tr$ are obtained as formal colimits of finite diagrams in $\Bifiso$ that describe \textit{grafting} of objects of $\Bifiso$ root-to-leaf as in \cref{fig:grafting}(b). Let $\Omega_0$ be the category whose objects are such rooted trees $\Tr$ and whose morphisms $\Tr[S] \to \Tr$ are (up to isomorphism) inclusions of rooted trees that preserve vertex valency (as in \cref{fig:grafting}(a)). Then $\Bifiso \subset \Omega_0$ is the full and dense subcategory of rooted trees with zero or one vertex.

	\begin{figure}[htb!]
		\[
		\begin{tikzpicture} \node at (3.3,2){(a)};
			
			\node at (4.5,0){	\begin{tikzpicture}[scale = 0.3]
				
				\draw[thick, cyan]
	
				(1,0)--(1,2)
				(1,2)--(3,4);
				\node at (1.6,.8){\scriptsize{$e_0$}};
					\draw[thick,cyan]
				
				(1,2)--(-.5,3.5)
				(3,4)--(3,6)
		%
				(3,4)--(1,6)
				(3,4)--(4.5,6)
				(3,4)--(6,6);
				
				\filldraw [cyan]
	
				(1,2) circle (6pt)
				(3,4) circle (6pt)
				;
				
				\end{tikzpicture}};
			\draw[thick, ->](5,0)--(6,0);
			
			\node at (8, 0){	\begin{tikzpicture}[scale = 0.3]
				
				\draw[thick]
				(0,-2)--(0,0)
				(0,3)--(0,4.5)
				
				(-2,2)--(0,0)
				(-2,5)--(-0,3)
				(3,5.5)--(2,7)
				(3,5.5)--(4,7)
				;
				
				\draw[thick,cyan]
				(0,0)--(1,2)
				(1,2)--(3,4);
				\draw[thick, cyan]
				(1,2)--(0,3)
				(3,4)--(3,5.5)
				
				(3,4)--(1,6)
				(3,4)--(4.5,6)
				(3,4)--(6,6);
				
				\filldraw 
				(0,0) circle (6pt)
				(0,3) circle (6pt)
				(0,4.5) circle (6pt)
				(3,5.5) circle (6pt);
		\filldraw[cyan]
				(1,2) circle (6pt)
				(3,4) circle (6pt)
				;
				\node at (1.2, .8){\scriptsize{$e_0$}};
				
				\end{tikzpicture}};

		\node at (11,2){(b)};	
		\node at (12.5,0)
		{
		\begin{tikzpicture}[scale = 0.25]
			\node at (0,-5){\begin{tikzpicture}[scale = 0.3]
				
				\draw[thick](0,-1)--(0,-3)
				(0,-1)--(-1.5,1)
				(0,-1)--(1.5,1)
				;
				
				\draw[red, thick, -<-=.5]	(0,-1) --(0,2);

				\draw [ draw=black, fill=white]
				(0,-1) circle (8pt);
				
				%

				\end{tikzpicture}};
			
			\node at (0,3){\begin{tikzpicture}[scale = 0.3]
				
				\draw[thick]
				(0,2)--(-1,4)
				(0,2)--(-.5,4)
				(0,2)--(1,4)
				;
				
				\draw[blue, thick, -<-=.5]	(0,-1) --(0,2);
				
				\draw[dashed]
				(-.3,4)--(.8,4);
				
				%
				
				
				\draw [ draw=black, fill=white]
				(0,2)circle (8pt);

				%

				\end{tikzpicture}};
			
			\node at (-10, 0){ \begin{tikzpicture}[scale = 0.3]

				\draw[black, thick, -<-=.5]	(0,-1) --(0,2);

				\end{tikzpicture}};
			
			\draw[ultra thick, blue, ->-=1]	(-8, 1) --(-3,2);
			\draw[ultra thick, red, ->-=1]	(-8, -1) --(-3,-3);
			
			\end{tikzpicture}
		};
		
		\draw[%
		gray, thick,
		decorate,decoration={%
			,zigzag
			,amplitude=2pt
			,segment length=2mm,pre length=8pt
		}
		] (14.7,0) -- (15.7,0);
		\draw[gray, thick, ->-=1](15.7,0) -- (16,0);
		\node at (15.3, .3){\small{colimit}};
		
		\node at (17.5,0){
			\begin{tikzpicture}[scale = 0.45]
			
			\draw[thick](0,-1)--(0,-3)
			(0,-1)--(-1.5,1)
			(0,-1)--(1.5,1)
			(0,2)--(-1,4)
			(0,2)--(-.5,4)
			(0,2)--(1,4)
			;
			
			\draw[black, ultra thick, -<-=.5]	(0,-1) --(0,2);
			
			\draw[dashed]
			(-.3,4)--(.8,4);
			
			%
			
			
			\draw [ draw=black, fill=white]
			(0,2)circle (8pt)
			(0,-1) circle (8pt);
			
			%

			\end{tikzpicture}};
		\end{tikzpicture}	\]
	\caption{(a) Subtree inclusion, (b) grafting of rooted corollas to form a rooted tree.}	\label{fig:grafting}
	\end{figure}

Hence, the induced nerve $\Omega_0 \to \pr{\Bifiso}$ is full and faithful, and $\Bifiso$ canonically induces a topology on $ \Omega_0$ whose sheaves are precisely $\Bifiso$-presheaves. In particular, $ \Omega_0 \to\pr{\Bifiso}$ is also dense (see e.g.~\cref{subs. sheaves} for comparison),
and there is a diagram of functors 
%
%
%
		\begin{equation} \label{eq: op weber}
	\xymatrix{ 
		&&
		\Omega\ar@{^{(}->} [rr]				_-{\text{ f.f.}}		&& 
	\Op \ar@<2pt>[d]^{\Uop} \ar[rr]	^-{N_{\Omega}}	&&
		\pr{\Omega}
		\ar[d]^-{j^*}\\ 
	\Bifiso 
		\ar@{^{(}->} [rr]	^-{\text{dense }}	_-{\text{ f.f.}}	&&
		\Omega_0 \ar@{^{(}->}[rr] 	^-{\text{dense }}	_-{\text{ f.f.}}	 \ar[u]^{j}_{\text{b.o.}}		&&
		\pr{\Bifiso} \ar@{^{(}->}[rr]	^-{N_{\Omega_0}}	_-{\text{ f.f.}} \ar@<2pt>[u]^{\Fop}	&&
		\pr{\Omega_0} }
	\end{equation}
	in which the left square commutes and the right square commutes up to natural isomorphism.
The full subcategory of $\Op$ induced by the bo-ff factorisation of the functor $ \Omega_0 \to \pr{\Bifiso} \to \Op$ 
is the \emph{dendroidal category} $\Omega$ of free operads on rooted trees. This is described in \cite{MW07}, where  
it was established that the full inclusion $\Omega \hookrightarrow \Op$ is dense, and hence the \textit{dendroidal nerve} $N_\Omega$ is fully faithful.

It is easy to show, e.g.~using methods similar to those described in \cref{s. nerve}, that the monad $\MMop$ on $\pr{\Bifiso}$ has arities $\Omega_0$. Hence, the abstract nerve theory of \cite{BMW12} may also be used to show that the nerve functor $N_\Omega\colon \Op\to \pr{\Omega}$ is fully faithful and its essential image consists of those $\Omega$-presheaves (or \textit{dendroidal sets}) 
 $O\colon \Omega^\mathrm{op} \to \Set$ that satisfy the \textit{dendroidal Segal condition} first proved in \cite[Corollary~2.6]{CM13}:
	\begin{equation}
	\label{eq: dendroidal segal}
	O(\Tr) = \mathrm{lim}_{ (t, i) \in (\Bifiso \ov \Tr)} j^*O(t) \ \text{ for all symmetric rooted trees } \Tr.
	\end{equation}

\end{ex}

In particular, since $ \Delta_0$ is the full subcategory of linear trees in $ \Omega_0$, the simplicial nerve theorem for categories is a special case of the dendroidal nerve theorem for operads.

%
%
\begin{defn}\label{defn: pointed endo} 
	A \emph{pointed endofunctor} on a category $\CCat$ is an endofunctor $E$ on $\CCat$ together with a natural transformation $\eta^E\colon 1 _{\CCat}\Rightarrow E$. An \emph{algebra for a pointed endofunctor} $(E, \eta^E)$ on $\CCat$ is a pair $(c, \theta)$ of an object $c$ of $\CCat$ and a morphism $\theta \in \CCat (Ec,c)$ such that $ \theta \circ \eta^{\CCat}_c = id_c \in \CCat(c,c).$ \end{defn}

For example, modular operads are algebras for the pointed endofunctor on $\GS$ described in \cite{JK11}. However, as discussed in \cref{degenerate}, the abstract nerve machinery of \cite{BMW12} cannot be modified for algebras of (pointed) endofunctors:

For any monad $\MM = (M, \mu^\MM, \eta^\MM)$ on a category $\CCat$, the EM category $\CCat^{\MM}$ of $\MM$-algebras embeds canonically in the category $\CCat^M$ of algebras for the pointed endofunctor $(M, \eta^\MM)$. The induced free functor $C  \to \CCat^M$, $c \mapsto (Mc, \mu^{\MM}c) $ factors through $\CCat^{\MM}$ and depends crucially on the monadic multiplication $\mu^\MM \colon M^2 \Rightarrow M$ of $\MM$. 


%
%
%
%
%

By contrast, for an arbitrary pointed endofunctor $(E, \eta^E)$ is on $\CCat$, there is, in general, no canonical choice of functor $\CCat \to \CCat^E$. 


%
%
%


\subsection{Distributive laws}\label{subs: dist}

With Examples \ref{ex: classical nerve} and \ref{ex: dendroidal} in mind, let us return to the case of modular operads. Recall that graphical species are presheaves on the category $\fisinv$ and that modular operads are graphical species equipped with certain operations. 

Informally, monads are gadgets that encode, via their algebras, {(algebraic) structure} on objects of categories. In \cite{JK11}, it is the combination of the contraction structure $\zeta$, and the multiplicative unit structure $\epsilon$ that provides an obstruction to extending the modular operad endofunctor on $\GS$ to a monad (see \cref{degenerate}). So, one approach to constructing the modular operad monad $\mathbb{O}$ on $\GS$ could be to find monads for the modular operadic multiplication, contraction, and unital structures separately, and then attempt to combine them. 

In general, monads do not compose. Given monads $\MM = (M, \mu^{\MM}, \eta^{\MM})$ and $ \MM' = ({M'}, \mu^{\MM'}, \eta^{\MM'})$ on a category $\CCat$, there is no obvious choice of natural transformation $\mu\colon (M{M'})^2 \Rightarrow M{M'}$ defining a monadic multiplication for the endofunctor $M{M'}$ on $\CCat$.

Observe, however, that any natural transformation $\lambda\colon {M'}M \Rightarrow M{M'}$ induces a natural transformation
\begin{equation}
	\mu_\lambda\colon \xymatrix{(M{M'})^2 \ar@{=>} [rr]^- {M\lambda {M'}} && M^2 {M'}^2\ar@{=>}[rr]^- {\mu^{\mathbb{M}}\mu^{\mathbb{{M'}}}} &&M{M'}. }
\end{equation}

\begin{defn}\label{defn: dist}
	A \emph{distributive law} \cite{Bec69} for $\MM$ and $ \MM'$ is a natural transformation $\lambda\colon{M'}M \Rightarrow M{M'}$ such that the triple $(M{M'}, \mu_\lambda, \eta^{\mathbb{M}}\eta^{\mathbb{{M'}}})$ defines a monad $\mathbb{M} \mathbb{{M'}}$ on $\CCat$.

\end{defn}
A distributive law $\lambda \colon M'M \Rightarrow M' M$ determines how the $\MM$-structures and $\MM'$-structures on $\CCat$ interact to form the structure encoded by the composite monad $\MM \MM'$.


\begin{ex}
	\label{ex: category composite} The category monad on $\pr{\E}$ (\cref{ex: classical nerve}) may be obtained as a composite of the \textit{semi-category monad}, which governs associative composition, and the \textit{reflexive graph monad} that adjoins a distinguished loop at each vertex of a graph $\mathbf G \in \pr{\E}$. 
	The corresponding distributive law encodes the property that the adjoined loops provide identities for the semi-categorical composition. 
	
	(There is also a distributive law in the other direction, 
	but the two structures do not interact in the composite. See also \cref{rmk: TD}.)
		
\end{ex}

As usual, let $\CCat^\MM$ denote the EM category of algebras for a monad $\MM$ on $\CCat$.

By \cite[Section~3]{Bec69}, given monads $\MM, \MM'$ on $\CCat$, and a distributive law $\lambda\colon {M'}M \Rightarrow M{M'}$, there is a commuting square of strict monadic adjunctions:
\begin{equation}\label{eq: adjunction diag}\small
{\xymatrix@C = .5cm@R = .3cm{
\CCat^{\MM'} \ar@<-5pt>[rr]\ar@<-5pt>[dd]&\scriptsize{\top}&\CCat^{\MM\MM'} \ar@<-5pt>[ll]\ar@<-5pt>[dd]\\
	\vdash&&\vdash\\
\CCat \ar@<-5pt>[uu] \ar@<-5pt>[rr]&\scriptsize{\top}& \CCat^\MM. \ar@<-5pt>[ll]\ar@<-5pt>[uu]}}\end{equation}




In \cref{sec: topology}, it is shown that the category $\Gr$ of connected Feynman graphs and \'etale morphisms (first defined in \cite{JK11}) fits into a chain $\fisinv \hookrightarrow \Gr \hookrightarrow \GS$ of fully faithful dense embeddings. And, in \cref{s. Unital}, the modular operad monad $\mathbb O$ on $\GS$ is constructed as a composite $\DD\TT$ of monads $\TT$ (that governs contraction and non-unital multiplication) and $\DD$ (that governs multiplicative units) on $\GS$.

 Hence, by (\ref{eq: adjunction diag}), there is a monad $\TTp$ on the EM category $\GSp$ of $\DD$-algebras, such that $\GSp^{\TTp} \cong \CSM$ and a diagram of functors 
\begin{equation} \label{big picture}
\xymatrix { 
	&&\Klgr\ar@{^{(}->} [rr]^-{\text{ }}_-{\text{ f.f.}}						&& \CSM\ar@<2pt>[d]^-{\text{forget}} \ar [rr]^-N				&&\pr{\Klgr}\ar[d]^{j^*}	\\
	\fisinvp \ar@{^{(}->} [rr]	^-{\text{dense }}	_-{\text{ f.f.}}			&& \Grp \ar@{^{(}->} [rr] ^-{\text{dense }}	_-{\text{ f.f.}}	 \ar[u]^{j}_{\text{b.o.}}			&& \GSp \ar@<2pt>[d]^-{\text{forget}} \ar@{^{(}->} [rr]^-{\text{ }}_-{\text{f.f.}}	 \ar@<2pt>[u]^-{\text{free}}	&& \pr{\Grp}\ar[d] \\
	\fisinv\ar@{^{(}->} [rr]^-{\text{dense }}_-{\text{ f.f.}}\ar[u]^{\text{b.o.}}	&& \Gr \ar@{^{(}->} [rr]^-{\text{dense}}_-{\text{ f.f.}}	 \ar[u]_{\text{b.o.}}	&& \GS \ar@{^{(}->} [rr]_-{\text{ }}_-{\text{f.f.}}	 \ar@<2pt>[u]^-{\text{free}}		&& \pr{\Gr}}
\end{equation}
in which the categories $\fisinvp$, $\Grp$ and $\Klgr$ are obtained via bo-ff factorisations. 

In \cref{s. nerve}, it is shown that $\TTp$ has arities $\Grp$ (see \cref{sec: Weber}), whence it follows that the induced nerve $N\colon \CSM \to \pr{\Klgr}$ is fully faithful and its essential image is characterised in terms of $\fisinvp$. 

	\newcommand{\defretract}[5]{\xymatrix{*[r]{#1} \ar@<1ex>[rrr]^-{#3} \ar@(ul,dl)[]_{#5} &&& #2 \ar@<1ex>[lll]^-{#4}}}

\section{Graphs and their morphisms }\label{s. graphs}

This section is an introduction to Feynman graphs as defined in \cite{JK11}. Most of this section and the next stay close to the original constructions there. Since \cite{JK11} was just a short note, it contained very few proofs, and so relevant results are proved in full here. Extensive examples are also given. Where possible, definitions and examples are presented in a way that builds on \cref{sec: definitions} and highlights similarities with familiar concepts in basic topology. 

This section deals with basic definitions and examples. The following section is devoted to a more detailed study of the topology of Feynman graphs, in terms of their \'etale morphisms.


\subsection{Graph-like diagrams and Feynman graphs}\label{subs:feynman graphs}

 Roughly speaking, a graph consists of a finite set of vertices $V$ and a finite set of connections $\widetilde E$, together with an incidence relation: if $\widetilde E$ is the set of orbits of a set $E$ under an involution $\tau$, then the incidence is a partial function $E \pto V$ that attaches connections to vertices. In this paper, all graphs are finite, and may have loops, parallel edges, and loose ends (ports).

\begin{ex}\label{ex: BM graphs}
	Section 15 of \cite{BB17} provides a nice overview of various graph definitions that appear in the operad literature. 
		The definition that is perhaps most familiar 
		is that found in, for example, \cite{GK98} and \cite{BM08}. There, a graph $G$ is described by sets $V$ of vertices and $E$ of edges, an involution $\hat \tau\colon E \to E$, and an incidence function $\hat t\colon E \to V$. The ports of $G$ are the fixed points of the involution $\hat \tau$. A formal exceptional edge graph $\eta$ is also allowed. Morphisms $\eta \to G$ are choices $\{* \} \to E$ of elements of $E$.
\end{ex}

Feynman graphs are defined similarly to the graphs described in \cref{ex: BM graphs}, except the involution on $E$ must be fixed-point free, while the incidence is allowed to be a partial map $ E \pto V$. These subtle differences make it possible to encode the whole calculus of Feynman graphs in terms of the formal theory of diagrams in finite sets.

\medspace
The category of \emph{graph-like diagrams} is the category $\GrShape$ of functors $\diag^{\mathrm{op}} \to \fin$, 
where $\diag$ is the small category $\label{grdiag}
\xymatrix{
	\bullet \ar@(lu,ld)[]_{}&\bullet \ar[l] _{}\ar[r]^-{}&\bullet},$ and $\fin$ is the category of finite sets and all maps between them.

The initial object in $\GrShape$ is the empty graph-like diagram:
\[
\xymatrix{\oslash =&\nul\ar@(lu,ld)[] && \nul\ar[ll] \ar[rr] &&\nul,}\] and the terminal object $\bigstar$ is the trivial diagram of singletons:
\[
\xymatrix{\bigstar =&{\one}\ar@(lu,ld)[] && {\one} \ar[ll] \ar[rr] &&{\one}.}\]

Feynman graphs, introduced in \cite{JK11}, are graph-like diagrams satisfying extra properties:

\begin{defn}\label{defn: graph}
	
	A \emph{Feynman graph} is a graph-like diagram
	\[ \G = \Fgraph\]
	such that $s\colon H \to E$ is injective and $\tau\colon E\to E $ is an involution without fixed points.
	
	A \emph{subgraph} $\H \hookrightarrow \G$ of a Feynman graph $\G$ is a subdiagram that inherits a Feynman graph structure from $\G$.	
	
	The full subcategory on graphs in $\GrShape$ is denoted by $\Grbig$.
\end{defn}

Elements of $V$ are \textit{vertices} of $\G$ and elements of $E$ are called \textit{edges} of $\G$. For each edge $e$, $\tilde e$ is the $\tau$-orbit of $e$, and $\widetilde E$ is the set of $\tau $-orbits in $E$. 
Elements of $H$ are \textit{half-edges} of $\G$. Together with the maps $s$ and $t$, $H$ encodes a partial map $ E \pto V$ describing the incidence for the graph. A half-edge $h\in H$ may also be written as the ordered pair $h = (s(h), t(h))$. 



In general, unless I wish to emphasise a point that is specific to the formalism of Feynman graphs, I will refer to Feynman graphs simply as `graphs'. 

\begin{rmk}\label{geom real} 
	A graph $\G$ may be realised geometrically by a one-dimensional space $|\G|$ obtained from the discrete space $\{ *_v\}_{v \in V}$, and, for each $e \in E$, a copy $[0,\frac{1}{2}]_e$ of the interval $[0,\frac{1}{2}]$ subject to the identifications 
		$0_{s(h)}\sim *_{t(h)}$ for $h \in H$, and 
		$(\frac{1}{2})_{e} \sim (\frac{1}{2})_{\tau e}$ for all $e \in E$. 
		

\end{rmk}

\begin{ex}\label{stick} (See also \cref{fig: stick and corollas}(a).)
	The graph $(\shortmid)$ has edge set $\two = \{1,2\}$ and no vertices.
	\[ \xymatrix{
		(\shortmid)\defeq&\mathbf 2\ar@(lu,ld)[] & \mathbf 0 \ar[l] \ar[r]&\mathbf 0.}\]

	A \emph{stick graph} is a graph that is isomorphic to $(\shortmid)$.

\end{ex}

For any set $ X$, $X^\dagger \cong X$ denotes its formal involution. 
\begin{ex}\label{corolla}(See also \cref{fig: stick and corollas}(b), (c).)
	The \emph{$X$-corolla} $\CX$ associated to a finite set $X$ has the form
	\[ \xymatrix{
		\CX: && *[r] {X \amalg X^\dagger }
		\ar@(lu,ld)[]_{\dagger}&& 
		X^\dagger \ar@{_{(}->}[ll]_-{\text{inc}}\ar[r]& \{*\}.
	}\]
	
\end{ex}

\begin{figure}[htb!]\centering{
		\begin{tikzpicture}
		\node at (-2,2.5) {(a)};
		\draw (0,1)--(0,2);
		\node at (-.3,1.2) { \small{$2$}};
		\node at (-.3, 1.9) {\small{$ 1$}};
		
		\node at (3,2.5) {(b)};
		\draw (5,1)--(5,2);
		\node at (4.7,1.3) { \small{$x^\dagger$}};
		\node at (4.7, 1.8) {\small{$ x$}};
		\filldraw 
		(5,1)circle (2pt);
		
		\node at (8,2.5) {(c)};
		\node at (11,1){\begin{tikzpicture}
			\filldraw(0,0) circle(2pt);
			\foreach \angle in {-60,60,180} 
			{
				\draw (\angle:0cm) -- (\angle:1.2cm);
				
				\node at (-0.2,-.2) {\tiny{$1^\dagger$} };
				\node at (0.3,-.2) {\tiny{$2^\dagger$} };
				\node at (0,.3) {\tiny{$3^\dagger$} };

				\node at (-1,-.2) {\small $ 1$};
				\node at (.8,-.8) {\small$ 2$};
				\node at (.2,1) {\small$ 3$};
				
			}\end{tikzpicture}};
		

	\end{tikzpicture}
	\caption{ Realisations of (a) the stick graph $(\shortmid)$, and the corollas (b) $\CX[\{x\}]$ and (c) $ \C_\mathbf3$.}
	\label{fig: stick and corollas}
}\end{figure}

\begin{defn}\label{defn: inner edges}
An \emph{inner edge} of $\G$ is an element $e \in E$ such that $\{e , \tau e \}\subset im(s)$. 
	The set $\EI \subset E$ of {inner edges} of $\G$ is the maximal subset of $im(s) \subset E$ that is closed under $\tau$, and $\widetilde {\EI}$ is the set of \emph{inner $\tau$-orbits} $\tilde e \in \widetilde E$ such that $e \in \EI$. 

 The set $E_0 = E\setminus im(s)$ is the \emph{boundary} of $\G$. Elements $e \in E_0$ are \emph{ports} of $\G$. 
 
 A \emph{stick component} of a graph $\G$ is a pair $\{e, \tau e\}$ of edges of $\G$ such that $e$ and $\tau e$ are both ports. 
 
\end{defn}
 Graph morphisms preserve inner edges by definition. The stick graph $(\shortmid)$ has $E_0(\shortmid) = E(\shortmid) = \two$, and, for all finite sets $X$, the $X$-corolla $\CX$ has boundary $E_0 (\CX) = X$. 
 

Since $\fin$ admits finite (co)limits, so does $\GrShape$, and these are computed pointwise. And, since $\Grbig$ is full in $\GrShape$, (co)limits in $\Grbig$, when they exist, correspond to (co)limits in $ \GrShape$.

\begin{ex}\label{ex: initial and terminal}

	The empty graph-like diagram $\oslash$ is trivially a graph, and is therefore initial in $ \Grbig$. However, there is no non-trivial involution on a singleton set, so the terminal diagram $\bigstar$ in $\GrShape$ is not a graph. Hence, $\Grbig$ is not closed under finite limits in $\GrShape$. (By Examples \ref{deg loop} and \ref{not cocomplete}, $\Grbig$ is also not closed under finite colimits in $\GrShape$.)
\end{ex}

The cocartesian monoidal structure on $\fin$ is inherited by $\GrShape$ and $\Grbig$, making these into strict symmetric monoidal categories under pointwise disjoint union $\amalg$, and with monoidal unit given by the empty graph $\oslash$.

\begin{ex}\label{ex: M graph}
	
Let $X$ and $ Y$ be finite sets. The graph $\mathcal M^{X,Y}_{x_0,y_0}$, illustrated in \cref{fig: MN}, has two vertices and one inner edge orbit (highlighted in bold-face in \cref{fig: MN}). It is obtained from the disjoint union $\CX[X \amalg \{x_0\}] \amalg \CX[Y \amalg \{y_0\}]$ by identifying 
the $\tau$-orbits of the ports $x_0$ and $y_0$ according to $x_0 \sim \tau y_0, y_0 \sim \tau x_0$. So, 

\[\mathcal M^{X,Y}_{x_0,y_0} = \ \xymatrix{
	*[r] { \left(\small
		{
			(X \amalg Y) \amalg (X \amalg Y)^\dagger 
			\amalg \{x_0,y_0\} }\right)
	}
	\ar@(lu,ld)[]_{\tau} &&&&
{ \left(\small
	{ (X \amalg Y)^\dagger 
		\amalg \{x_0,y_0\} }\right)
}\ar@{_{(}->}[llll]_-{s} \ar[r]^-{t}& \small{\{v_X, v_Y\},}}\] where $s$ is the obvious inclusion, and the involution $\tau$ is described by $x_0 \leftrightarrow y_0$ and $ z \leftrightarrow z^\dagger$ for $z \in X \amalg Y$. The map $t$ is described by $t^{-1}(v_X) = X^\dagger \amalg \{y_0\}$ and $t^{-1}(v_Y) = Y^\dagger\amalg \{x_0\}$. 

In the construction of modular operads, graphs of the form $\mathcal M^{X,Y}_{x_0,y_0}$ are used to encode formal multiplications in graphical species. \end{ex}

\begin{ex}\label{ex: N graph}
Formal contractions in graphical species are encoded by graphs of the form $\mathcal N^X_{x_0,y_0}$ (see \cref{fig: MN}): For $X$ a finite set, the graph $\mathcal N^X_{x_0,y_0}$ is the quotient of the corolla $\CX[X \amalg \{x_0,y_0\}]$ obtained 
by identifying the $\tau$-orbits of the ports $x_0$ and $y_0$ according to $x_0 \sim \tau y_0$ and $y_0 \sim \tau x_0$. It has boundary $E_0 = X$, one inner $\tau$-orbit $\{x_0,y_0\}$ (bold-face in \cref{fig: MN}), and one vertex $v$. So, 
\[\mathcal N^X_{x_0,y_0} = \ \xymatrix{
	*[r] { \left (
		X \amalg X^\dagger 
		\amalg \{x_0,y_0\} 
		\right) }
	\ar@(lu,ld)[]_{\tau} &&&{ \left (
	 X^\dagger 
		\amalg \{x_0,y_0\} 
		\right) }\ar@{_{(}->}[lll]_-{s} \ar[r]^-{t}& \{v\}. }\] 
\end{ex}

\begin{figure}[htb!] 
\begin{tikzpicture}

\node at(0,0){
	\begin{tikzpicture}{
		\node at (6,0) 
		{\begin{tikzpicture}[scale = 0.5]
			\filldraw(0,0) circle(3pt);
			\foreach \angle in {0,120,240} 
			{
				\draw(\angle:0cm) -- (\angle:1.5cm);

			}
			\draw [ ultra thick] (0,0) -- (1.5,0);
			\node at (.5, -0.3) {\tiny $y_0$};
			\node at (2.2, -0.3) {\tiny$ x_0$};

			\end{tikzpicture}};
		
		\node at (7,0){\begin{tikzpicture}[scale = 0.5]
			\filldraw(0,0) circle(3pt);
			\foreach \angle in {-0,90,180,270} 
			{
				\draw(\angle:0cm) -- (\angle:1.5cm);

			}
			\draw [ ultra thick] (0,0) -- (-1.5,0);

			\end{tikzpicture}};

		\node at (6.5,-2){ |$\mathcal M^{X,Y}_{x_0,y_0}$|};
}\end{tikzpicture}};

\node at (5, 0){	\begin{tikzpicture}

\node at (5,0) 
{\begin{tikzpicture}[scale = 0.5]
	\filldraw(0,0) circle(3pt);
	\draw (0,-1.5)--(0,1.5);
	\draw[ultra thick] 
	(0,0)..controls (3,3) and (-3,3)..(0,0);
	\node at (-1,.6) {\tiny $y_0$};
	\node at (1, .6) {\tiny $x_0$};
	\end{tikzpicture}};

\node at (5,-2){ |$\mathcal N^{X}_{x_0,y_0}$|};

\end{tikzpicture}
};
\end{tikzpicture}
\caption{ Realisations of $\mathcal M^{X,Y}_{x_0,y_0}$ and $\mathcal N^{X}_{x_0,y_0}$ for $X \cong \two, \ Y \cong \mm[3]$.}
\label{fig: MN}
\end{figure}

%
%

\label{Hv Ev} 
Let $\G$ be a graph with vertex and edge sets $V$ and $E$ respectively. For each vertex $v$, define $\vH \defeq t^{-1}(v) \subset H$ to be the fibre of $t$ at $v$, and let $\vE\defeq s(\vH) \subset E$.
\begin{defn}\label{defn: valency} 
	
Edges in the set $\vE$ are said to be \emph{incident on $v$}. 

The map $|\cdot|\colon V \to \N$, $v \mapsto |v| \defeq |\vH|$, defines the \emph{valency} of $v$ and $\nV \subset V$ is the set of \emph{$n$-valent} vertices of $\G$. A \emph{bivalent graph} is a graph $\G$ with $V = \nV[2]$. 

A vertex $v$ is \emph{bivalent} if $|v| = 2$. An \emph{isolated vertex} of $\G$ is a vertex $v \in V(\G)$ such that $|v| = 0$. 

\end{defn}

Bivalent and isolated vertices are particularly important in \cref{s. Unital}.

\label{Hn En} Vertex valency also induces an $\N$-grading on the edge set $E$ (and half-edge set $H$) of $\G$: For $n \geq 1$, define $\nH \defeq t^{-1}(\nV)$ and $\nE \defeq s(\nH)$. Since $s(H) = E\setminus E_0 = \coprod_{n \geq 1} \nE$, 
\[E = \coprod_{n \in \N } \nE.\]

\begin{ex}
Recall the stick graph $(\shortmid)$ from \cref{stick}. Since $H(\shortmid)$ is empty, both edges of $(\shortmid)$ are ports: $E(\shortmid) = E_0(\shortmid)$. The corolla $\CX$ (\cref{corolla}) with vertex $*$ has $X \cong \vE[*] = \vH[*] $. If $|X| = n$, then $|*| = n$, so $ V(\C_\nn) = \nV[n]$, and $E = \nE[n] \amalg \nE[0]$. 
\end{ex}

\begin{ex}\label{ex: more M and N} For finite sets $X$ and $Y$, the graph $\mathcal M^{X,Y}_{x_0,y_0}$ (\cref{ex: M graph}) has $ \vE[v_X] = X^\dagger \amalg \{y_0\}$ and $ \vE[v_Y] = Y^\dagger \amalg \{x_0\}$. If $X \cong \nn$ for some $n \in \N$, then $ v_X \in \nV[n+1]$, and $ \vE[v_X] \subset \nE[n+1]$. 

The graph $ \mathcal N^X_{x_0,y_0}$ (\cref{ex: N graph}) has $ \vE = X ^\dagger \amalg \{x_0, y_0\} \cong H$, so $V = \nV[n+2]$ when $ X\cong \nn$. \end{ex}

Since $\Grbig$ is full in the diagram category $\GrShape$, morphisms $f \in \Grbig(\G, \G') $ are commuting diagrams in $\fin$ of the form
\begin{equation}\label{morphism}
\xymatrix{
\G\ar[d]_f&& E \ar@{<->}[rr]^-\tau\ar[d]_{f_E}&&E \ar[d]_{f_E}&& H \ar[ll]_s\ar[d]_{f_H} \ar[rr]^t&& V\ar[d]^{f_V}\\
\G'&& E' \ar@{<->}[rr]_-{\tau'}&&E'&& H' \ar[ll]^{s'} \ar[rr]_{t'}&& V'.}\end{equation}

\begin{lem} \label{all fE}
For any morphism $f = (f_E, f_H,f_V) \in \Grbig(\G, \G')$, the map $f_H$ is completely determined by $f_E$. 
Moreover if $\G$ has no isolated vertices, then $f_E$ also determines $f_V$, and hence $f$.

If $\G$ has no stick components or isolated vertices, then $f$ is completely determined by $f_H$.

\end{lem}
(A directed version of \cref{all fE} appeared as \cite[Proposition~1.1.11]{Koc16}.)

\begin{proof}
The map $f_H \colon H \to H'$ given by $h \mapsto(s')^{-1}f_E s(h) $ is well defined since $s$ is injective. 
 If $\G$ has no isolated vertices, then, for each $v \in V$, $ \vH$ is non-empty and the map 
$ f_V \colon V \to V'$ given by $v \mapsto t'(s')^{-1}f_Es(h)$ does not depend on the choice of $h \in \vH$.

 If $\G$ has no stick components then, for each $e \in E$, there is an $h \in H$ such that $e = s(h)$ or $ e = \tau s(h)$, and the last statement of the lemma follows from the first.\end{proof}

\begin{ex}\label{ex: choose}
For any graph $\G$ with edge set $E$, 
$\Grbig(\shortmid, \G) \cong E$. The morphism $1 \mapsto e \in E$ in $\Grbig(\shortmid, \G)$ -- that \emph{chooses} an edge $e$ -- is denoted $ ch_e$, or $ch^\G_e$. 

\end{ex}

\begin{ex}\label{deg loop} 
	The stick graph $(\shortmid)$ has endomorphisms $ ch_1 = id$ and $ch_2 = \tau$ in $\Grbig$. The coequaliser of $id, \tau\colon (\shortmid) \rightrightarrows (\shortmid)$ in the category $\GrShape$ of graph-like diagrams is the \emph{exceptional loop} $\bigcirc$:
	\[ 
	\bigcirc\defeq \Fgraphvar{\mathbf 1}{\mathbf 0}{\mathbf 0}{}{}{}.\]
	Clearly $\bigcirc $ is not a graph since a singleton set does not admit a non-trivial involution. Hence $\Grbig$ does not admit all finite colimits. 
	This example is the subject of \cref{degenerate}.
	
\end{ex}

\begin{defn}\label{def: locally inj sur}
A morphism $ f \in \Grbig (\G, \G')$ is \emph{locally injective} if, for all $v \in V$, the induced map $f_v\colon\vE \to \sfrac{E'}{f(v)}$ is injective, and \emph{locally surjective} if $f_v\colon\vE \to \sfrac{E'}{f(v)}$ is surjective for all $v \in V$.

Locally bijective morphisms are called \emph{\'etale}. 

\end{defn}

Local bijections are preserved under composition, so \'etale morphisms form a subcategory, $\Gret$ of $\Grbig$. This is the subject of \cref{sec: topology}. 

\begin{ex} \label {morphism ex1}
The following display illustrates the two morphisms $f_a$ and $f_b$ in $\Grbig$ described by the commuting diagrams (a) and (b) below. 
Both morphisms are locally injective, and (b) is surjective, and also locally surjective, hence \'etale. By \cref{all fE}, both $f_a$ and $f_b$ are completely determined by the image of $E(\C_\two)$.

\[
\begin{tikzpicture}
\node at (-7,1.5){(a)};
\node at (-4,0) {		\begin{tikzpicture}[scale = 0.5]

\filldraw (0,0) circle (4pt);
\draw[thick] (0,0)--(0,2)
(0,0)--(-2,1);
\node at (0,-2) {$|\CX[\two]|$};
\node at (-.5, 0){\tiny{$e_1$}};
\node at (.4, 0.3){\tiny{$e_2$}};

\draw [thick]
(5,0)..controls (4.1,-0.5) and (3.8,1.2)..(4,1.5)
(5,0)..controls (5.2,0.8) and (4.8,1.2)..(4,1.5);

\filldraw [gray]
(4,1.5) circle (4pt);
\filldraw (5,0) circle (4pt);
\draw [thick]
(5,0)--(7,0.5);

\node at (5,-2) {$|\G|$};

\draw[->](1.3,0)--(3.3,0);
\node at (2.5,.5){$f_a$};

\end{tikzpicture}};
\node at (1,1.5){(b)};
\node at (4,0) {		\begin{tikzpicture}[scale = 0.5]

\filldraw (11,0) circle (4pt);
\draw[thick] (11,0)--(11,2)
(11,0)--(9,1);
\node at (10.5, 0){\tiny{$e_1$}};
\node at (11.4, 0.3){\tiny{$e_2$}};

\node at (11,-2) {$|\CX[\two]|$};
\draw [thick]
(16,0)..controls (18,2.5) and (14,2.5)..(16,0);

\filldraw(16,0) circle (4pt);
\node at (13.5,0.5){$f_b$};
\draw[->](12.3,0)--(14.3,0);
\node at (16,-2) {$|\W|$};

\end{tikzpicture}};
\end{tikzpicture}
\]

In each example, the horizontal maps are the obvious projections, and the columns in the edge sets represent the orbits of the involution. 


(a) 
\[ \hspace {-.8cm}\small{\xymatrix { 
\CX[\two] \ar[d]_{f_a} 
&
*[r] {\left \{ \begin{array}{ll}
1,&2,\\
1^\dagger, & 2^\dagger \end{array}\right \} }
\ar@(lu,ld)[]_{\dagger}\ar[d]&& 
{\left \{ (1^\dagger, *), (2^\dagger,*) \right \} }\ar[ll]\ar[d]\ar[r]& {\{*\}} \ar[d]^{* \mapsto v_1}\\
\G& 
*[r] { \left \{ 
\begin{array}{lll}
{f_a(1),} & {f_a(2),} &\tau{e_3,} \\ 
{\tau f_a(1)},&{ \tau f_a(2)},&{ \tau e_3}\end{array}
\right \} }
\ar@(lu,ld)[]_{\tau} &&
{\left \{ \begin{array} {lll} (f_a(1), v_2), & (f_a(2), v_2), \\
(\tau f_a(1), v_1), & (\tau f_a(2), v_1), &(e_3, v_1) \end{array} \right \} }\ar[ll]_-{} \ar[r]^-{}& {\{v_1, v_2\}. }}}\] 

(b) 
\[ \small{ \xymatrix {
\CX[\two] \ar[d]_{f_b} &&
*[r] {\left \{ \begin{array}{ll}
		1,&2,\\
		1^\dagger, & 2^\dagger \end{array}\right \} }
\ar@(lu,ld)[]_{\dagger}\ar[d]&& 
{\left \{ (1^\dagger, *), (2^\dagger,*) \right \} }\ar[ll]\ar[d]\ar[rr]&&{\{*\}} \ar[d]\\
\W&& 
*[r] { \{ {f_b(1^\dagger),}
{f_b(2^\dagger)} \} }\ar@(lu,ld)[]_{\tau'} &&
{\left \{ (f_b(1^\dagger), w), (f_b(2^\dagger), w) \right \} }\ar[ll]_{} \ar[rr]^-{}&& \{w\}}}\]

\end{ex}

\begin{ex}\label{ex: locally injective MN}
Recall Examples \ref{ex: M graph} and \ref{ex: N graph}, above. For finite sets $ X$ and $ Y$, the canonical morphisms
$ \xymatrix{ \CX \ar[r] & \mathcal M^{X,Y}_{x_0,y_0} & \CX[Y]\ar[l]} $ and $\xymatrix{\CX \ar[r]& \mathcal N^X_{x_0,y_0}}$ are locally injective, but not locally surjective.

The canonical morphisms $\xymatrix{ \CX[X \amalg \{x_0\}] \ar[r] & \mathcal M^{X,Y}_{x_0,y_0}& \CX[Y\amalg \{y_0\}]\ar[l] }$ are locally injective and locally surjective (hence \'etale), but neither is surjective. However, the canonical morphism $\CX[X \amalg \{x_0, y_0\}] \to \mathcal N^X_{x_0,y_0}$ is \'etale and surjective. (See \cref{morphism ex1}(b) for the case $X = \emptyset$.)

\end{ex}

\begin{ex}\label{ex: embedding fin}
The assignment $X \mapsto \CX$ describes a full embedding of $\fin$ into $\Grbig$. Since $\Grbig(\shortmid, \shortmid) \cong \fisinv(\S, \S) $ canonically, and any morphism in $\Grbig$ with domain $(\shortmid)$ is \'etale, it follows that $\fisinv$ embeds in $\Grbig$ as the subcategory of \'etale morphisms between the corollas $\CX$ ($X \in \fin$), and $(\shortmid)$. 
\end{ex}
\begin{rmk}
\label{rmk: fisinv graphs}
By \cref{ex: embedding fin}, $\fisinv$ will henceforth also be viewed as a subcategory of $\Grbig$. The choice of notation for objects -- $(\shortmid)$ or $\S$, $X$ or $\CX$ -- will depend on the context. The same notation will be used for morphisms in $\fisinv$ and their image in $\Grbig$. So $ch_x \in \fisinv (\S, X)$ may also be written as $ch_x \in \Grbig ((\shortmid), \CX)$, and $f \in \fisinv (X, Y)$ also describes an \'etale morphism $f \in \Grbig(\CX, \CX[Y])$.
\end{rmk}

\begin{ex}\label{ex: first gluing example} 

For all finite sets $X$ and $Y$, the diagram 
\begin{equation}
\label{eq: M coeq}
\xymatrix{\CX[X \amalg \{x_0\}] && \ar[ll]_-{ch_{x_0}}(\shortmid) \ar[rr]^-{ch_{y_0}\circ \tau}&& \CX[Y \amalg \{y_0\}]
}
\end{equation} is in the image of the inclusion $\fisinv \hookrightarrow \Grbig$. It has colimit $\mathcal M^{X,Y}_{x_0,y_0}$ in $\Grbig$. 

The graph $\mathcal N^X_{x_0,y_0}$ is the colimit in $\Grbig$ of the diagram of parallel morphisms 
\begin{equation}
\label{eq: N coeq} ch_{x_0}, \ ch_{y_0}\circ \tau\colon (\shortmid) \rightrightarrows \CX[X \amalg \{x_0, y_0\}]
\end{equation} 
in the image of $\fisinv$ in $\Grbig$. (See also \cref{ex: still more M N} and \cref{fig: glue MN}.)
\end{ex}

As will be shown in \cref{sec: topology}, all graphs may be constructed canonically as colimits of diagrams in the image of $\fisinv \subset \Grbig$. 

\subsection{Connected components of graphs}
%
%
%

A graph is connected if it cannot be written as a disjoint union of non-empty graphs. Precisely:
\begin{defn}\label{conn}\label{component}
A non-empty graph-like diagram $\G$ is \emph{connected} if, 
for each $\mathsf f \in \GrShape(\G, \bigstar \amalg \bigstar)$, the pullback of $f$ along the inclusion $\bigstar \overset {incl_1}{\hookrightarrow} \bigstar \amalg \bigstar$
is either the empty graph-like diagram $\oslash$ or $\G$ itself. A graph $\G$ is connected if it is connected as a graph-like diagram.

A \emph{(connected) component} of a graph $\G$ is a maximal connected subdiagram of $\G$.
\end{defn}

By \cref{defn: graph}, a subdiagram $\H \hookrightarrow \G$ is a subgraph precisely when $E(\H) \subset E$ is closed under $\tau \colon E \to E$. Hence:
\begin{lem}
A connected component of a graph $\G$ inherits a subgraph structure from $\G$. If $\H \hookrightarrow \G$ is a subgraph of $\G$, then so is its complement $\G \setminus \H$. 
\end{lem}
Therefore, every graph is the disjoint union of its connected components.

\begin{rmk}
A graph $\G$ is connected if and only if its realisation $|\G|$ (\cref{geom real}) is a connected space. 
\end{rmk}

\begin{ex}\label{ex: stick examples} 
Following the terminology of \cite{Koc16}, a \emph{shrub} $\mathcal S$ is a graph that is isomorphic to a disjoint union of stick graphs. Hence a shrub $\mathcal S = \mathcal S(J)$ is determined by a set $J$ (of edges) equipped with a fixed-point free involution $\tau_J \colon J \xrightarrow{\cong} J$. 
A morphism in $\Grbig$ whose domain is a shrub is trivially \'etale.

Given any graph $\G$, the shrub $\mathcal S(E)$ determined by $(E, \tau)$ is canonically a subgraph of $\G$. Components of $\mathcal S (E)$ are of the form 
%
%
\[ \xymatrix{
(\shorte)\defeq&\{e, \tau e\}\ar@(lu,ld)[] & \emptyset \ar[l] \ar[r]&\emptyset, } \] for each $\tilde e \in \widetilde E$. 
The inclusion $ \{e, \tau e \} \hookrightarrow E$ induces a subgraph inclusion $\ese\colon (\shorte) \hookrightarrow \G$ called the \textit{essential morphism at $\tilde e$ (for $\G$)}. 

%

Recall from \cref{defn: inner edges} that a stick component of $\G$ is a $\tau$-orbit $\{e, \tau e\}$ in the boundary $E_0$ of $\G$. In particular, $\{e, \tau e\}$ is a stick component of $\G$ if and only if $(\shorte)$ is a connected component of $\G$.

\end{ex}

\begin{ex}\label{ex: vertex examples}
Recall that, for each $v \in V$, $\vE \defeq s(t^{-1}(v))$ is the set of edges incident on $v$. Let $\mathbf{v} = (\vE)^\dagger$ denote its formal involution. Then the corolla $\Cv$ is given by 
\[ \Cv = \qquad \xymatrix{	
*[r] { \left(\vE \amalg ({\vE})^\dagger \right)}\ar@(lu,ld)[] && \vH \ar[ll]_-s \ar[r]^t&{\{v\}}. 
}\] 

The inclusion $\vE \hookrightarrow E$ induces a morphism $\esv^\G$ or $\esv \colon \Cv \to \G$ called the \textit{essential morphism at $v$ for $\G$}. If there exists an edge $e$ such that both $e $ and $\tau e$ are incident on $v$, then $\esv$ is not injective on edges.

If $ \vE$ is empty -- so $\Cv$ is an isolated vertex -- then $ \Cv \hookrightarrow \G$ is a connected component of $\G$.
%
%

\end{ex}

\begin{ex}\label{def Lk} 
	For $k \geq 0$, the \emph{line graph} $\Lk$ (illustrated in \cref{fig. line and wheel}) is the connected bivalent graph with boundary $ E_0 = \{1_{\Lk}, 2_{\Lk}\}$, and 
\begin{itemize}
\item ordered set of edges $E (\Lk)= (l_j)_{j = 0}^{2k+1}$ where $l_0 = 1_{\Lk} \in E_0$ and $l_{2k+1} = 2_{\Lk} \in E_0$, and the involution is given by $\tau (l_{2i}) = l_{2i +1}$, for $0 \leq i \leq k$, 
\item ordered set of $k$ vertices $ V(\Lk) = (v_i)_{i = 1}^k $, such that $ \vE[v_i] = \{ l_{2i -1}, l_{2_i}\}$ for $1 \leq i \leq k$.
\end{itemize}
So, $\Lk$ is described by a diagram of the form 
 $\Fgraphvar{\ \two \amalg 2(\mathbf k)\ }{2(\mathbf k)}{\mathbf k.}{}{}{} $
\end{ex}

\begin{ex}
	The \emph{wheel graph} $\W = \Wm[1]$ with one vertex is the graph 	\begin{equation}\label{wheel}
	\W \defeq \ \Fgraphvar{\{a, \tau a\}}{\{a, \tau a\}}{\{*\}}{}{}{}\end{equation} 
	obtained as the coequaliser in $\Grbig$ of the morphisms $ch_1, ch_2 \circ \tau \colon (\shortmid) \rightrightarrows \C_\two$ (see \cref{morphism ex1}(b)). 

\label{wheels}
More generally, for $m \geq 1$, the {wheel graph} $\Wm$ (illustrated in \cref{fig. line and wheel}) is the connected bivalent graph obtained as the coequaliser in $\Grbig$ of the morphisms $ch_{ 1_{\Lk[m]}}, ch_{2_{\Lk[m]}} \circ \tau \colon (\shortmid) \rightrightarrows \Lk[m]$. So $\Wm$ has empty boundary and
\begin{itemize}
	\item $2m$ cyclically ordered edges $E(\Wl) = (a_j)_{j = 1}^{2m}$, such that the involution satisfies $\tau a _{2m} = a_1$ and $\tau (a_{2i}) = a_{2i +1}$ for $1 \leq i < m$,
\item $ m$ {cyclically} ordered vertices $V(\Wl) = (v_i)_{i = 1}^m$, that $ \vE[v_i] = \{ a_{2i -1}, a_{2i }\}$ for $1 \leq i \leq m$. 

\end{itemize} 
So $\Wm$ is described by a diagram of the form $\Fgraphvar{\ 2(\mathbf m)\ }{2(\mathbf m)}{\mathbf m.}{}{}{}$


\begin{figure}[htb!]
\begin{tikzpicture}
\node at (0,-1.5) {$\mathcal L^4$};
\node at (0,0) { \begin{tikzpicture}[scale = 0.4]
\draw (0,0) -- (10,0);
\filldraw (2,0) circle (3pt);
\filldraw (4,0) circle (3pt);
\filldraw		(6,0) circle (3pt);
\filldraw		(8,0) circle (3pt);
\node at (0,-.7) {\tiny $l_0$};
\node at (1.7,-.7) {\tiny $l_1$};
\node at (10,-.7) {\tiny $l_9$};

\end{tikzpicture}
};
\node at (10,-1.5) {$\mathcal W^4$};
\node at(10,0){\begin{tikzpicture}[scale = 0.4]
\draw (0,0) circle (2cm);
\filldraw (2,0) circle (3pt);
\filldraw (0,2) circle (3pt);
\filldraw		(-2,0) circle (3pt);
\filldraw		(0,-2) circle (3pt);

\node at (-2.4,0.3) {\tiny{$a_8$}};	
\node at (-.6,2.2) {\tiny{$a_1$}};
\end{tikzpicture}};
\node at (5.7,-1.5) {$\W = \Wl[1]$};
\node at(5.7,0){\begin{tikzpicture}
\draw (0,0) circle (15pt);
\filldraw (.53,0) circle (1.5pt);

\end{tikzpicture}};

\end{tikzpicture}

\caption{Line and wheel graphs.}
\label{fig. line and wheel}
\end{figure}

\end{ex}

In \cref{bivalent graphs}, it will be shown that a connected bivalent graph is isomorphic to $\Lk$ or $\Wl$ for some $k \geq 0$ or $m \geq 1$.

\begin{ex}\label{not cocomplete}
The wheel graph $\W$ with one vertex is weakly terminal in $\Grbig$: Since $\widetilde{E}(\W) \cong V(\W) \cong \{*\}$, by \cref{all fE}, morphisms in $\Grbig(\G, \W)$ are in canonical bijection with projections in $\Grbig (\mathcal S (E), \shortmid) $. Hence, for all graphs $\G$, there are precisely $ 2^{ |\tilde E|} \geq 1$ morphisms $\G \to \W$ in $\Grbig$ and 
 every diagram in $\Grbig$ forms a cocone over $\W$. 

In particular, \[\Grbig (\W, \W) = \{id_\W, \tauG[\W]\}\cong \Grbig(\shortmid, \W) \cong \Grbig(\shortmid, \shortmid).\] 

The morphisms $id_\W, \tauG[\W]\colon \W \rightrightarrows \W$ do not admit a coequaliser in $\Grbig$ since their coequaliser in $\GrShape$ is the terminal diagram $\bigstar$, which is not a graph. 
\end{ex}

\cref{not cocomplete} 
leads to another characterisation of connectedness:
\begin{prop}\label{conn 1}\label{conn 2}
The following are equivalent:
\begin{enumerate}
\item 	A graph $\G $ is connected;
\item $\G$ is non-empty and, for every morphism $f \in \Grbig(\G, \W \amalg \W)$, the pullback in $\GrShape$ of $f$ along the inclusion $inc_1\colon \W \hookrightarrow \W \amalg \W$ is either the empty graph $\oslash$ or isomorphic to $\G$ itself;
\item for every finite disjoint union of graphs $\coprod_{i = 1}^k \H_i$, \begin{equation}\label{conn cond} \Grbig(\G, \coprod_{i = 1}^k \H_i) \cong \coprod_{i = 1}^k \Grbig(\G, \H_i).\end{equation}
\end{enumerate}


\end{prop}

\begin{proof}
$(1)\Leftrightarrow (2)$: Since $\W$ is weakly terminal, any morphism $ f \in \GrShape (\G, \bigstar \amalg \bigstar)$ factors as a morphism $\tilde f \in \Grbig (\G, \W \amalg \W)$ 
followed by the componentwise projection $ \W \amalg \W \to \bigstar \amalg \bigstar $ in $\GrShape$. 

$(1)\Rightarrow (3)$: For any finite disjoint union of graphs $\coprod_{i = 1}^k \H_i$, and $1 \leq j \leq k$, let $p_j \in \GrShape(\coprod_{i = 1}^k \H_i , \bigstar \amalg \bigstar)$ be the morphism that projects $\H_j$ onto the first summand, and $\coprod_{i \neq j}\H_i$ onto the second summand. Then, for any graph $\G$ and any $f \in \Grbig(\G, \coprod_{i = 1}^k \H_i)$, the diagram 
\begin{equation}\label{conn proof}\xymatrix{
\mathcal P_j \ar[rr] \ar[d]&& \G \ar[d]^{f}\\
\H_j \ar@{^{(}->}[rr]_{inc_j} \ar[d] && \coprod_{i = 1}^k \H_i \ar[d]^{p_j}\\
\bigstar \ar@{^{(}->}[rr]_{inc_1} && \bigstar \amalg \bigstar}\end{equation}
where the top square is a pullback, commutes in $\GrShape$. Since the lower square is a pullback by construction, so is the outer rectangle.

In particular, if $ \G$ is connected, then $\mathcal P_j$ is either empty or isomorphic to $\G$ itself. But this implies that there is some unique $1 \leq j \leq k$ such that $f$ factors through the inclusion $ inc_j \in \Grbig(\H_j , \coprod_{i = 1}^k \H_k)$. In other words, $ \Grbig (\G, \coprod_{i = 1}^k \H_i) \cong \coprod_{i = 1}^k \Grbig(\G, \H_i)$.

$(3)\Rightarrow (2)$: If $\G$ satisfies condition (3), then $\Grbig(\G, \W \amalg \W) \cong \Grbig(\G, \W) \amalg \Grbig(\G, \W)$. So, taking $\coprod_{i = 1}^k \H_k = \W \amalg \W$ in (\ref{conn proof}), we have $\mathcal P_j = \oslash$ or $\mathcal P_j \cong \G$ for $j = 1,2$. 
\end{proof}

%
%
%

\subsection{Paths and cycles} 
Paths and cycles in a graph $\G$ may be defined using line and wheel graphs (Examples \ref{def Lk} and \ref{wheels}). 


%

\begin{defn}\label{ex. paths and loops 1}\label{ex: paths}

For any graph $\G$, a morphism $p \in \Grbig(\Lk, \G)$ is called a \emph{path of length $k$} in $\G$. 
Given any pair $ x_1, x_2 \in E \amalg V$, $x_1$ and $x_2$ are \emph{connected by a path $ p \in \Grbig(\Lk,\G)$} 
if $\{x_1, x_2\} \subset im(p)$. 

\label{rmk: path connectedness}
A non-empty graph $\G$ is \emph{path connected} if each pair of distinct elements $ x_1, x_2 \in E \amalg V$ is connected by a path in $\G$. 

%


\end{defn}

\begin{ex}
	The isolated vertex $\C_\nul$ is trivially path connected. 
Since $ \Grbig (\Lk, \C_\one)$ is non-empty only when $k = 0 $ or $k = 1$, the unique path $\Lk[1] = \C_\two \to \C_\one$ is the only path that connects the unique vertex $v$ of $\C_\one$ with an edge $e \in E(\C_\one)$. 
\end{ex}

%
%
%

\begin{cor}[Corollary to \cref{conn 1}]\label{cor:path connected}
	A graph $\G$ is connected if and only if it is path connected.

\end{cor}
\begin{proof}
	A morphism $f\colon \G \to \W_1 \amalg \W_2$ that does not factor through an inclusion $\W\hookrightarrow \W_1 \amalg \W_2$ exists if and only if there are distinct $x_1, x_2 \in E \amalg V$ such that $f(x_1) \in \W_1$ and $f(x_2) \in \W_2$. By \cref{conn 1}, since $\Lk$ is connected for all $k$, this is the case if and only if there is no $ p \in \Grbig(\Lk,\G)$ connecting $x_1$ and $ x_2$.
\end{proof}

\begin{lem}\label{lem: injective path}
 Let $\G$ be a connected graph. For any pair $(e_1, e_2)$ of edges of $\G$, there 
  is a locally injective path connecting $e_1$ and $e_2$ in $\G$. 
\end{lem}

\begin{proof}For all edges $e$ of $\G$, $ch_e \colon (\shorte) = \Lk[0] \to \G$ describes an injective path connecting $e$ and $\tau e$.
	
So, let $e_1$ and $e_2 \neq \tau e_1$ be distinct edges of a connected graph $\G$. By \cref{cor:path connected}, there is a path $p \in \Grbig (\Lk, \G)$ connecting $ e_1$ and $e_2$ in $\G$. Moreover, we may assume, without loss of generality, that, for $i = 1,2$, $p (i_{\Lk}) \in \{e_i, \tau e_i\}$: if not, we may replace $p$ with a path $p \circ \iota$ -- where $\iota \colon \Lk[k'] \to \Lk$ ($1\leq k' < k$) is injective -- for which this holds.

If $p$ is not locally injective then there is some $1 \leq j \leq k$, such that 
$p(l_{2j-1}) = p(l_{2j}) \in E (\G)$. 

In this case, if $j = 1$, then $p$ may be replaced by a path $p_{ \hat 1} \colon \Lk[k-1] \to \G$ obtained by precomposing $p$ with the \'etale inclusion $\Lk[k-1] \hookrightarrow \Lk$, $l'_{i} \mapsto l_{i+2}$, $0 \leq i \leq 2k-1$. 

If $1 <j <k$, then 
$p(l_{2j-1}) = p(l_{2j})$ implies that $p(v_{j-1}) = p(v_{j+1})$. Therefore, $p$ may be replaced with a path $p_{ \hat j} \colon \Lk[k-2]$ of length $ k-2$ given by 
 \[p_{ \hat j} (l'_i) = \left \{ \begin{array}{ll}
p(l_i) & \text{for } 0 \leq i \leq 2j -3, \\
p(l_{i + 4}) & \text{for } 2j -2 \leq i \leq 2k -3. \end{array} \right . \] 
Finally, if $ j = k$, then replace $p$ with the path $p_{\hat k} \colon \Lk[k-1] \to \G $ obtained by precomposing $p$ with the inclusion $\Lk[k-1] \hookrightarrow \Lk$, $l'_i \mapsto l_i$, $1 \leq i \leq k -1$.

By iterating this process (always starting with the lowest value of $j$ for which the path $p$ is not injective at $v_j$), we obtain a unique, locally injective path $p_I$ connecting $ e_1$ and $e_2$.
\end{proof}



 


Morphisms from wheel graphs $\Wm$ describe the higher genus structure of graphs (see \cref{path homotopy}).

\begin{defn}

\label{cycle}\label{ex. paths and loops 2}

A \emph{cycle} in $\G$ is a morphism $c \in \Grbig(\Wl, \G)$ for some $m \geq 1$.

A connected graph $\G$ is \emph{simply connected} if it has no locally injective cycles. 

A cycle $c \colon \Wm \to \G$ is \emph{trivial} if there is a simply connected graph $\H$ such that $c$ factors through $\H$. 



\end{defn}
It is straightforward, using the cyclic ordering on the edges of each $\Wm$, to verify that a graph $\G$ is simply connected if and only if its geometric realisation $|\G|$ is.



\begin{ex}\label{ex: flattening}
	For all finite sets $X$, the corolla $\CX$ is trivially simply connected since it has no inner edges and therefore does not admit any cycles.
	
Since the edge sets of the line graphs $\Lk$ are totally ordered for all $k$, there can be no locally injective morphism $\Wm \to \Lk$. Hence $\Lk$ is simply connected. 
However, for all $k \geq 1$, there are morphisms $ \Grbig(\Wl[2k], \Lk[k+1])$ that are surjective on vertices and inner edges. For example, let $V(\Wl[2k]) = (w_i)_{ i =1}^{2k}$ and $V(\Lk[k+1]) = (v_j)_{j =1 }^{k +1 }$ be canonically ordered as in Examples \ref{def Lk}, and \ref{wheels}. Then the assignment $ w_i\mapsto v_i$ for $1 \leq i \leq 2k+1$, and $w_{k+ 1 +j }\mapsto v_{k+ 1 -j }$ for $1 \leq j< k$ induces a morphism $q \colon \Wl[2k] \to \Lk[2k+1]$ that \textit{flattens} $\Wl[2k]$ (\cref{fig: flattening}(a)). This fails to be a local injection at $w_1$ and $w_k$.

	More generally, by flattening $\Wm[2]$ as above, we see that, for any graph $\G$, the set $\EI$ of inner edges of $\G$ is non-empty if and only if $\Grbig (\Wl[2], \G)$ is (see \cref{fig: flattening}(b)).

	\begin{figure}
		[htb!]\begin{tikzpicture}
		\node at (0,0){\begin{tikzpicture}[scale =.45]
		\draw (0,0) circle (2cm);
	\filldraw (2,0) circle (4pt);
	\filldraw (0,2) circle (4pt);
	\filldraw		(-2,0) circle (4pt);
	\filldraw (0,-2) circle (4pt);
	
	\node at (2.5,0){{\scriptsize{ $w_2$}}};
	\node at (0,-2.5){{\scriptsize{ $w_3$}}};
	\node at (-2.6, 0){{\scriptsize{ $w_4$}}};
	\node at (0,2.3){\scriptsize{ $w_1$}};
	\draw [->] (4,0)--(6.7,0);
		\node at (5.3,.5){\small{$q$}};
	\node at (5.3,-.5){\scriptsize{{(flatten)}}};
	\draw (8,-3)--(8,3);
	\draw[ultra thick] (8,-2) --(8,2);
		\filldraw (8,-2) circle (4pt);
	\filldraw (8,0) circle (4pt);
	\filldraw		(8,2) circle (4pt);
		\node at (8.5,-2){{\scriptsize{ $v_3$}}};
	\node at (8.5,0){{\scriptsize{ $v_2$}}};
	\node at (8.5,2){{\scriptsize{ $v_1$}}};
		\end{tikzpicture}};
		\node at (-3,1.5){(a)};
		\node at (9,0){\begin{tikzpicture}[scale =.45]
			\draw (0,.5) circle (1 cm);
				\filldraw (0,1.5) circle (4pt);
			\filldraw (0,-.5) circle (4pt); 
			\draw [->] (2,.5)--(4,.5);
			
			\draw[ultra thick](5,-.5)--(5,1.5);
				\filldraw (5,1.5) circle (4pt);
			\filldraw (5,-.5) circle (4pt); 
			
				\draw [->] (6,.5)--(8,.5);
				\filldraw(10,0) circle(4pt);
					\filldraw(11.5,0) circle(4pt);
			\draw (10,-1.5)--(10,1.5)
			(10,0)--(11.5,0)
			(11.5,0)--(12.5,1.2)
			(11.5,0)--(12.5,-1.2);
			\draw[ultra thick] 
			(10,0)..controls (13,3) and (7,3)..(10,0);
			
				\end{tikzpicture}};
			\node at (5.3,1.5){(b)};
			
		\end{tikzpicture}
	\caption{(a) A morphism $q \colon \Wm[4] \to \Lk[3]$ in $\Grbig$ that flattens $\Wm[4]$.\; (b) If a graph $\G$ has an inner edge, then $\Grbig (\Wm[2], \G)$ is non-empty.} \label{fig: flattening}
	\end{figure}

\end{ex}

\begin{rmk} \label{path homotopy} 
	For any graph $\G$, we may define an equivalence relation 
	of \textit{path homotopy} on paths in $\G$. Two paths in $\G$ are \textit{homotopic} if applying the proof of \cref{lem: injective path} to each leads to the same locally injective path $p_I$ in $\G$. 
	When $\EI \neq \emptyset$, this relation extends to an equivalence relation on cycles in $\G$. If $\G$ is also connected, the set of equivalence classes of cycles has a canonical group structure that is isomorphic to the fundamental group $\pi_i(|\G|)$ of the geometric realisation of $\G$. 
	
	
The fundamental group construction can be extended, using \cref{Gnov construction}, to all graphs $\G $ without isolated vertices. These ideas are not developed in the current work.

%
\end{rmk}

	\section{The \'etale site of graphs}\label{sec: topology}

Recall that $\Gret \subset \Grbig$ is the bijective on objects subcategory of graphs and \'etale morphisms and that, by \cref{ex: embedding fin}, there is a canonical categorical embedding $\fisinv \hookrightarrow \Grbig$ whose image consists of the exceptional graph $(\shortmid)$, the corollas $\CX$, and the \'etale (locally bijective) morphisms between them. 



The goal of this section is to describe $\Gret$ -- and its full subcategory $\Gr$ on the connected graphs --  in detail, and establish the chain 
\[ \xymatrix{ \fisinv \ar@{^{(}->}[r] &\ \Gr \ \ar@{^{(}->}[r]& \GS}\]
	of dense fully faithful categorical embeddings discussed in \cref{sec: Weber}. 


The following is immediate from \cref{def: locally inj sur} and the universal property of pullbacks of sets:
\begin{prop}\label{etale}\label{etale condition}
 A morphism $f \in \Grbig (\G, \G')$ is \emph{\'etale} if and only if the right square in the defining diagram (\ref{morphism})
	is a pullback of finite sets.

\end{prop}



\begin{ex}\label{ex: esv ese etale}
For any graph $\G$ and each edge $e$ of $\G$, the essential morphism $\ese \colon (\shorte)\to \G$ (\cref{ex: stick examples}) is trivially \'etale. For each vertex $v$ of $\G$, the essential morphism $\esv\colon \Cv \to \G$ (\cref{ex: vertex examples}) is also \'etale.
	
	Indeed, a morphism $f \in \Grbig(\G, \G')$ is \'etale if and only if $f$ induces an isomorphism $\Cv \xrightarrow {\cong} \Cv[f(v)]$ for all $v \in V(\G)$.
	
	\end{ex}

%

\begin{ex}\label{wheels and lines} As discussed in \cref{ex: flattening}, there are no \'etale morphism $\Wm \to \Lk$, for any $k \geq 0, m \geq 1$. 
	
All \'etale morphisms between line graphs are pointwise injective, and 
for $k, n \in \N$, 
	\[\Gret(\Lk, \Lk[n]) \cong \left \{
	\begin{array}{ll}
	2(\mathbf{ n-k + 1})& n \geq k\\
\emptyset, & n <k.
	\end{array}
	 \right.\]
	
		For $m \geq 1$, a morphism $f \in \Gret(\Lk, \Wl[m])$ is pointwise injective precisely when $k < m$. For all $k \geq 0$, $f$ is fixed by $f(1_{\Lk}) \in E(\Wl)$. Hence, $\Gret(\Lk, \Wl[m]) \cong E(\Wl) \cong 2(\mm)$.
	

\'Etale morphisms between wheel graphs are surjective and for $l,m \geq 1$ 
	\[\Gret(\Wl[l], \Wl[m]) \cong \left \{
\begin{array}{ll}
2(\mathbf{m})& \text{ if }\frac{l}{m} \in \N\\
\emptyset, & \text{ otherwise.}
\end{array}
\right.\]
\end{ex}

\subsection{Pullbacks and embeddings in $\Grbig$.}\label{ssec embeddings}

As local isomorphisms, \'etale morphisms of graphs have similar properties to local homeomorphisms of topological spaces. 

\begin{lem}\label{lem: Grbig admits pullbacks}\label{lem: Gret admits pullbacks}
	The graph categories $\Grbig$ and $\Gret$ admit pullbacks. Moreover, \'etale morphisms are preserved under pullbacks in $\Grbig$.
\end{lem}
\begin{proof}

	The pullback $\mathcal P = (\underline E, \underline H, \underline V, \underline s, \underline t, \underline \tau)$ of morphisms $f_1 \in \Grbig(\G_1, \G)$ and $f_2 \in \Grbig (\G_2,\G)$ exists in the presheaf category $\GrShape$. Moreover, since pullbacks in $\GrShape$ are computed pointwise, $\underline \tau$ is a fixed-point free involution, and $\underline s$ is injective. So, $\mathcal P$ is a graph, and $\Grbig$ admits pullbacks.	
	
	\'Etale morphisms pull back to \'etale morphisms since limits commute with limits, and therefore, by symmetry, $\Gret$ admits pullbacks.
\end{proof}

\begin{defn}\label{defn: etale preimage}
	For any morphism $f \in \Grbig (\H, \G)$, not necessarily \'etale, and any morphism $w\colon\G' \xrightarrow{}\G$, 
	the \emph{preimage $f^{-1}(\G')\to \G'$ of $\G'$ under $f$} is defined by the pullback
	\[
	\xymatrix{
		f^{-1}(\G') \ar[rr] \ar[d]&& \H\ar[d]^{f}\\
		\G' \ar[rr]_-{w}&& \G.	}\]
	
\end{defn}

In particular, by \cref{lem: Gret admits pullbacks}, if $f \colon \H \to \G$ is \'etale, then so is the preimage $f^{-1}(\G') \to \G$. 

Observe that any (possibly empty) graph $\H$ has the form $\H' \amalg \mathcal S$ where $\H'$ is a graph without stick components and $\mathcal S$ is a shrub.

\begin{defn}\label{def embedding}
	A morphism $f \in \Grbig (\H, \G)$ (with $\H  =  \H'\amalg \mathcal S$ as above) is called an \emph{embedding} if the following three conditions hold:
	\begin{enumerate}[(i)]
		\item the images $f(\H')$ and $f(\mathcal S)$ are disjoint in $\G$;
		\item the restriction of $f$ to $\mathcal S$ is injective;
		\item $f$ is injective on $V(\H)$ and $H(\H)$ (but not necessarily on $E(\H))$.
	\end{enumerate}
	
\end{defn}

This terminology is due to Hackney, Robertson and Yau \cite[Section~1.3]{HRY19a}.

\begin{lem}\label{lem: mono}
	An embedding $f\colon \H {\to} \G$ 
	is either pointwise injective or there exists a pair of ports $e_1, e_2 \in E_0(\H)$ such that 
	\begin{itemize}
		\item $\tauG[\H] e_1, \tau_\H e_2 \in s(H)$, and hence $ e_2 \neq \tau_\H e_1$ (where $\tauG[\H]$ is the involution on $E(\H)$),
		\item $\tauG[\G]f(e_2) = f(e_1) \in {\EI } (\G)$ so $\{f(e_1), f(e_2)\} $ forms a $\tau_\G$-orbit of inner edges of $\G$.
	\end{itemize} 

If $e_1, e_2 \in E_0(\H)$ and $\tauG[\G] f(e_2) = f(e_1) \in \EI(\G)$, then $f$ is said to \emph{glue $e_1$ and $e_2$ in $\G$}.
\end{lem}

\begin{proof}
	Let $f \in \Grbig(\G, \H)$ and assume that $e$ and $e'$ are edges of $\H$ such that $f(e) = f(e')$. If $f$ is an embedding, then either $e$ or $e'$ is a port, since otherwise $e = s(h)$ and $e' = s(h')$, so $f(h) = f(h')$. Moreover, since $f(\tau e) = f(\tau e')$, either $\tau e $ or $\tau e'$ is a port by the same argument.
	
	Assume therefore, that $e$ is a port. 
	If $ \tau e$ is a port, then $e$ and $\tau e$ define a stick component of $\H$ and so $f$ violates either condition (i) or condition (ii) of \cref{def embedding}. 
	
	So, if $f$ is an embedding, then either $e, \tau e' \in E_0$ and $\tau e, e' \in \EI$, or $e, \tau e' \in \EI$ and $\tau e, e' \in E_0$. In particular $f(e) = f(e')$ and $f(\tau e) = f(\tau e')$ are inner edges of $\G$.

\end{proof}


\begin{rmk}\label{rmk monos}
	Monomorphisms in $\Grbig$ are pointwise injective morphisms and hence embeddings. If $f\colon \H {\to} \G$ 
is an embedding such that $e_1, e_2 \in E_0(\H)$ are ports and $f(e_1) =  f(\tau e_2) \in \EI(\G)$, then \[f \circ ch_{e_1}  = f \circ ch_{\tau e_2} \colon (\shortmid) \to \G\] and hence $f$ is not a monomorphism in $\Grbig$.

\end{rmk}
\begin{ex}\label{ex embedding}
	Let $\W$ be the wheel graph with vertex $v \in V(\W)$. The essential morphism $ \esv\colon \Cv {\to} \W$ is a pointwise surjective embedding. In fact, for all $k \geq 1$, the canonical morphism $\Lk \to \Wm[k]$ (Example 3.31) is an epimorphic embedding that is not a monomorphism. 
	
	For all finite sets $X$ and $Y$, the canonical \'etale morphisms $ \CX[X \amalg \{x_0\}] \amalg \CX[Y \amalg \{ y_0\}] \to \mathcal M^{X,Y}_{x_0,y_0}$ and $ \CX[X \amalg \{x_0, y_0\}] \to \mathcal N^X_{x_0,y_0}$ are epimorphic embeddings but not monomorphisms.
\end{ex}

\subsection{Graph neighbourhoods and the essential category $\esG$.}\label{subs. esG}

A family of morphisms $\mathfrak U = \{f_i \in \Gret (\G_i , \G)\}_{i \in I} $ is \emph{jointly surjective} on $\G$ if $\G = \bigcup_{i \in I} im (f_i)$. 
By \cref{lem: Gret admits pullbacks}, $\Gret$ admits pullbacks, and jointly surjective families of \'etale morphisms $\{f_i \in \Gret (\G_i , \G)\}_{i \in I} $ define the {covers at $\G$} for a canonical \emph{\'etale topology} {$J$} on $\Gret$. Sheaves for this topology are those presheaves $P\colon{\Gret}^\mathrm{op}\to \Set$ such that $P(\G) \cong \mathrm{lim}_{f_i \in \mathfrak U }P(\G_i)$ for all graphs $\G$, and all covers $\mathfrak U = \{f_i \in \Gret (\G_i , \G)\}_{i \in I}$ at $\G$. 

As will be shown in \cref{prop: GS sheaves}, the category $\sh{\Gret, J}$ of sheaves for the \'etale site $(\Gret,J)$ is canonically equivalent to the category $\GS$ of graphical species (\cref{defn: graphical species}). 

As motivation for this result, let us first establish more properties of \'etale morphisms. 

\begin{defn}
	A \emph{neighbourhood} of an embedding $w\colon \G' {\to} \G$ is an \'etale embedding $u\colon \mathcal U {\to} \G$ such that $w = u \circ \tilde w\colon \G' {\to} \mathcal U{\to} \G$, for some embedding $\tilde w\colon \G' {\to} \mathcal U$.

	A neighbourhood $ (\mathcal U, u)$ of $ w\colon \G' {\to} \G$ is \emph{minimal} if every other neighbourhood $ (\mathcal U', u')$ of $ w\colon \G' {\to} \G$ is also a neighbourhood of $(\mathcal U, u)$. 
	
\end{defn}

Since vertices $v$ of $\G$ correspond to subgraphs $v\colon \C_\nul \to \G$, and edges $e$ of $\G$ are in bijection with subgraphs $ch_e\colon (\shortmid) \hookrightarrow \G$, we may also refer to neighbourhoods of vertices and edges. Moreover, since $u\colon \mathcal U {\to} \G$ is a neighbourhood of $e \in E$ if and only if it is a neighbourhood of $\ese\colon (\shorte)\to \G$, there is no loss of generality in referring to neighbourhoods of $\tau$-orbits $\tilde e \in \widetilde E$.


Let $\mathcal S (\EI) = \coprod_{\tilde e \in \widetilde{\EI}} (\shorte)$ be the shrub on the inner edges of a graph $ \G$. Given any subgraph $\I \hookrightarrow \mathcal S({\EI})$, there is a graph $\G_{\widehat{\I}}$ and a canonical surjective embedding $i_{\widehat{\I}}\colon\G_{\widehat{\I}}{\to} \G$ (see \cref{fig: graph covers}): 
\begin{equation}
\label{eq: covering inclusion}
\xymatrix{
	\G_{\widehat{\I}}\ar[d]_{i_{\widehat{\I}}}&& {E \amalg (E(\I))^\dagger }\ar@{<->}[rr]^-{\tau_{\widehat\I}}\ar@{->>}[d]&& {E \amalg (E(\I))^\dagger }\ar@{->>}[d]&E \ar@{_{(}->}[l]&& H \ar[ll]_-s\ar@{=}[d] \ar[rr]^t&& V\ar@{=}[d]\\
	\G&& E \ar@{<->}[rr]_-{\tau}&&E&&& H \ar[lll]^{s} \ar[rr]_{t}&& V,}
\end{equation}
where 
\begin{itemize}
	\item $(E(\I))^\dagger$ is the formal involution $e \mapsto e^\dagger$ of the set $E(\I)$ of edges of $\I$ ,
	\item the involution $\tau_{\widehat\I}$ on $E \amalg (E(\I))^\dagger$ is defined by
	\[ e \mapsto \left \{ \begin{array}{ll}
	\tau e, & e \in E \setminus E(\I)\\
	e^\dagger, & e \in E(\I),
	\end{array} \right . \]
	\item the surjection $E \amalg (E(\I))^\dagger \twoheadrightarrow E$ is the identity on $E$ and $e^\dagger \mapsto \tau e$, $e \in E(\I)$.
\end{itemize} 
 So, $\G_{\widehat{\I}}$ has inner edges $\EI (\G_{\widehat{\I}}) = \EI \setminus E(\I)$, and boundary $E_0(\G_{\widehat{\I}}) = E_0 \amalg (E(\I))^\dagger$.

Informally, $	\G_{\widehat{\I}}$ is the graph obtained from $\G$ by `breaking the edges' of $\I$, as in \cref{fig: graph covers}.

\begin{figure}
	\begin{tikzpicture}[scale = .3]

\draw[->] (-15,0)--(-11,0);

\node at (-25,0){\begin{tikzpicture}[scale = .25]
	
	\draw [thick]
		(-3,2).. controls (-2.5,1.5) and (-2.5,0.5).. (-3,0)
	(-3,0).. controls (-3.5,-1).. (-5,-1)
	(-5,3)..controls (-2,4) and (1,4) .. (3,3)
	
		(1,-4)..controls (2, -4) and (6,-2).. (5,2)
	
	;
	\filldraw[white] (-3,0) circle (25 pt)
	(-1,3.56) circle (35 pt)
		(4, -2) circle (35 pt);
	
	\draw[ thick] 	
	(-6,1).. controls (-6,2).. (-5,3)
	(-5,3).. controls (-4.5,3) and (-3.5,3).. (-3,2)
	(-5,-1).. controls (-6,-0) .. (-6,1)
	(-5,-1)-- (-3,2)
	(-3,2)--(-6,1)
	(-6,1)..controls (-7.5,1) and (-9, 0.5) ..(-11,0.5)
	(-1,-5)..controls (-0.5,-3.5) and (0.5, -3.5)..(1,-4)
	(-1,-5)..controls (-0.5,-5) and (0.5, -5)..(1,-4)
	(-1,-5) .. controls (-2.5,-5) and (-5, -4.5) ..(-6, -9)
	(4,1)--(5,2)
	(5,2)--(3,3)
	(3,3)--(4,1)
	(3,3).. controls (4,3.5) and (5,2.5)..(5,2)
	(3,3).. controls (3,4) and(4,6) ..(8,6)
	(-3,2).. controls (1,-1.5) and (5,-1).. (5,2)
	;
	
	\filldraw [black]	
	(-6,1) circle (8pt)
	(-5,3) circle (8pt)
	(-3,2) circle (8pt)
	(-5,-1) circle (8pt)
	(1,-4) circle (8pt)
	(-1,-5) circle (8pt)
	(4,1) circle (8pt)
	(5,2) circle (8pt)
	(3,3) circle (8pt);
	

	(0,6) .. controls (2,6) and (7,4) ..(7,0)
	(7,0) .. controls (7,-2) and (2,-8) ..(0,-8)--(0,0)--(-7.5,0);
	\end{tikzpicture}};

\node at (-2,0){\begin{tikzpicture}[scale = .25]
	
	\draw [thick, red]
	(-3,2).. controls (-2.5,1.5) and (-2.5,0.5).. (-3,0)
	(-3,0).. controls (-3.5,-1).. (-5,-1)
	(-5,3)..controls (-2,4) and (1,4) .. (3,3)
	
	(1,-4)..controls (2, -4) and (6,-2).. (5,2)
	
	;

	\draw[ thick] 	
	(-6,1).. controls (-6,2).. (-5,3)
	(-5,3).. controls (-4.5,3) and (-3.5,3).. (-3,2)
	(-5,-1).. controls (-6,-0) .. (-6,1)
	(-5,-1)-- (-3,2)
	(-3,2)--(-6,1)
	(-6,1)..controls (-7.5,1) and (-9, 0.5) ..(-11,0.5)
	(-1,-5)..controls (-0.5,-3.5) and (0.5, -3.5)..(1,-4)
	(-1,-5)..controls (-0.5,-5) and (0.5, -5)..(1,-4)
	(-1,-5) .. controls (-2.5,-5) and (-5, -4.5) ..(-6, -9)
	(4,1)--(5,2)
	(5,2)--(3,3)
	(3,3)--(4,1)
	(3,3).. controls (4,3.5) and (5,2.5)..(5,2)
	(3,3).. controls (3,4) and(4,6) ..(8,6)
	(-3,2).. controls (1,-1.5) and (5,-1).. (5,2)
	;
	
	\filldraw [black]	
	(-6,1) circle (8pt)
	(-5,3) circle (8pt)
	(-3,2) circle (8pt)
	(-5,-1) circle (8pt)
	(1,-4) circle (8pt)
	(-1,-5) circle (8pt)
	(4,1) circle (8pt)
	(5,2) circle (8pt)
	(3,3) circle (8pt);
	

	(0,6) .. controls (2,6) and (7,4) ..(7,0)
	(7,0) .. controls (7,-2) and (2,-8) ..(0,-8)--(0,0)--(-7.5,0);
	\end{tikzpicture}};

\end{tikzpicture}
\caption{ The morphism $i_{\widehat{\I}}\colon\G_{\widehat{\I}}{\to} \G$ with bivalent subgraph $\I$ indicated in red.}
\label{fig: graph covers}
\end{figure}

For $\tilde e \in \widetilde E(\I)$, the essential morphism $\ese\colon (\shorte) \hookrightarrow \G$ (\cref{ex: stick examples}) factors in two ways through $\G_{\widehat{\I}}$:
\begin{equation}\label{eq: parallel maps}
\xymatrix{ (\shorte) \ar@<4pt>[rrr]^{(e, \tau e) \mapsto (e, e^\dagger)} \ar@<-4pt>[rrr]_{(e, \tau e) \mapsto ((\tau e)^\dagger, \tau e)} &&&\G_{\widehat{\I}} \ar
[rr]^{i_{\widehat{\I}}}&& \G. } \end{equation}
Hence there exist parallel morphisms $\I \rightrightarrows \G_{\widehat{\I}}$, and a coequaliser diagram in $\Grbig$: 
\begin{equation}\label{eq: general coequaliser}
\xymatrix{ \I \ar@<4pt>[rr]\ar@<-4pt>[rr]&&\G_{\widehat{\I}} \ar[rr]^{i_{\widehat{\I}}}&& \G . } \end{equation}
(The choice of morphisms $\I \rightrightarrows \G_{\widehat{\I}}$ in (\ref{eq: general coequaliser}) is not unique -- there are $2^{|\widetilde E(\I)|}$ pairs -- but it is unique up to isomorphism.)

For each $\I \subset \mathcal S(\EI)$, the set of components of $\I \ \amalg \ \G_{\widehat{\I}} $, together with the canonical embeddings to $\G$, define an \'etale cover at $\G$.

The collection $\{(\G_{\widehat{\I}}, i_{\widehat{\I}})\}_{\I \subset \mathcal S(\EI)} \subset \Gret \ov \G$ inherits a poset structure from the poset of subgraphs of $\mathcal S(\EI)$, and 
 the graph $\G_{\widehat{\mathcal{S}(\EI)}}$ with no inner edges is initial in this poset. Moreover, any surjective embedding $\G' {\to} \G$ factors as $\G' \xrightarrow {\cong} \G_{\widehat{\I}}\xrightarrow {i_{\widehat{\I}}}\G$ for some unique $\I \subset \mathcal S(\EI)$. Hence, 
we have proved the following:
\begin{lem}\label{lem: minimal nbd}\label{cor: minimal nbd of v e}
A neighbourhood $(\mathcal U,u)$ of an embedding $w \in \Grbig(\G', \G)$ is minimal if and only if $\EI(\mathcal U) = \EI (\G')$ and the embedding $\G' \to \mathcal U$ induces a surjection on connected components. 

The essential morphisms $ \ese\colon(\shorte){\to} \G$ and $ \esv\colon \Cv {\to} \G$ describe minimal neighbourhoods of each edge $e$ and vertex $v$ of $\G$. 

\end{lem}

When $\G$ has no stick components, one readily checks that $\G_{\widehat{\mathcal{S}(\EI)}} = \coprod_{v \in V} \Cv.$ In particular, for any graph $\G$, there is a canonical choice of \textit{essential cover} $\mathfrak {Es}_\G $ at $\G$ by the essential morphisms $\ese$ and $\esv$. 

\begin{defn}
	\label{def: essential}
	Let $\G$ be a graph. The \emph{essential category $\esG$ of $\G$} is the full subcategory of $\Gret \ov \G$ on the essential embeddings $ \ese\colon(\shorte){\to} \G, \ \tilde e \in \widetilde E,$ and $ \esv\colon \Cv {\to} \G, \ v \in V$.	
\end{defn}

By definition, $\esG$ has no non-trivial isomorphisms. Hence, 
there is a canonical bijection $h = (e,v) \leftrightarrow \left(\esh \colon \ese \to \esv\right)$ between half-edges $h$ of $\G$, and non-identity morphisms $\delta$ in $\esG$. 

\begin{lem}\label{lem: essential cover}
	Each graph $\G$ is canonically the colimit of the forgetful functor $\esG \to \Gret$. 
	
	A presheaf $P \in \pr{\Gret}$ is a sheaf for the \'etale site $(\Gret, J)$ if and only if for all $\G$,
	\begin{equation}\label{eq: sheaf condition essential}
	P (\G) \cong \mathrm{lim}_{ (\C,b) \in \esG}P (\C).
	\end{equation}
\end{lem}

\begin{proof}
If $e \in E_0$ is a port of $\G$, then there is at most one non-trivial morphism $\esh = \esh[(\tau e,v)]$ with domain $\shorte$ in $\esG$. In this case, 
$\Cv$ is the colimit of the diagram $\shorte \xrightarrow{\esh} \Cv$. The first statement then follows from (\ref{eq: parallel maps}) and (\ref{eq: general coequaliser}). The second statement is immediate since, by \cref{lem: minimal nbd}, the essential cover $\mathfrak {Es}_\G $ refines every \'etale cover $\mathfrak U $ of $\G$.
\end{proof}


%


\subsection{Boundary-preserving \'etale morphisms}
\label{subs. covering}
	
	In general, morphisms $ f \in \Gret(\G', \G)$ do not satisfy $f(E'_0) \subset E_0$. Those that do are componentwise surjective \textit{graphical covering morphisms} in the sense of \cref{graph cover} below. In particular, embeddings $f \in \Gret(\G', \G)$ such that $f(E'_0) = E_0$ are componentwise isomorphisms.

	\begin{prop}\label{graph cover}
		For any \'etale morphism $ f \in \Gret(\G', \G) $, $f (E'_0) \subset E_0$ if and only if there exists an \'etale cover $\mathfrak U = \{ \mathcal U_i, u_i\}_{i \in I}$ of $\G$, such that, for all $i$, $f^{-1}(\mathcal U_i)$ is isomorphic to a disjoint union of $ k(\mathcal U_i, f) \in \N$ copies of $ \mathcal U_i$. 
		
		In this case, $ k(\mathcal U_i,f) = k_f \in \N$ is constant on connected components of $\G$. 
	\end{prop}

	\begin{proof}
	If $u\colon \mathcal U {\to} \G$ is an \'etale embedding for which there is a $k \in \N$ such that $f^{-1}(\mathcal U) \cong k (\mathcal U)$, then also $f^{-1}(\mathcal V) \cong k( \mathcal V)$ for all embeddings $\mathcal V {\to} \mathcal U$. So, we may assume, without loss of generality, that $\mathfrak U = \mathfrak{Es}_\G$ is the essential cover of $\G$. 
		
	Observe first that, if $f(E'_0) \not \subset E_0$, there exists a vertex $v $ of $\G$ and a port $e' $ of $\G'$ such that $f(e') \in \vE$. Hence $(\shorte[\tilde e']) \not \cong \Cv$ is a connected component of $ f^{-1}(\Cv)$.

%
		For the converse, let $v \in V$ be a vertex of $\G$. 
		Since $f$ is \'etale, $\Cv \cong \Cv[v']$ for all $v' \in V'$ such that $f(v') = v$. By the universal property of pullbacks, the canonical embedding $\coprod_{ v'\colon f(v') = v}\Cv {\to} \G'$ factors through $f^{-1}(\Cv)\to \G'$, and therefore
		\[ f^{-1}(\Cv) \cong \left (\coprod_{v'\colon f(v') = v} \Cv[v'] \right) \amalg \mathcal S, \ \text{ for some shrub } \mathcal S. \] 
By construction, a connected component of $\mathcal S$ must be of the form $\ese[\tilde e']\colon (\shorte[\tilde e']) \to \G'$ for some port $e'$ of $\G'$ satisfying $f(e') = e \in \vE$. But $f(E'_0) \subset E_0$ by assumption, so there is no such port. Hence $\mathcal S$ is the empty graph, and so $f^{-1}(\Cv) \cong \coprod_{v'\colon f(v') = v} \Cv$, whereby $k(\Cv,f) = |f^{-1}(v)| \in \N$.
	
	 It is immediate that $f^{-1}(\shorte) \cong \coprod_{e' \in E', f(e') = e } (\shorte[\tilde e'])$ for all $e \in E$,. So $k(\shorte,f) = |f^{-1}(e)| \in \N$, and the first statement of the proposition is proved.

By condition (3) of \cref{conn 1}, it is sufficient to verify the second part of the proposition componentwise on $\G$. Therefore, let $f \in \Gret(\G', \G)$ satisfy $f(E_0') \subset E_0$ , and assume, without loss of generality, that $\G$ is connected. 


If $\G \cong (\shortmid)$ is a stick, there is nothing to prove. Otherwise, for any half-edge $h = (e,v)$ of $\G$, if $e' \in f^{-1}(e)$, then $e' = s'(h')$ for some half-edge $h' = (e',v') \in f^{-1}(h)$ of $\G'$. Hence, the following diagram commutes:
	\[
\xymatrix{
	\coprod_{e'\in f^{-1}(e)}(\shorte) \ar[rr] ^- {\coprod_{e' } \esh[(e',v')]} \ar[d] && \coprod_{v' \in f^{-1}(v)} \Cv[v'] \ar[rr] ^- {\coprod_{v'} \esv[v']} \ar[d]&& \G' \ar[d]^f
	\\
	(\shorte) \ar[rr]_-{\esh[h]} && \Cv[v] \ar[rr]_-{\esv[v]}&& \G.}\]

The first part of the proof implies that both squares are pullbacks. So, if $ f^{-1} (\Cv)$ is isomorphic to $k_v = k(\Cv, \esv)$ copies of $\Cv$, then $f^{-1}(\shorte) \cong k_v (\shorte) $ for all $e \in \vE$. Hence $\Cv \mapsto k_v$ extends to a functor $k_\G$ from $ \esG$ to the discrete category $\N$. Since $\G$ is connected, so is $\esG$, and therefore $k_\G$ is constant. 
	\end{proof}

\begin{defn}\label{pp}
	A morphism $ f \in \Gret (\G', \G)$ is called \emph{boundary-preserving} if it restricts to an isomorphism $f_{E_0}\colon E_0 \xrightarrow \cong E'_0$. 
\end{defn}
The following is an immediate corollary of \cref{graph cover}.
\begin{cor}
	\label{cor: coverings}\label{pp iso}
If $\G$ is connected, and $\G'$ is non-empty, then an \'etale morphism $f \in \Gret(\G', \G)$ such that $f(E'_0)\subset E_0$ is surjective. If $ \G'$ is also connected and its boundary $E'_0 $ is non-empty, then $f $ is boundary-preserving if and only if it is an isomorphism.

\end{cor}

	\begin{rmk}
		The condition that $ E'_0$ is non-empty is necessary in the statement of \cref{pp iso}. For example, for any $ m > 1$, each of the two \'etale morphisms $\Wl \to \W$ (\cref{wheels and lines}) is trivially boundary-preserving, but certainly not an isomorphism. 
		
	\end{rmk}

 Recall from \cref{def Lk} that the line graph $\Lk$ has totally ordered edge set $E(\Lk) = (l_j)_{j = 0}^{2k+1}$ with ports $l_0 = 1_{\Lk}$ and $l_{2k+1} = 2_{\Lk}$. For each vertex $w_i \in V(\Lk)$, $ \vE[w_i]= \{ l_{2i-1}, l_{2i}\}$. 
 
	\begin{prop}\label{bivalent graphs}
		Let $ \G$ be a connected graph with only bivalent vertices. Then $ \G = \Lk$ or $\G = \Wl$ for some $k \geq 0$ or $m \geq 1$.
	\end{prop}
	
	\begin{proof} 
Since $\G$ is bivalent, every embedding $\Lk {\to} \G$ from a line graph is \'etale. 
			
		The result holds trivially if $\G \cong \Lk[0]$ is a stick graph. Otherwise, if $V = V_2$ is non-empty, then, for each $v \in V$, a choice of isomorphism $\Lk[1] \xrightarrow {\cong } \Cv$ describes an embedding $\Lk[1] {\to} \G$. Since $\G$ is finite, there is a maximum $ M \geq 1$ such that there exists an embedding $f\colon \Lk[M] {\to} \G$.

				Let $f \in \Gr(\Lk[M], \G)$ be such a map. 	By \cref{lem: mono}, $f$ is either injective on edges, 
				or $f(2_{\Lk[M]}) = \tau f(1_{\Lk[M]}) \subset \EI$. 
%
%
		Let $e_1 = f(1_{\Lk[M]}), $ and 
			$e_2 =	f(2_{\Lk[M]})$. If $e_j $ is not a port of $\G$ for some $j = 1,2$, then $e_j \in \vE$ for some vertex $v \in V_2$. 
			
			If $f$ is injective on edges, then $v$ is not in the image of $f$. But then, since $v$ is bivalent, this means that $f$ factors through an embedding $\Lk[M] {\to} \Lk[M+1]$, contradicting maximality of $ M$. 
			Therefore, $\{e_1, e_2 \} \subset E_0$ so $f$ is surjective and boundary-preserving, whence $\Lk[M]  {\cong} \G $ by \cref{cor: coverings}.
	
			Otherwise, if $f$ is not injective on edges, then, by \cref{lem: mono}, it must be the case that $f(l_0) = f(l_{2M})$. Therefore, $f$ factors through a cycle $ f \colon \Lk[M] \to \Wl[M] \xrightarrow {\tilde f} \G$. Since $f$ is an \'etale embedding, so is $\tilde f \colon \Wl[M] \to \G$. Hence, $f$ is boundary-preserving and $\G \cong \Wl[M]$ by \cref{cor: coverings}. 
	\end{proof}

%

\'Etale morphisms of simply connected graphs are 
either subgraph inclusions or isomorphisms. (This is why the combinatorics of cyclic operads are much simpler than those of modular operads.) 

\begin{cor}
	\label{cor:simp conn etale}	\label{cor:simp conn iso}
	Let $\G$ be simply connected. If $f \in \Grbig(\G', \G)$ is locally injective, then $f$ is pointwise injective on connected components of $\G'$. Hence, if $\G'$ is connected, it is simply connected.
	
It follows that any \'etale morphism of simply connected graphs is a pointwise injection.

\end{cor}

\begin{proof} We may assume, without loss of generality, that $\G'$ and $\G$ are connected. Since the result holds trivially when either $\G'$ or $\G$ is an isolated vertex, assume further that both graphs have non-empty edge sets. 
	
Let $f\colon\G'\to \G$ be a local injection. 
For any locally injective path $p \colon \Lk \to \G'$, the path $f \circ p \colon \Lk \to \G$ is locally injective in $\G$. If $f \circ p$ is not pointwise injective, then either 
$f \circ p$ factors through a locally injective cycle in $\G$ -- and hence $\G$ is not simply connected -- or there are $1 \leq i <j \leq k$ such that $f \circ p (v_i) = f\circ p (v_j) \in V(\G)$. 

So, let $1 \leq i <j \leq k$ be such that $f \circ p (v_i) = f\circ p (v_j) \in V(\G)$. We may assume, moreover, that if $i \leq i' < j' \leq j$ also satisfy $f \circ p(v_{i'})  = f\circ p (v_{j'})$, then $(i', j') = (i,j)$. 

Let $L = j-i $. Then there is a cycle $c \colon \Wm[L] \to \G$ described by $c(a_{2L-1}) = f \circ p (l_{2j-1})$, $c(a_{2L}) = f \circ p (l_{2i})$, and 
$c(a_{2s}) = f \circ p(l_{2(s+i)})$ (and hence $c(a_{2s -1}) = f \circ p(l_{2(s+i)-1})$) for $1 \leq s <L$. In particular, for $1 \leq s <L$, $c$ is injective at the vertex $w_s$ since $ f \circ p$ is locally injective. And 
$c(a_{2L}) = f \circ p (l_{2i}) \neq f \circ p (l_{2j-1}) = c (a_{2L-1})$ since $f \circ p$ is locally injective, so $c$ is injective at $w_L$. Therefore $c$ is locally injective, and $\G$ is not simply connected.

Hence, if $f \colon \G' \to \G$ is a local injection from a connected graph $\G'$ to a simply connected graph $\G$, then $f$ is pointwise injective, so $\G'$ is also simply connected. 
 
The final statement is immediate since \'etale morphisms are locally injective by definition. 
\end{proof}

\cref{prop: dag prop} gives analogous results for directed acyclic graphs (\cref{def: DAG}).

\subsection{\'Etale sheaves on $\Gret$}\label{subs. sheaves}

Recall that graphical species are presheaves on the category $\fisinv$, and that there is a full inclusion $\Phi \colon \fisinv \hookrightarrow \Gret$. 
%
To prove that $\Phi$ induces an equivalence $\GS \simeq \sh{\Gret,J}$ between graphical species and sheaves for the \'etale topology on $\Gr$, first observe:

\begin{lem}
	\label{lem: fisinv dense}\label{equiv}
	The inclusion $\Phi\colon \fisinv \hookrightarrow \Gret$ is dense. 
\end{lem}

\begin{proof}
	\label{essentially surjective}
	
	It is easy to check that any connected graph without inner edges is isomorphic to $\S$ or $\CX$ for some finite set $X$, and therefore the essential image $im^{es}(\Phi)$ of $\fisinv$ in $\Gret$ is the full subcategory of connected graphs with no inner edges. Moreover, it follows immediately from the definition of $\esG$ (\cref{def: essential}) that the canonical inclusion $\esG \hookrightarrow im^{es}(\Phi)\ov \G$ is full and essentially surjective on objects, and hence an equivalence of categories. 
	
	Therefore $\esG \simeq \fisinv \ov \G$, and the lemma follows from \cref{lem: essential cover}.
\end{proof}

In particular, $\Gret$ is a full subcategory of $\GS$ under the induced nerve functor $\yet \defeq N_{\Gret}\colon \Gret \to \GS$, and I will write $\G$, rather than $\yet \G$, where there is no risk of confusion. The category 
$\elG[\yet \G] = \fisinv \ov \G $, whose objects are \textit{elements of $\G$}, will be denoted by $\elG$. 



Let $J_{\mathsf C}$ be the restriction to $\Gr$ of the topology $J$ on $\Gret$. 

\begin{prop}\label{prop: GS sheaves}
	There is a canonical equivalence of categories
	$\sh{\Gret,J} \simeq \GS$, and hence also an equivalence  $\sh{\Gr,J_{\mathsf C}} \simeq \GS$. 
	\end{prop}

\begin{proof}
	This is straightforward from the definitions and \cref{equiv}. Namely, the inclusion $\Phi\colon \fisinv \to \Gret$ induces an \emph{essential geometric morphism} between the presheaf categories $\pr{\fisinv} = \GS$ and $\pr{\Gret}$. The right adjoint 
	to the pullback $\Phi^* \colon \pr{\Gret} \to \GS$ is given by
	\begin{equation}\label{right adjoint}\Phi_*\colon \GS = \pr{\fisinv} \to \pr{\Gret}, \quad S \longmapsto (\G \mapsto \mathrm{lim}_{(\C,b) \in \elG} S(\C)). \end{equation} 
	Since $\Phi$ is fully faithful, so is $\Phi_*$ (e.g.\ by \cite[Section~VII.2]{MM94}).
	
By Lemmas \ref{lem: essential cover} and \ref{equiv}, a presheaf $P$ on $\Gret$ is a sheaf for the canonical \'etale topology $J$ on $\Gret$ if and only if, for all graphs $\G$, 
	\begin{equation}\label{eq. sheaf ef} P(\G) \cong \mathrm{lim}_{(\C,b) \in \elG}P(\C).\end{equation}
	Hence $\sh{\Gret,J} \simeq \GS$. Moreover, for all $J$-sheaves $P$, and all graphs $\G$, $P(\G)$ is computed componentwise on $\G$, whence $\sh{\Gr,J_{\mathsf C}} \simeq \GS$ and the proposition is proved.	
 	\end{proof}
	 	
I will use the same notation to denote a graphical species $S$ and the corresponding sheaf on $(\Gret, J)$. So, for any graph $\G$, $S(\G) \defeq \mathrm{lim}_{{(\C,b) \in \elG}}S(\C).$

\begin{defn}\label{S-graph} 
	
An \emph{$S$-structured graph} $(\G, \alpha)$ 
 is a graph $\G$ together with an element $\alpha \in S(\G)$ (or $\alpha \in \GS (\G, S)$). The category of $S$-structured graphs is denoted by $\ovP{S}{\Gret}$, and $\ovP{S}{\Gr}$ is the subcategory of connected $S$-structured graphs.

%
%
\end{defn}

\subsection{Directed graphs}\label{subs ex: Directed graphs} 
By way of example, and to provide extra context, this section ends with a discussion of directed graphs.

Let $Di$ be the terminal directed graphical species from Examples \ref{Comm c} and \ref{ex:direction}. 
For any graph $\G$, a $Di$-structure $ \xi\in Di(\G)$ is precisely a partition $E = E_{\In} \amalg E_{\Out}$, where $e \in E_{\In}$ if and only if $\tau e \in E_{\Out}$. So, $\tau$ induces bijections $E_{\In} \cong \widetilde E \cong E_{\Out}$, and an object $ (\G,\xi)$ of $\Grets[Di]$ -- called an \textit{orientation on $\G$} -- is given by a diagram of finite sets
 \begin{equation}\label{Directed graph}
 \xymatrix{
 	\widetilde E&& H_{\In} \ar[ll]_-{\widetilde {s_{\In}}} \ar[rr]^-{t_\In} && V && H_{\Out} \ar[ll]_-{t_\Out}\ar[rr]^-{\widetilde {s_{\Out}}}&& \widetilde E},\end{equation}
	where the maps $\widetilde {s_\In}, \widetilde {s_\Out}$, and $t_\In, t_\Out$ denote the appropriate (quotients of) restrictions of $s\colon H \to E$, respectively $t\colon H \to V$.
	Then morphisms in $ \Grets[Di]$ are quadruples of finite set maps making the obvious diagrams commute, and such that the outer left and right squares are pullbacks. In particular, $ \Grets[Di]$ is the 
	\textit{category of directed graphs and \'etale morphisms} used in \cite[Section~1.5]{Koc16} to prove a nerve theorem for properads in the style of \cite{BMW12}.
	
\begin{ex}\label{ex: directed lines and wheels}
	The line graphs $\Lk$ with $E(\Lk) = \{l_i\}_{i = 0}^{2k+1}$ admit a distinguished choice of orientation $\theta_{\Lk} \in Di(\Lk)$ given by
	\[\theta_{\Lk}\colon E(\Lk) \to \{ \In, \Out\}, \ l_{2i } \mapsto (\In) \text{ and } l_{2i+1} \mapsto (\Out) \text{ for } 0 \leq i \leq k.\] For $m \geq 1$, the canonical morphism $\Lk[m] \to \Wm$ induces an orientation $\theta_{\Wm}$ (with $a_{2j}\mapsto (\In)$) on the wheel graph $\Wm$. 
	
%
\end{ex}

\begin{defn}\label{def: DAG}
	A \emph{directed path of length $k$ in $ (\G, \xi)$} is a path $p \colon \Lk \to \G$ in $\G$ such that, for all $l \in E(\Lk)$, \[Di(ch_l) (\theta_{\Lk}) = Di(ch_{p(l)}) (\xi) \in \{\In, \Out\}.\]
	A \emph{directed cycle of length $m$ in $(\G, \xi)$} is a cycle $c \colon \Wm \to \G$ in $\G$ such that the induced morphism $ \Lk[m] \to \Wm \to \G$ is a directed path. 
%
%

	A \emph{directed acyclic graph (DAG)} is a directed graph $(\G, \xi)$ without directed cycles. 
\end{defn}

It follows immediately from the definitions that 
any directed path or cycle in a directed graph $(\G, \xi)$ is locally injective. Hence, if $(\G, \xi)$ admits a directed cycle, $\G$ is not simply connected. The converse is not true.

The following directed version of \cref{cor:simp conn etale} is not necessary for the constructions of this paper, so I leave its proof as an exercise for the interested reader: 

\begin{prop}\label{prop: dag prop} 
	

For all \'etale morphisms $ f \colon (\G', \xi') \to (\G, \xi)$ 
between connected DAGs, the underlying morphism $f \colon \G' \to \G$ is an \'etale embedding. 

Moreover, if $(\G, \xi)$ is a DAG and the set of morphisms $(\G', \xi') \to(\G, \xi) $ in $ \Grets[Di]$ is non-empty, then $(\G',\xi')$ is a DAG. Hence, any morphism to a DAG in $\Grets[Di]$ is pointwise injective on connected components.
\end{prop}

A consequence of \cref{prop: dag prop} is that the combinatorics of properads, which are governed by DAGs, are much simpler than those of wheeled properads or modular operads.


	 \section{Non-unital modular operads}\label{a free monad}\label{sec: non-unital}\label{nonunital csm section}
 
The goal of the current section is to construct a monad $\TT = (T, \mu^\TT , \eta^\TT)$ on $\GS$ whose EM category of algebras $ \GS^{\TT}$ is isomorphic to the category $\nuCSM$ of non-unital modular operads (\cref{rmk:non-unital}). 


To provide context for this section, consider the following example:
\begin{ex}\label{ex: operad endo} 
	Recall, from \cref{ex: operad}, the category $\Bifiso$, whose objects are finite sets $X$, viewed as rooted corollas $t_X$, and the directed exceptional edge $(\downarrow)$. 

	The operad endofunctor $\Mop$ on $\pr{\Bifiso}$ from \cref{ex: dendroidal} is described in detail in \cite[Section~3]{BM07}. It takes a presheaf $O\colon {{\Bifiso}^{\mathrm{op}}}\to \Set$ to the presheaf $\Mop O$ on $\Bifiso$ with $\Mop O(\downarrow) =O(\downarrow)$, and such that elements of each $\Mop O(t_X)$ are {formal operadic compositions (i.e.~root-to-leaf graftings of decorated corollas as in \cref{fig:grafting}(b)) of elements of $O$}. 
	In other words, they are represented by rooted trees $\Tr \in \Omega$, whose leaves are bijectively labelled by $X$, together with a decoration of the vertices of $\Tr$ by elements of $O$ (according to valency), that also 
	determines a colouring of edges of $\Tr$ by $O(\downarrow)$. 
	
	The monadic unit $\eta^{\MMop}$ is induced by the inclusion of rooted corollas, or trees with one vertex, in $\Omega$. So, $\eta^{\MMop}(\phi) = (t_X, \phi)$ for all $\phi \in O(t_X)$ (\cref{fig: operad monad}, left side). Applying the monad twice describes a nesting of $O$-decorated trees, and the multiplication 
	$\mu^{\MMop}$ for $\Mop$ is induced by erasing the inner nesting (the blue circles in the right hand side of \cref{fig: operad monad}). 
%
%

\begin{figure}[htb!]
	\begin{tikzpicture}
\node at (0,0)
{\begin{tikzpicture}[scale = .6]
	
\node at (-4,2){\small{$O(X)\ni \phi$}}; 
	\draw[thick]
	(0,0) --(0,2)
	(0,2)--(-1,4)
	(0,2)--(-.5,4)
	(0,2)--(1,4);
	
	\draw[|->,]
	(-2.5,2)--(-1,2);
	\node at (-1.5,2.6){\scriptsize{$\eta^{\MMop}$}};
	
	\draw[dashed]
	(-.3,4)--(.8,4);
	
	\draw[decoration={brace, raise=5pt},decorate]
	(-1,4) -- node[above=6pt] {\small{$X$}} (1,4);
	
	\draw [ draw=red, fill=white]
	(0,2)circle (15pt);
	
		\node at(0,2) {\small{$\phi $}};
		\node [anchor = west]at (.7,2){\small{$\in \Mop O(X) $}};
	
	\end{tikzpicture}};
%

	\node at (6,0){	\begin{tikzpicture}[scale = 0.25]

	{	\draw[blue, fill = blue!20]
		(0, 0) circle (50pt)
		(0,5) circle(50pt)
		(-3.5,4) circle(50pt)
		(4.5, 4) circle(50pt);

	{	\draw[thick, gray!80]
		(0,0)--(0,-3)
		(0,0)--(-3.5,3.5)
		(0,0)--(0,5)
		(0,0)--	(4.5,3.5)
		(-3.5,3.5)--(-7,7)
		(-3.5,3.5)--(-3,4.2)
		(-3,4.2)--(-5,7)
		(-3,4.2)--(-3,7)
		(4.5,3.5)--(1,7)
		(4,4)--(4,6)
		(4,6)--(4,7)
		(4.5,3.5)--(5,4.2)
		(5,4.2)--(5,7)
		(5,4.2)--(6,7)
		(5,4.2)--(7.5,7);}
	
		\clip (0, 0) circle (40pt)
		(0,5) circle(40pt)
		(-3.5,4) circle(40pt)
		(4.5, 4) circle(40pt);
		
		\draw[ thick, blue]
		(0,0)--(0,-3)
		(0,0)--(-3.5,3.5)
		(0,0)--(0,5)
		(0,0)--	(4.5,3.5)
		(-3.5,3.5)--(-7,7)
		(-3.5,3.5)--(-3,4.2)
		(-3,4.2)--(-5,7)
		(-3,4.2)--(-3,7)
		(4.5,3.5)--(1,7)
		(4,4)--(4,6)
		(4,6)--(4,7)
		(4.5,3.5)--(5,4.2)
		(5,4.2)--(5,7)
		(5,4.2)--(6,7)
		(5,4.2)--(7.5,7);

		\filldraw[ red]
		(0,0) circle (6pt)
		(-3.5,3.5) circle (6pt)
		(-3,4.2) circle (6pt)
		(0,5) circle (6pt)
		(0,4.2) circle (6pt)
		(4.5,3.5) circle (6pt)
		(4,4) circle (6pt)
		(4,4.7) circle (6pt)
		(5,4.2) circle (6pt);


	}
		\end{tikzpicture}};
	
	\node at (8.8,.3){\scriptsize{$\mu^{\MMop}$}};
	\draw[->]
	(8.2,0)--(9.5,0);
	
	\node at (11, 0){	\begin{tikzpicture}[scale = 0.25]
		
		\draw[ thick]
	(0,0)--(0,-3)
	(0,0)--(-3.5,3.5)
	(0,0)--(0,5)
	(0,0)--	(4.5,3.5)
	(-3.5,3.5)--(-7,7)
	(-3.5,3.5)--(-3,4.2)
	(-3,4.2)--(-5,7)
	(-3,4.2)--(-3,7)
	(4.5,3.5)--(1,7)
	(4,4)--(4,6)
	(4,6)--(4,7)
	(4.5,3.5)--(5,4.2)
	(5,4.2)--(5,7)
	(5,4.2)--(6,7)
	(5,4.2)--(7.5,7);

	\filldraw[ red]
	(0,0) circle (8pt)
	(-3.5,3.5) circle (8pt)
	(-3,4.2) circle (8pt)
	(0,5) circle (8pt)
	(0,4.2) circle (8pt)
	(4.5,3.5) circle (8pt)
	(4,4) circle (8pt)
	(4,4.7) circle (8pt)
	(5,4.2) circle (8pt);

		\end{tikzpicture}};

\end{tikzpicture}
\caption{Visualising the unit and multiplication for the operad monad on rooted corollas.}
\label{fig: operad monad}
\end{figure}

If $(O,h)$ is an algebra for $\MMop$, then $h$ describes a rule for collapsing the inner edges of each $O$-decorated tree, according to the axioms of operadic composition. 

\end{ex}

Just as the operad endofunctor $\Mop$ takes a $\Bifiso$-presheaf $O$ to trees decorated by $O$, the non-unital modular operad endofunctor $T$ on $\GS$ 
takes a graphical species $S$ to the graphical species $TS$ whose elements are formal multiplications and contractions in $S$, represented by $S$-structured connected graphs.

\subsection{$X$-graphs and an endofunctor for non-unital modular operads}
The first step in defining the endofunctor $T\colon \GS \to \GS$ is to bijectively label graph boundaries by finite sets.

\begin{defn}
	\label{X graph}
	
	Let $X $ be a finite set. An \emph{(admissible) $X$-graph} is a pair $\X = (\G, \rho)$, where $\G $ is a connected graph such that $ V \neq \emptyset$ and $\rho\colon E_0\xrightarrow{\cong} X$
	is a bijection, called an \emph{$X$-labelling} for $\G$.
	
An \emph{$X$-isomorphism} $\X \to \X'$ of $X$-graphs $\X = (\G, \rho)$ and $ \X' = (\G', \rho')$ is an isomorphism $ g \in \Gr(\G, \G')$ that preserves the $X$-labelling: $\rho' \circ g_{E_0} = \rho\colon E_0 \to X.$
	
	The groupoid of $X$-graphs and $X$-isomorphisms is denoted by $X{\Griso}$.

\end{defn}

\begin{rmk}
	It is sometimes convenient to use the same notation for labelled and unlabelled graphs. In particular, an $X$-graph $\X = (\G, \rho)$ is denoted simply by $\G$ when the labelling $\rho$ is trivial or canonical. For example, for any finite set $X$, the corolla $\CX$ canonically defines an $X$-graph $\CX = (\CX,id)$.

\end{rmk}

\begin{ex}\label{line labelling} 
	For $k \geq 0$, the line graph $\Lk$, with $E_0(\Lk) =\{1_{\Lk}, 2_{\Lk}\}$, $k \geq 0$ is labelled by $1_{\Lk} \mapsto 1 \in \two$ and therefore has the structure of a $\two$-graph when $k \geq 1$. However, $\Lk[0] = (\shortmid)$ has empty vertex set and is therefore not an (admissible) $\two$-graph. 
\end{ex}

For all finite sets $X$, there is a canonical functor $X\Griso \to \Gr \hookrightarrow \GS$. 
A graphical species $S$ defines a presheaf on $X\Griso$ with $S(\X) = S(\G)$ for $\X = (\G, \rho)$. 
Objects of the corresponding element category $\ovP{S}{X\Griso}$ are called \textit{$S$-structured $X$-graphs}.


We can now define the non-unital modular operad endofunctor $T$ on $\GS$, that takes a graphical species $S$ to equivalence classes of $S$-structured graphs. 

For all graphical species $S$, let $TS$ be defined on objects by
\begin{equation}\label{free}
\begin{array}{llll}
TS_\S &= & S_\S, &\\
TS_X &= & \mathrm{colim}_{\X\in X{\Griso}} S(\X) & \text{ for all finite sets } X.
\end{array}
\end{equation}

Let $Aut_X(\X)\defeq X{\Griso}(\X,\X)$ be the automorphism group of an $X$-graph $\X$. If $g, g' \in X{\Griso}(\X,\X')$ are parallel $X$-isomorphisms, then there are $\sigma \in Aut_X (\X)$ and $ \sigma' \in Aut_X(\X')$ such that $g' = \sigma' g \sigma$.
Therefore, there is a completely canonical (independent of $g \in X{\Griso}(\X,\X')$) choice of natural (in $\X$) isomorphism \begin{equation} 
\label{S Aut}
\frac{S(\X)}{Aut_X(\X)} \xrightarrow{\cong} \frac{S(\X')}{Aut_X(\X')}, \ \ [\alpha] \mapsto [g(\alpha)], \text{ for } \alpha \in S(\X). \end{equation}


It follows from (\ref{S Aut}), that 
\begin{equation}\label{eq: component identity}\begin{array}{lll}TS_X& = & \coprod_{ [ \X] \in \pi_0(X{\Griso})} \frac{ S(\X)}{Aut_X (\X)}\\
{}&{}&{}\\
{}&=& \pi_0(\ovP{S}{X\Griso})\end{array} \end{equation}
where $ [ \X] \in \pi_0(X{\Griso})$ is the connected component of $ \X$ in $X{\Griso}$. 

Hence, elements of $TS_X$ may be viewed as isomorphism classes of $S$-structured $X$-graphs, and two $S$-structured $X$-graphs $(\X, \alpha)$ and $(\X', \alpha') $ represent the same class $[\X, \alpha] \in TS_X$ precisely when there is an isomorphism $g \in X{\Griso}(\X,\X')$ such that $S(g)(\alpha') = \alpha$. 

Since bijections $f \colon X \xrightarrow{\cong}Y$ of finite sets induce isomorphisms $\ovP{S}{X\Griso} \xrightarrow{\cong}\ovP{S}{Y\Griso}$, the action of $TS$ on isomorphisms in $\fisinv$ is the obvious one. 

The projections $ TS(ch_x)\colon TS_X \to TS_\S = S_\S$ are induced by the projections $\ovP{S}{X \Griso} \to S_\S $ given by $ \ (\X, \alpha) \mapsto S({ch^{\X}_x})(\alpha),$ 
where $ch_x^{\X} \in \Gr(\shortmid, \G)$ is the map $ch_{\rho^{-1}(x)}$ defined by $ 1 \mapsto \rho^{-1}(x) \in E_0 (\G).$ 

This is well-defined since, if $(\X, \alpha)$ and $(\X', \alpha')$ represent the same element
of $ TS_X$, then there is an $X$-isomorphism $g \colon \X \to \X'$ such that $S(g) (\alpha')= \alpha \in S_X$ 
and hence
\[S({ch^{\X'}_x}) (\alpha') = S({ch^{\X}_x}) \circ S(g) (\alpha') = S({ch^{\X}_x})(\alpha).\] 

So $TS$ describes a graphical species. Moreover, it is clear from the definition that the assignment $S \mapsto TS$ extends to an endofunctor $T$ on $\GS$, with unit $ \eta^\TT \colon id_{\GS} \Rightarrow T$ given by the canonical maps $S_X \xrightarrow \cong S(\CX) \to TS_X$ for all $X$. 

\subsection{Gluing constructions}\label{subs: gluing}

A monadic multiplication $\mu^\TT$ for the pointed endofunctor $(T, \eta^\TT)$ will be defined in terms of colimits of 
a certain class of diagrams in $\Gr$. However, since $\Gr$ does not admit general colimits (see Examples \ref{deg loop} and \ref{not cocomplete}), a small amount of preparation is necessary.


%

Let $S$ be a graphical species and $Y$ a finite set. Since, elements of $TS_Y$ are represented by $S$-structured $Y$-graphs, it follows that, for all finite sets $X$, elements of $T^2S_X$ are represented by $X$-graphs $\X$ that are decorated by $S$-structured graphs. In other words, each $[\X, \beta] \in T^2S_X$ 
is represented by a functor 
\[ el (\beta)\colon \elG[\X] \to \ElS[TS] ,\ \left \{ \begin{array}{lll} (\CX[X_b],b)& \mapsto (\CX[X_b], S(b)(\beta)), & \text{ where } S(b) (\beta) \in TS_{X_b}\\ (\shortmid, ch_e) & \mapsto (\shortmid, c) & c \in S_\S \end{array} \right . \] such that 
\[ el(\beta) (ch_{x_b})(S(b)(\beta)) = el(\beta)(ch_e) \in S(\shortmid)\]
 for all morphisms in $\elG[\X]$ of the form \[ \xymatrix{(\shortmid) \ar[rr]^-{ch_{x_b}} \ar[rd]_-{ch_e} && \CX[X_b] \ar[ld]^-b\\ &\X.&}\] 



Then, as in the operad case (\cref{ex: operad endo}, \cref{fig: operad monad}), we would like to think of the monad multiplication as forgetting the vertices of the original graph $\X$, to obtain an element of $TS_X$ (\cref{fig: graph nesting}). 

Graphs of graphs are functors that encode this idea:

\begin{defn} \label{graph of graphs}
	Let $\G$ be a graph. A \emph{$\G$-shaped graph of graphs} is a functor $ \Gg\colon \elG \to \Gret$ such that
	\[\begin{array}{ll}
\Gg(ch_e) = (\shortmid) & \text{ for all } (\shortmid, ch_e) \in \elG, \\
E_0(\Gg(b)) = X_b & \text{ for all } (\CX[X_b], b) \in \elG,
	\end{array}\]
	and, for all $(\C_{X_b},b) \in \elG $ and all $ x_b \in X_b$,
	\[ \Gg(ch_{x_b}) = ch^{\Gg(b)}_{x_b} \in \Gret(\shortmid, \Gg(b)).\]
	\label{defn degenerate}
	A $\G$-shaped graph of graphs $ \Gg\colon \elG \to \Gret$ is \emph{non-degenerate} if, for all $ v \in V$, $\Gg (\esv)$ has no stick components. Otherwise, $\Gg$ is \emph{degenerate}.

\end{defn}
\begin{figure}[!htb]
	
	\includestandalone[width = .4\textwidth]{standalones/substitutionstandalone}
	\caption{A $\G$-shaped graph of graphs $\Gg$ describes \textit{graph substitution}: a graph $\G_v$ and bijection $ E_0(\G_v) \xrightarrow{\cong} (\vE)^\dagger$ is assigned to each vertex $v$ of $\G$. 
		When $\Gg$ is non-degenerate, taking its colimit in $\Gret$ corresponds to erasing the inner (blue) nesting. } \label{fig: graph nesting}
\end{figure}

	Informally, a non-degenerate $\G$-shaped graph of graphs is a rule for substituting graphs into vertices of $\G$ as in \cref{fig: graph nesting}. However, this intuitive description of a graph of graphs in terms of graph insertion does not always apply in the degenerate case (see Sections \ref{degenerate} and \ref{s. Unital}). 


By \cref{lem: essential cover}, every graph $\G$ is the colimit of the (non-degenerate) \textit{identity $\G$-shaped graph of graphs} $\Gid$ given by the forgetful functor $\elG \to \Gret$, $(\C, b) \mapsto \C$. It follows from \cref{subs. esG} 
that, if $\G$ has no stick components, this is equivalent to the statement that $\G$ is the coequaliser of the canonical diagram
\begin{equation}\label{eq: graph data}
\xymatrix{ \mathcal S(\EI) \ar@<4pt>[rr]\ar@<-4pt>[rr]&&\coprod_{v \in V} \Cv \ar[rr]^-{\coprod (\esv)}&& \G.} \end{equation}

To prove that all non-degenerate graphs of graphs admit a colimit in $\Gret$, we generalise this observation using a modification of \cite[Section~1.5.1]{Koc16}, where \textit{gluing data for directed graphs} were described. A directed graph version of \cref{glue} appears as \cite[Proposition~1.5.2]{Koc16}.

\begin{defn}\label{def: gluing}
Let $\mathcal S = \coprod_{ i \in I} (\shortmid_i) $ be a shrub, and let $\G$ be a (not necessarily connected) graph without stick components. 
 A pair of parallel morphisms $\delta_{1}, \delta_{2}\colon\mathcal S \rightrightarrows \G$ such that
\begin{itemize}
	\item $\delta_{1}, \delta_{2}$ are injective and have disjoint images in $\G$; and
	\item for all $i \in I$, $\delta_{1}(1_i) $ and $\delta_2(2_i) $ are ports of $\G$,
\end{itemize}
 is called a \emph{gluing datum} in $\Gret$.
\end{defn}

\begin{lem}\label{glue}
	Gluing data admit coequalisers in $\Gret$.

\end{lem}

\begin{proof}
	
	Let $\G$ be a graph without stick components and let $\delta_{1}, \delta_{2}\colon\mathcal S = \coprod_{i \in I}(\shortmid_i)\rightrightarrows \G$ be a gluing datum with coequaliser $\overline p\colon \G \to \overline\G = (\overline E, \overline H,\overline V, \overline s,\overline t, \overline \tau)$ in the category $\GrShape$ of graph-like diagrams. 
	
	Since $\delta_1$ and $\delta_2$ are injective and have disjoint images, the induced map $\overline \tau\colon \overline E \to \overline E$ is a fixed-point free involution. Moreover $H = \overline H$ since, if half-edges $h$ and $h'$ of $\G$ are identified in $\overline H$, then there is an edge $l \in E(\mathcal S) $ such that $\delta_1(l) = s(h)$ and $ \delta_2(l) = s(h')$. This contradicts the conditions of \cref{def: gluing} since $\G$ has no stick components. Likewise, edges $ e, e' \in E(\G)$ are identified in $\overline E$ if and only if there is an $ l \in E(\mathcal S)$ such that $\delta_1(l) = e$ and $\delta_2 (l) = e'$ (or vice versa). Since $ \G$ has no stick components, and $\delta_1, \delta_2$ have disjoint images, we may assume that $ e $ and $\tau e'$ are ports and $e', \tau e  \in s(H)$. Therefore,$\overline s \colon \overline H \to \overline E$ is injective, and $\overline \G$ is a graph.

In particular, $\overline V = V$ since $\overline H = H$ and $\mathcal S$ is a shrub. 
It follows that $\overline p \colon \G \to \overline \G$ is an \'etale embedding, and the lemma is proved. \end{proof}

\begin{ex}\label{ex: still more M N}The graphs $\mathcal M^{X,Y}_{x,y}$ and $\mathcal N^X_{x,y}$ (Examples \ref{ex: M graph}, \ref{ex: N graph}, \ref{ex: first gluing example} ) are coequalisers of gluing data (\ref{eq: M coeq}), (\ref{eq: N coeq}): 
\[ \left(ch_x, \ ch_y\circ \tau\colon (\shortmid)\ \rightrightarrows \ (\CX[X \amalg \{x\}] \amalg \CX[Y \amalg \{y\}])\right) \longrightarrow \mathcal M^{X,Y}_{x,y}, \qquad
\left(ch_x,\ ch_y\circ \tau\colon (\shortmid)\ \rightrightarrows \CX[X \amalg \{x,y\}] \right)\longrightarrow \mathcal N^{X}_{x,y}. \]

This is visualised in \cref{fig: glue MN}. 

\begin{figure}[htb!] \label{precontraction}
	\begin{tikzpicture}
		\node at (0, 0){\begin{tikzpicture}[scale = .85]
			%
					\node at (0,0){\begin{tikzpicture}[scale = 0.5]
						\filldraw(0,0) circle(3pt);
						\foreach \angle in {-0,90,180,270} 
						{
							\draw(\angle:0cm) -- (\angle:1.5cm);

						}
						\draw [ ultra thick] (0,0) -- (1.5,0);
						
						\node at (-1, -0.3) {};
						\node at (0.3, 1) {};
						\node at (1, -0.3) {\tiny x};
						\node at (0.3,-1) {};
						\end{tikzpicture}};
					
					\node at (0,-1.5) {$\CX[{X \amalg \{x\}}]$};

					\node at (2,0){\begin{tikzpicture}[scale = 0.5]
						\filldraw(0,0) circle(3pt);
						\foreach \angle in {60,180, 300} 
						{
							\draw(\angle:0cm) -- (\angle:1.5cm);

						}
						\draw [ ultra thick] (0,0) -- (-1.5,0);
						\node at (-1, -0.3) {\tiny y};
						\node at (1, 0.8) {};
						\node at (1, -0.8) {};

						\end{tikzpicture}};
					
					\node at (2,-1.5){$\CX[{Y \amalg \{y\}}]$};
					
					\node at (4,0){\large $\longrightarrow$};
					
					\node at (6,0) 
					{\begin{tikzpicture}[scale = 0.5]
						\filldraw(0,0) circle(3pt);
						\foreach \angle in {-0,90,180,270} 
						{
							\draw(\angle:0cm) -- (\angle:1.5cm);

						}
						\draw [ ultra thick] (0,0) -- (1.5,0);
							\node at (.3, -0.3) {\tiny y};
						\node at (2.2, -0.3) {\tiny x};
						
						\end{tikzpicture}};
					
					\node at (7,0){\begin{tikzpicture}[scale = 0.5]
						\filldraw(0,0) circle(3pt);
						\foreach \angle in {60,180, 300} 
						{
							\draw(\angle:0cm) -- (\angle:1.5cm);

						}
						\draw [ ultra thick] (0,0) -- (-1.5,0);

						\end{tikzpicture}};

					\node at (6.5,-1.5){ $\mathcal M^{X,Y}_{x,y}$};
			%
			
			\end{tikzpicture}};
	
	\node at (10, 0){\begin{tikzpicture}[scale = .85]
	\node at (0,0){\begin{tikzpicture}[scale = 0.5]
		\filldraw(0,0) circle(3pt);
		\foreach \angle in {-0,90,180,270} 
		{
			\draw(\angle:0cm) -- (\angle:1.5cm);

		}
		\draw [ ultra thick] (0,0) -- (1.5,0);
		\draw [ ultra thick] (0,0) -- (-1.5,0);
		\node at (-1, -0.3) {\tiny x};
		\node at (1, -0.3) {\tiny y};
		\end{tikzpicture}};
	
	\node at (0,-1.5) {$\CX$};

	\node at (2.5,0){\large $\longrightarrow$};
	
	\node at (5,0) 
	{\begin{tikzpicture}[scale = 0.5]
		\filldraw(0,0) circle(3pt);
		\draw (0,-1.5)--(0,1.5);
		\draw[ultra thick] 
		(0,0)..controls (3,3) and (-3,3)..(0,0);
		\node at (-1,.6) {\tiny y};
		\node at (1, .6) {\tiny x};
		\end{tikzpicture}};

	\node at (5,-2){ $\mathcal N^{X}_{x,y}$};
	
\end{tikzpicture}};
\end{tikzpicture}

	\caption{ Construction of $\mathcal M^{X,Y}_{x,y}$ and $\mathcal N^X_{x,y}$ as coequalisers of gluing data. }
\label{fig: glue MN}
\end{figure}
\end{ex}

\begin{prop}\label{colimit exists}
	
	A non-degenerate $\G$-shaped graph of graphs $\Gg\colon \elG \to \Gret$ admits a colimit $\Gg(\G)$ in $\Gret$.
	
\end{prop}

\begin{proof} 
	
	For all graphs $\G$, $\elG$ is a connected category if and only if $\G$ is a connected graph. So, the colimit $\Gg(\G)$ of a $\G$-shaped graph of graphs $\Gg\colon \elG \to \Gret$, if it exists, may be constructed componentwise. In particular, we may assume that $\G$ is connected.
	
A non-degenerate $(\shortmid)$-shaped graph of graphs is just an isomorphism $ (\shortmid) \xrightarrow {\cong} (\shortmid) $. 
	
Assume therefore, that $\G \not \cong (\shortmid)$ is a connected graph. 
Since $\Gg$ preserves graph boundaries objectwise on $\elG$, we may apply $\Gg$ to each component of (\ref{eq: graph data}) to obtain a diagram in $\Gret$:
\begin{equation}\label{eq: Gg gluing} \xymatrix{ \coprod_{\tilde e \in \widetilde{\EI}} (\shorte)\ar@<4pt>[rr]
	\ar@<-4pt>[rr]
	&& \coprod_{v \in V} \Gg(\Cv, \esv).}\end{equation}
This is a gluing datum since $\Gg$ is non-degenerate. Therefore (\ref{eq: Gg gluing}) has a colimit $\overline \G$ in $\Gret$ by \cref{glue}. 

For vertices $v' \in V$, and inner edges $e' \in \EI$, there are canonical inclusions 
\begin{equation} \label{eq. inc a} \xymatrix{\Gg(\esv[v']) \ar@{^{(}->} [r]& \coprod_{v \in V} \Gg(\esv) ,}\ {\text{ and }} \ \xymatrix{\Gg(\ese[ \tilde e']) \ar@{^{(}->} [r]&\coprod_{\tilde e \in \widetilde{\EI}} \Gg(\ese) \ar [r]^{\cong}& \coprod_{\tilde e \in \widetilde{\EI}}(\shorte). }\end{equation}

Moreover, since $\G \not \cong (\shortmid)$ is connected, for any port $e \in E_0(\G)$, there is a unique half-edge $(\tau e, w) \in H(\G)$ and the morphism $\Gg(\esh[(\tau e, w)])\colon \Gg(\ese) \hookrightarrow \Gg(\esv[w])$ induces an inclusion
\begin{equation} \label{eq. inc b}\xymatrix{\Gg(\ese) \ar@{^{(}->} [r]&\Gg(\esv[w]) \ar@{^{(}->} [r]& \coprod_{v \in V} \Gg(\esv).} \end{equation} 

The inclusions (\ref{eq. inc a}) and (\ref{eq. inc b}) describe a functor from $\esG $ to the diagram (\ref{eq: Gg gluing}), and hence a cocone of $ \Gg$ above $\overline \G$. 



Conversely, $\Gg$ has a colimit $\Gg(\G)$ in the category $\GrShape$ of graph-shaped diagrams and the cocone of $\Gg$ above $\Gg(\G)$ factors through (\ref{eq: Gg gluing}). 
Hence, by the universal property of colimits, $\coGg = \overline \G$ is a graph. It is the colimit of $\Gg$ in $\Gret$ since $\overline G$ is the coequaliser of (\ref{eq: Gg gluing}) in $\Gret$. 
%
%
\end{proof}

\begin{rmk} In fact, as will follow from \cref{Gnov construction}, all graphs of graphs admit a colimit in $\Gr$. However, the non-degeneracy condition simplifies the proof of \cref{colimit exists}, and is all that is needed for now. 
\end{rmk}

\begin{cor}\label{graph of graphs edges}
	If $\G$ is a graph, and $\Gg$ is a non-degenerate $\G$-shaped graph of graphs with colimit $\coGg$, then the induced map $E(\G) \to E(\Gg)$ on edges is injective, and restricts to the identity
	$ E_0(\G) \xrightarrow{=} E_0(\coGg)$ on ports. 
	For each $(\C,b) \in \elG$, the universal map $ \Gg(b) \to \coGg$ is an \'etale embedding. In particular, \[E(\coGg) \cong E(\G) \amalg \coprod_{v \in V} \EI (\Gg(\esv)).\]
\end{cor}

\begin{proof}\label{iCf}
	The final statement follows directly from the first two. 
	
	By the proof of \cref{colimit exists}, only the inner edges of $\G$, and, for all $(\C,b) \in \elG$, the $\tau$-orbits of ports of $\Gg(b)$ are involved in forming the colimit 
	$\coGg$ of $\Gg$. Hence $\Gg$ induces a strict inclusion 
	 \[ \xymatrix{\coprod_{\tilde e \in \widetilde E} (\shorte)\ar[r]^{\cong} &\coprod_{\tilde e \in \widetilde E} \Gg(\ese) \ar@{^{(}->} [r] &\coGg}\] 
	 that restricts to an identity $ E_0(\G) = E_0(\coGg)$ of ports.
	The second part is immediate. 
\end{proof}

The following corollary was proved, for directed graphs, in \cite[Lemma~1.5.12]{Koc16}.

\begin{cor}\label{connected colimit}
Let $\Gg$ be a non-degenerate $\G$-shaped graph of graphs with colimit $\coGg$ in $\Gret$. If $\Gg (\C, b)$ is connected for each $(\C, b) \in \elG$, then $\coGg$ is a connected graph 
if and only if $\G$ is. 
\end{cor}

\begin{proof}
	
A $(\shortmid)$-shaped graph of graphs is isomorphic to the identity functor $(\shortmid) \mapsto (\shortmid)$ with colimit $(\shortmid)$. 

So, assume that $\G$ has no stick components and let $\Gg\colon \elG \to \Gr$ be a non-degenerate $\G$-shaped graph of graphs with colimit $\coGg$. 
	
		A morphism $\gamma \in \GrShape(\coGg, \bigstar \amalg \bigstar)$ 
	is equivalently described by a commuting diagram in $\GrShape$: \begin{equation}\label{connected cocone} \xymatrix{
		\mathcal S(\EI)\ar@<4pt>[rr]
		\ar@<-4pt>[rr]
		&& \coprod_{v \in V(\G)} \Gg(\esv) \ar[rr]&& \bigstar \amalg \bigstar.}\end{equation}
	
	Let $\Gg(\esv)$ be connected for each $v \in V$. Then each map $\Gg(\esv) \to \bigstar \amalg \bigstar$ is 
	constant, and morphisms $\coprod_{v \in V} \Gg(\esv) \to \bigstar \amalg \bigstar$ are in bijection with morphisms $ \coprod_{v \in V} \Cv \to \bigstar \amalg\bigstar $.
	
	So, $\GrShape(\G, \bigstar \amalg \bigstar) \cong \GrShape(\coGg, \bigstar \amalg \bigstar)$, and it 
	follows from \cref{conn 1} that $\coGg$ is connected if and only if $\G$ is.
\end{proof}

	Let $\G$ be a graph and $S$ any graphical species. As usual, let $\Comm$ be the terminal graphical species. 
\begin{defn} \label{graph of S graphs}
 A \emph{(non-degenerate) $\G$-shaped graph of $S$-structured graphs} is a functor $\Gg_S\colon \elG\to \Grets $ such that the functor $\Gg\colon \elG \to \Gret$ induced by the unique morphism $S \to \Comm$
	\begin{equation}\label{eq. Gg forget} \Gg\colon \xymatrix{\elG \ar[rr]^-{\Gg_S}&& \Grets \ar[rr] &&\Grets[\Comm] \ar@{=}[r] &\Gret} \end{equation}
	is a (non-degenerate) $\G$-shaped graph of graphs.
%
%

	\label{defn: graphs of graphs cat}
	For a connected graph $\G$, objects of the category $\GrSG$ are {non-degenerate $\G$-shaped graphs of $S$-structured graphs $\Gg_S\colon \elG \to \Grs $,} and morphisms are natural transformations. 
\end{defn}

\begin{lem}\label{GrSG components}
	For $\G$ connected, $ \G \not \cong \C_\nul$, two $\G$-shaped graphs of (connected) $S$-structured graphs $\Gg_S ^1, \Gg_S ^2$ are in the same connected component of $\GrSG$ if and only if, for all $(\CX[X_b], b) \in \elG$, $\Gg_S ^1(b)$ and $ \Gg_S ^2(b)$ are in the same connected component of $\ovP{S}{X_b\Griso}$.
	
	In particular, if $\Gg_S ^1 $ and $ \Gg_S ^2$ are in the same connected component of $\GrSG$, then $\Gg_S ^1$ and $ \Gg_S ^2$ have isomorphic colimits in $\Grs $.
\end{lem}

\begin{proof}
	Let $\G \not \cong \C_{\nul}$ be connected. So, if $(\CX[X_b], b) \in \elG$, then $X_b\neq \emptyset$. Given a morphism $\phi\colon\Gg_S ^1\Rightarrow \Gg_S ^2$ in $\GrSG$, its 
	component $ \phi_{(b)}$ at $b$ is, by definition, a boundary-preserving morphism in $ \Grs $, and hence 
	by \cref{pp iso}, an $X_b$-isomorphism in $\ovP{S}{X_b\Griso}$. 
	
	The converse is immediate, as is the final statement.	
\end{proof}

\subsection{Multiplication for the monad $\TT$}\label{subs. unpointed multiplication}
The aim of this section is to describe the multiplication $ \mu^\TT \colon T^2 \Rightarrow T$ in terms of colimits of graphs of graphs.

From now on, all graphs will be connected, unless explicitly stated otherwise.

	Let $X$ be a finite set and $\X = (\G, \rho)$ an $X$-graph.
	If $\Gg \colon \elG[\X] \to \Gr$ is a non-degenerate $\X$-shaped graph of graphs, then its colimit $\coGg[\X] = \mathrm{colim}_{ \elG[\X]} \Gg$ exists by \cref{colimit exists} and, by \cref{graph of graphs edges}, it inherits the $X$-labelling $\rho$ of $\X$.

Given a graphical species $S$ and finite set $X$, elements $[\X, \beta]$ of $T^2S_X$ are represented by pairs $(\X, \Gg_S)$ where $\X$ is an $X$-graph, and $\Gg_S$ is an $\X$-shaped graphs of connected $S$-structured graphs. The colimit of $\Gg_S$ in $\Grs $ is given by a pair $(\Gg(\X), \alpha)$, where $\Gg(\X)$ is the colimit of the underlying $\X$-shaped graph of graphs $\Gg$ defined as in (\ref{eq. Gg forget}) and $\alpha \in S(\Gg(\X))$. 


For $j = 1,2$,  let $(\X^j, \Gg_S^j)\colon \elG[\X^j] \to \Grs $ represent the same element $[\X, \beta] $ of $ T^2 S_X$. Then, by definition of $T$, $\X^1 \cong \X^2$ in $X{\Griso}$ and, by \cref{GrSG components},
\begin{equation}\label{same colimit} \mathrm{colim}_{ \elG[\X^1]} \Gg_S^1 \cong \mathrm{colim}_{ \elG[\X^2]} \Gg_S^2 \in \ovP{S}{X\Griso}.\end{equation}

A multiplication $\mu^\TT \colon T^2 \Rightarrow T$ for $(T, \eta^\TT)$ will be induced by the (by (\ref{same colimit}) well-defined) assignments:
\begin{equation} \label{eq. T mult} [\X, \beta] \mapsto [\Gg(\X), \alpha]\end{equation}

To see that (\ref{eq. T mult}) extends to a morphism $ \mu^\TT S \colon T^2S \to TS$ of graphical species, let $[\X, \beta] \in T^2S_X$ be represented by an $\X$-shaped graph of $S$-structured graphs $\GSg\colon \elG[\X] \to \Grs$ with colimit $\GSg(\X) = (\Gg(\X), \alpha)$ in $ \ovP{S}{X\Griso}$. 

By \cref{graph of graphs edges}, there is a canonical inclusion $E(\X) \hookrightarrow E(\coGg[\X])$ of edge sets, and for each $e \in E(\X)$, 
\[ S(ch^{\coGg[\X]}_e) (\alpha) = S(ch^\X_e) (\beta) \in S(\shortmid).\]
Hence, for all $x \in X$, there is a commuting diagram of sets \[
\xymatrix{ T^2S_X \ar[rr]^-{ \mu^\TT S_X}\ar[dr]_{T^2S(ch_x)}&& TS_X \ar[dl]^{TS(ch_x)}\\& T^2S(\shortmid) = TS(\shortmid). }\]

Naturality of $ \mu^\TT S$ in $S$ is immediate from the definition and, by a straightforward modification of \cite[Section~2.2]{Koc16}, it may be shown that $\TT = (T, \mu^\TT , \eta^\TT)$ satisfies the two axioms 
for a monad.

\begin{rmk}
For all graphical species $S$, $\mu^\TT S$ and $ \eta^\TT S$ are palette-preserving morphisms in $\GS$. So $\TT$ restricts to a monad $\TT^{(\CCC, \omega)}$ on $\CGS$, for all $(\CCC, \omega)$. If $A$ is a $(\CCC, \omega)$-coloured graphical species and $h \in \GS(TA, A)$, then $(A,h)$ is a $\TT$-algebra if and only if it is a $\TT^{(\CCC, \omega)}$-algebra.
\end{rmk} 

\begin{ex}\label{terminal MO}\label{terminal C MO}
If $\Comm$ is the terminal graphical species, then $ \Grs[\Comm] \cong \Gr$ and hence elements of $T\Comm$ are boundary-preserving isomorphism classes of graphs in $\Gr$. The unique morphism $!\in \GS(T \Comm, \Comm)$ makes $\Comm$ into an algebra for $\TT$. 
	Likewise, for any palette $(\CCC, \omega)$, the terminal $(\CCC, \omega)$-coloured graphical species $\CComm{}$ is a $\TT$-algebra together with the unique palette-preserving morphism $!^{(\CCC, \omega)} \colon T\CComm{}\to \CComm{}$.
\end{ex}

\subsection{$\TT$-algebras are non-unital modular operads.}
Having constructed the monad $\TT$, it remains to prove that $\TT$-algebras are non-unital modular operads. 

\begin{lem}\label{lem: algebra operations}
A $\TT$-algebra $(A,h)$ admits a multiplication ${{}_h \diamond}$ and contraction ${{}_h \zeta}$, that are natural with respect to morphisms in $\GS^\TT$.
\end{lem}
\begin{proof}
	Let $X$ and $Y$ be finite sets and
let $\mathcal M^{X,Y}_{x,y}$ be the $X \amalg Y $-graph (described in Examples \ref{ex: M graph} and \ref{ex: still more M N}) obtained by gluing the corollas $\CX[X \amalg \{x\}]$ and $\CX[Y \amalg \{y\}]$ along ports $ x$ and $y$. 

Let $S$ be a $(\CCC, \omega)$-coloured graphical species. For $\underline c \in \CCC^X$, $\underline d \in \CCC^Y$ and $c \in \CCC$, let $\mathcal M_c(\phi, \psi)$ be the element of $ S(\mathcal M^{X,Y}_{x,y})$ determined by an ordered pair $(\phi, \psi) \in S_{(\underline c,c)}\times S_{(\underline d, \omega c)}$. 
The canonical map $ S(\mathcal M^{X,Y}_{x,y}) \to TS_{X\amalg Y}$ is injective unless $X = Y = \emptyset$, in which case $[\mathcal M_c(\phi_1, \psi_1)] = [\mathcal M_c(\phi_2, \psi_2)]$ when $(\phi_2, \psi_2) = (\psi_1, \phi_1) \in S_{\{x\}}\times S_{\{y\}}$. 

If $(A, h)$ is a $(\CCC, \omega)$-coloured $\TT$-algebra, then the maps given by the compositions
\[ {{}_h \diamond}\colon \xymatrix{S_{(\underline c,c)}\times S_{(\underline d, \omega c)} \ar[rr]^-{[\mathcal M(\cdot,\cdot)]} && 
TS_{\underline c\underline d}\ar[rr]^-{h}&& S_{\underline c \underline d}}\] are $\fisinv$-equivariant by construction and hence induce a multiplication on $A$ (see \cref{fig: h operations}). 

Recall similarly that, for any finite set $X$, $\mathcal N^X_{x,y}$ is the $X$-graph (described in Examples \ref{ex: N graph} and \ref{ex: still more M N}) obtained by gluing the ports $x$ and $y$ of $\CX[X \amalg \{x,y\}]$. For $\underline c \in \CCC^X$ and $c \in \CCC$, let $\mathcal N^S_c (\phi) \in S(\mathcal N^{X}_{x,y})$ be the element determined by $ \phi \in S_{(\underline c,c, \omega c)} \subset S_{X \amalg \{x,y\}}$. 

The only non-trivial boundary-preserving automorphism of $ \mathcal N^X_{x,y}$ is the permutation $\sigma_{x,y} \in Aut (X\amalg \{x,y\})$ that fixes $X$ and switches $x$ and $y$. So, $[\mathcal N^S_c (\phi)] = [\mathcal N^S_c (\psi)]$ in $TS_X$ if and only if $ \phi = \psi$ or $S(\sigma_{x,y})(\phi) = \psi$.

If $(A, h)$ is a $(\CCC, \omega)$-coloured algebra for $T$, then the maps given by the compositions
\[{{{}_h \zeta}}\colon \xymatrix{ A_{(\underline c,c, \omega c)} \ar[rr]^-{[\mathcal N^A(\cdot)]} &&TA_{\underline c }\ar[rr]^- {h}&& A_{\underline c}}\]
are $\fisinv$-equivariant and induce a contraction ${{}_h \zeta}$ on $A$ (see \cref{fig: h operations}). Naturality of ${{}_h \diamond}$ and ${{}_h \zeta}$ is immediate from the construction. 
\begin{figure}[htb!] 

			\begin{tikzpicture}
			
			\node at (-2.5,0){
			\begin{tikzpicture}
		\node at (6,0) 
		{\begin{tikzpicture}[scale = 0.5]
			\foreach \angle in {-0,90,180,270} 
			{
				\draw(\angle:0cm) -- (\angle:1.5cm);
				\draw[thick]
				(0,0)--(2.1,0);
				\draw [ draw=red, fill=white]
				(0,0) circle (15pt);
				
				\node at(0,0) {\small{$\phi $}};

			}
			
			\end{tikzpicture}};
		
		\node at (7,0){\begin{tikzpicture}[scale = 0.5]
			\foreach \angle in {60,180, 300} 
			{
				\draw(\angle:0cm) -- (\angle:1.5cm);

			}
			\draw [ thick] (0,0) -- (-1.5,0);
			\draw [ draw=red, fill=white]
			(0,0) circle (15pt);
			
			\node at(0,0) {\small{$ \psi $}};
			
			\end{tikzpicture}};
		

	\end{tikzpicture}	};
	\draw[|->] (-1,0)--(0,0);
	\node at (-.5,.3){$h$};
		\node at (1,0){ $\phi \ \diamond_{x,y} \ \psi$};
			\node at (5,0){
		\begin{tikzpicture}
		
%
%
%
%
%

		
		\node at (5,.5) 
		{\begin{tikzpicture}[scale = 0.5]
			\draw (0,-1.5)--(0,3.5);
			\draw[ thick] 
			(0,0)..controls (3,3) and (-3,3)..(0,0);
			\node at (-1,.5) {\tiny {$y$}};
			\node at (1, .5) {\tiny {$x$}};

			\draw [ draw=red, fill=white]
			(0,0) circle (15pt);
			
			\node at(0,0) {\small{$\phi $}};
			\end{tikzpicture}};
	
			\end{tikzpicture}};
			\draw[|->] (6,0)--(7,0);
		\node at (6.5,.3){$h$};
		
		\node at (8,0){ $ \zeta_{x,y} \ (\phi)$};
		
	\end{tikzpicture}
	\caption{If $(A,h)$ is a $\TT$-algebra, $h$ induces a multiplication and contraction on $A$.}	\label{fig: h operations}

\end{figure}
\end{proof}

We are now able to show that algebras for the monad $\TT$ on $\GS$ are precisely non-unital modular operads.

\begin{prop}\label{unpointed modular operad} 
There is a canonical isomorphism of categories $\GS^\TT \cong \nuCSM$.
\end{prop}

\begin{proof}

A $\TT$-algebra structure $h\colon TA \to A$ equips a graphical species $A$ with a multiplication $\diamond = {{}_h \diamond}$, and contraction $\zeta = {{}_h \zeta}$ as in \cref{lem: algebra operations}. We must show that $(A, \diamond, \zeta)$ satisfies conditions (M1)-(M4) of \cref{defn: Modular operad}. 

The proof is based on the observation that (up to its boundary $E_0$) any connected graph with two inner edge orbits has one of the forms illustrated in Figures \ref{condition 1}-\ref{condition 4}, and each of these relates to one of the conditions (M1)-(M4). 

Condition (M1) is illustrated in \cref{condition 1}. 
Let $ \phi _1 \in A_{(\underline b,c)}, \phi_2 \in A_{(\underline c, \omega c, d)}$ and $\phi_3 \in A_{(\underline d, \omega d)}$. By \cref{lem: algebra operations} and the monad algebra axioms, 
\[\begin{array}{lll}
(\phi_1\diamond_{c} \phi_2)\diamond_{d} \phi_3 & = &h \left [ \mathcal M^A_d \left ((\phi_1\diamond_{c} \phi_2) , \phi_3 \right)\right]\\
&= & h \left [ \mathcal M^A_d\left (h[\mathcal M_c(\phi_1, \phi_2)], h \eta^\TT A (\phi_3) \right)\right]\\
& = & h \mu^\TT \left [ \mathcal M^{TA}_d\left ([\mathcal M_c(\phi_1, \phi_2)], \eta^\TT A (\phi_3) \right)\right],
\end{array} \]
{ and, likewise } 
\[\phi_1\diamond_{c} (\phi_2\diamond_{d} \phi_3) = h \mu^\TT \left[ \mathcal M^{TA}_c \left(\eta^\TT A (\phi_1) , [\mathcal M^A_d(\phi_2 , \phi_3)] \right)\right]. \]
Hence, to prove (M1), it suffices to show that, for all $\phi_1, \phi_2, \phi_3$ as above,
\[ \mu^\TT \left [ \mathcal M^{TA}_d\left ([\mathcal M_c(\phi_1, \phi_2)], \eta^\TT A (\phi_3) \right)\right]= \mu^\TT \left[ \mathcal M^{TA}_c \left(\eta^\TT A (\phi_1) , [\mathcal M^A_d(\phi_2 , \phi_3)] \right)\right].\]
By \cref{ex: still more M N} and since colimits commute,  this follows from: 
\[ \mathrm{coeq}_{\Gr} \left (ch_{y}, ch_z \circ \tau\colon (\shortmid) \rightrightarrows \mathcal M^{X_1, (X_2 \amalg \{ y\})}_{w,x} \amalg \CX[X_3 \amalg \{z\}] \right) = \mathrm{coeq}_{\Gr} \left (ch_{w}, ch_x \circ \tau\colon (\shortmid) \rightrightarrows \CX[X_1 \amalg \{w\}] \amalg \mathcal M^{(X_2 \amalg \{x\}), X_3}_{y,z} \right). \]

The coherence conditions (M2)-(M4) all follow in the same way from the defining axioms 
of monad algebras. Figures \ref{condition 2}-\ref{condition 4} illustrate each condition. 

\begin{figure}[h]
	\includestandalone[width = .75\textwidth]{standalones/Coherence1}
	\caption{ Coherence condition (M1) Applying $\mu^\TT A\colon T^2 A \to A$ amounts to \emph{erasing inner nesting}.}\label{condition 1}
\end{figure}


\begin{figure}[h]
	\includestandalone[width = .75\textwidth]{standalones/Coherence_2}
	\caption{Coherence condition (M2)}\label{condition 2}
\end{figure}

\begin{figure}[h]
	\includestandalone[width =.75 \textwidth]{standalones/Coherence_3}
	\caption{Coherence condition (M3).}\label{condition 3}
\end{figure}

\begin{figure}[h]
	\includestandalone[width = .75\textwidth]{standalones/Coherence_4}
	\caption{Coherence condition (M4).}\label{condition 4}
\end{figure}
The induced assignment $(A, h) \mapsto (A, \diamond, \zeta)$ clearly extends to a functor $\GS^\TT \to \nuCSM$.

The proof of the converse closely resembles that of \cite[Theorem~3.7]{GK98}. Namely, let $(S, \diamond, \zeta)$ be a non-unital modular operads. We construct a morphism $h\in \GS(TS ,S)$ by successively using $\diamond$ and $\zeta$ to \textit{collapse} 
 inner edge orbits of $ S$-structured $X$-graphs $(\X, \alpha)$, resulting in a finite sequence of $ S$-structured $X$-graphs that terminates in an $S$-structured corolla $(\CX,\phi)$.

As usual, let $X$ be a finite set and let $(\X, \alpha)$ be a representative of $[\X, \alpha]\in TS_{X}$. 

If $\X$ has no inner edges, 
then $\X = \CX$, and so $[\X, \alpha] = \eta^\TT S(\phi)$ for some $\phi \in S_{X}$. In this case, define \begin{equation}\label{h alg 1} h [\X, \alpha] \defeq \phi \in S_{X}.\end{equation}

Otherwise, let $\X$ have vertex set $V$, edge set $E$, and let $\tilde e \in \widetilde{\EI}$ be the orbit of a pair $e, \tau e $ of inner edges of $\X$. Write $t(e) \defeq ts^{-1}(e)$ for the vertex $v$ with $e \in \vE$. 

There are two possibilities: either $t(e) = t(\tau e)$ or $t(e) \neq t(\tau e)$.


{\sc{Case 1:} }$t(e) = v_1$ and $t(\tau e) = v_2$ are distinct vertices of $\X$.\\
Let $\X_{/\tilde e}$ be the graph obtained from $\X$ by removing the $\tau$-orbit $\{e, \tau e\}$ and identifying $v_1$ and $v_2$ to a vertex $\overline v \in \faktor{V}{(v_1 \sim v_2)}$: \[\X_{/\tilde e}\defeq \ \ \Fgraphvar{(E\setminus \{e,\tau e\})}{(H \setminus s^{-1}\{e, \tau e\})}{\small{\faktor{V}{(v_1 \sim v_2)}}}{s}{\overline t}{\tau}.\]
(Here $\overline t$ is the composition of $ t\colon H \to V$ with the quotient $V \twoheadrightarrow \faktor{V}{(v_1 \sim v_2)}$.) 
So, $\X$ is the colimit of the non-degenerate $\X_{/\tilde e}$-shaped graph of graphs $ \elG[\X_{/\tilde e}] \to \Gr$ given by 
 \[ (\C, b) \mapsto \left \{ \begin{array}{ll} 
	\mathcal M^{X_1, X_2}_{x_1, x_2}& \text{ if } (\C, b)  = (\C_{X_1 \amalg X_2}, \overline b) \text{ is a neighbourhood of } \overline v \in V (\X_{/\tilde e}), \\
	\C &\text{ if } (\C,b)   \text { is not a neighbourhood of } \overline v. \end{array} \right .\]

In particular, if $(\C, b) \in \elG[\X_{/\tilde e}]$ is not a neighbourhood of $\overline v$, then it describes an element $(\C, b^\X)\in \elG[\X]$.


For $i = 1,2$, let $(\C_{X_i \amalg \{x_i\}}, b'_i) \in \elG[\X]$ be minimal neighbourhoods of $v_i$ in $\X$ such that $b'_1(x_1) = \tau e$ and 
$b'_2(x_2) = e$. And let $\phi_i \defeq S(b'_i)(\alpha)\in S(\CX[X_i \amalg \{x_i\}])$. Then there is an $S$-structure $\alpha_{/\tilde e}$ on $\X_{/\tilde e}$ defined, for all $(\C,b) \in \elG[\X_{/\tilde e}]$, by
\[ S(b)(\alpha_{/\tilde e})  = \left \{ \begin{array}{ll} 
	\phi_1 \diamond_{x_1, x_2} \phi_2 & \text{ if } (\C, b)  = (\C_{X_1 \amalg X_2}, \overline b),\\
	S(b^\X) (\alpha) &\text{ if } (\C,b)   \text { is not a neighbourhood of } \overline v.\end{array} \right .\]

%

{\sc{Case 2:} $t(e) = t(\tau e) = v \in V$.}\\
 In this case, the graph $\X_{/\tilde e}$ obtained from $\X$ by collapsing $\{e, \tau e\}$ has the form
\[\X_{/\tilde e}\defeq \ \ \Fgraphvar{(E\setminus \{e,\tau e\})}{(H \setminus s^{-1}\{e, \tau e\})}{V}{s}{\overline t}{\tau},\]  and $\X$ is the colimit of the non-degenerate $\X_{/\tilde e}$-shaped graph of graphs $ \elG[\X_{/\tilde e}] \to \Gr$:
\[ (\C, b) \mapsto \left \{ \begin{array}{ll} 
	\mathcal N^{X_v}_{x,y}& \text{ if } (\C, b)  = (\C_{X_v}, \overline b) \text{ is a neighbourhood of } v, \\
	\C &\text{ if } (\C,b)   \text { is not a neighbourhood of } v. \end{array} \right .\]

As before, if $(\C, b) \in \elG[\X_{/\tilde e}]$ is not a neighbourhood of $v$, then it describes an element $(\C, b^\X)\in \elG[\X]$.

 Now, let $(\C_{X_v \amalg \{x, y\}}, b') \in \elG[\X]$ be the neighbourhood of $v$ in $\X$ such that $b'(x) = \tau e$ and $b'(y) = e$. 
 Let $\phi \defeq S(b') (\alpha) \in S_{ X_v \amalg \{x,y\}}$. Then there is an $S$-structure $\alpha_{/\tilde e}$ on $\X_{/\tilde e}$ defined, for all $(\C, b) \in \elG[\X_{/\tilde e}]$, by 
 \[ S(b)(\alpha_{/\tilde e})  = \left \{ \begin{array}{ll} 
 	\zeta_{x,y}(\phi) & \text{ if } (\C, b)  = (\C_{X_v}, \overline b), \\
 	S(b^\X) (\alpha) &\text{ if } (\C,b)   \text { is not a neighbourhood of } v. \end{array} \right .\]

%

\medspace

It follows that an ordering $(\tilde e_1, \dots, \tilde e_N)$ of the set $\widetilde{\EI}$ of inner $\tau$ -orbits of $\X$ defines a terminating sequence of $S$-structured $X$-graphs:
\[ (\X, \alpha) \mapsto (\X_{/\tilde e_1}, \alpha_{/\tilde e_1}) \mapsto ((\X_{/\tilde e_1})_{/\tilde e_2}, (\alpha_{/\tilde e_1})_{/\tilde e_2})\mapsto \dots\mapsto (((\X_{/\tilde e_1}) \dots)_{/\tilde e_{N}},(\alpha_{/\tilde e_1})\dots)_{/\tilde e_N}).\]

Since $((\X_{/\tilde e_1}) \dots)_{/\tilde e_{N}} = \CX$ has no inner edges, there is a $\phi_{(\X, \alpha)} \in S_{X}$ such that 
\[(\alpha_{/\tilde e_1} \dots)_{/\tilde e_N}= \eta^\TT S(\phi_{(\X, \alpha)}) \in TS_{X}.\] 

The coherence conditions (M1)-(M4) are equivalent to the statement that $\phi_{(\X, \alpha)} \in S_{X}$ so obtained is independent of the choice of ordering of $\widetilde{\EI}$. Moreover, by construction, the assignment $(\X, \alpha) \mapsto  \phi_{(\X, \alpha)} $  is equivariant with respect to morphisms in $\ovP{S}{X\Griso}$, and so also independent of the choice of representative of $[\X, \alpha] \in TS_{X}$. Hence it extends to a morphism $h \colon TS \to S$, $[\X, \alpha]\mapsto  \phi_{(\X, \alpha)} $ in $\GS$. 

To complete the proof of the proposition, it remains to establish that $ h$ satisfies the monad algebra axioms 
for $\TT$. Compatibility of $h$ with $\eta^\TT $ 
 is immediate from equation (\ref{h alg 1}). Compatibility of $h$ with $\mu^\TT $ 
 follows since the coherence conditions (M1)-(M4) ensure that $ h[\X, \alpha]$ is independent of the order of collapse of the inner edges of $\X$. 

So $(S, \diamond, \zeta)$ defines a $\TT$-algebra $(S, h)$, and this assignment extends in the obvious way to a functor $ \nuCSM \to \GS^\TT$ that, by \cref{lem: algebra operations} and (M1)-(M4), is inverse to the functor $\GS^\TT \to \nuCSM$ defined above.
\end{proof}

		\section{The problem of loops}\label{degenerate}

 Before constructing the (unital) modular operad monad in \cref{s. Unital}, 
  let us first pause to discuss the obstruction to obtaining a monadic multiplication for the modular operad endofunctor in the construction outlined in \cite{JK11}. 

\begin{ex} 	\label{ex: operad units} In \cref{ex: operad endo}, I sketched the idea behind the construction of the symmetric operad monad $\MMop$ on $\pr{\Bifiso}$, whose underlying endofunctor takes a $\DDD$-coloured presheaf $O$ to the $\DDD$-coloured presheaf of formal operadic compositions in $O$, encoded as $O$-decorated rooted trees. However, I did not describe how the units for the operadic composition are obtained. 
	
Unlike $T$ on $\GS$, the definition of $\Mop$ on $\pr{\Bifiso}$ allows \textit{degenerate substitution} of the exceptional directed tree $(\downarrow)$ into the vertex of the rooted corolla $t_\one$ with one leaf (\cref{fig: deg tree}). 
Grafting an exceptional directed edge $({\downarrow})$ onto the leaf or root of any tree $\Tr$ leaves $\Tr$ unchanged (see \cref{fig:grafting}(b)). So, if $(O, \theta)$ is a $\DDD$-coloured algebra for $\MMop$, and hence a $\DDD$ coloured operad, then the elements $\theta (\downarrow, d) \in O(t_\one)$ provide the $\DDD$-coloured units for the operadic composition. 

	\begin{figure}[htb!]
		\begin{tikzpicture}
		
		\node at (2,0){	\begin{tikzpicture}[scale = 0.25]

			{	
				
				%
				{	\draw[thick, -<- = .5]
					(0,.5)--(0,6.5)
					;}

			}
			\end{tikzpicture}};
		
		\draw[->]
		(.7,0)--(1.5,0);
		
		\node at (6,0){	\begin{tikzpicture}[scale = 0.25]

			{	\draw[blue, fill = blue!20]
				(0, 0) circle (50pt)
				(0,5) circle(50pt)
		(-2.7,2.7)circle(50pt)
				(-5.5,5.5) circle(50pt)
				(4.5, 4) circle(50pt);

				{	\draw[thick, gray!80]
					(0,0)--(0,-3)
					(0,0)--(-3.5,3.5)
					(0,0)--(0,7)
					(0,0)--	(4.5,3.5)
					(-3.5,3.5)--(-7,7)
					(4.5,3.5)--(1,7)
					(4,4)--(4,6)
					(4,6)--(4,7)
					(4.5,3.5)--(5,4.2)
					(5,4.2)--(5,7)
					(5,4.2)--(6,7)
					(5,4.2)--(7.5,7);}
				
				\clip (0, 0) circle (40pt)
				(0,5) circle(40pt)
					(-2.7,2.7)circle(40pt)
					(-5.5,5.5) circle(40pt)
				(4.5, 4) circle(40pt);
				
				\draw[ thick, blue]
				(0,0)--(0,-3)
				(0,0)--(-3.5,3.5)
				(0,0)--(0,7)
				(0,0)--	(4.5,3.5)
				(-3.5,3.5)--(-7,7)
				(4.5,3.5)--(1,7)
				(4,4)--(4,6)
				(4,6)--(4,7)
				(4.5,3.5)--(5,4.2)
				(5,4.2)--(5,7)
				(5,4.2)--(6,7)
				(5,4.2)--(7.5,7);

				\filldraw[ red]
				(0,0) circle (6pt)
				(4.5,3.5) circle (6pt)
				(4,4) circle (6pt)
				(4,4.7) circle (6pt)
				(5,4.2) circle (6pt);


			}
			\end{tikzpicture}};
		
		\node at (8.8,.3){\scriptsize{$\mu^{\MMop}$}};
		\draw[->]
		(8.2,0)--(9.5,0);
		
		\node at (11, 0){	\begin{tikzpicture}[scale = 0.25]
			
			\draw[ thick]

	(0,0)--(0,-3)
(0,0)--(-3.5,3.5)
(0,0)--(0,5)
(0,0)--	(4.5,3.5)
(-3.5,3.5)--(-5,5)
(4.5,3.5)--(1,7)
(4,4)--(4,6)
(4,6)--(4,7)
(4.5,3.5)--(5,4.2)
(5,4.2)--(5,7)
(5,4.2)--(6,7)
(5,4.2)--(7.5,7);

			\filldraw[ red]
				(0,0) circle (6pt)
			(4.5,3.5) circle (6pt)
			(4,4) circle (6pt)
			(4,4.7) circle (6pt)
			(5,4.2) circle (6pt);

			\end{tikzpicture}};
		\node at (0,0){	\begin{tikzpicture}[scale = 0.25]

			{	\draw[ blue, fill = blue!20]
				
				(0,4)circle(40pt)
				;

				{	\draw[thick, gray!80]
					(0,0)--(0,7)
					;}
				
				\clip 
				(0,4) circle(35pt);
				
				\draw[ ultra thick, blue]
				(0,0)--(0,7)
				;

			}
			\end{tikzpicture}};
%
%
%
%
%
%
%
%
%
%
%
%
%
%
%
%
%
%
		
%
%
%
%
%
%
%
%
%
%
%
%
%
		
	\end{tikzpicture}
	\caption{The combinatorics of the operadic unit are represented graphically by the degenerate substitution of the exceptional tree into $t_{\mathbf 1}$. Applying the monad multiplication $\mu^{\MMop} O$ to nested trees in ${\Mop}^2 O$ \textit{deletes} vertices decorated by elements of $ O(\downarrow)$. (See also \cref{fig: operad monad}.) }\label{fig: deg tree}
\end{figure}
\end{ex}

The endofunctor $\TJK\colon \GS \to \GS$ defined in \cite[Section~5]{JK11}, whose algebras are modular operads, is obtained by a slight modification of the non-unital modular operad endofunctor $T$, to allow degenerate substitutions analogous to those in \cref{ex: operad units}.


For a finite set $X$, let $\XGrJK$ be the groupoid obtained from $X{\Griso}$ by dropping the condition that $X$-graphs must have non-empty vertex set. So,
\[ \XGrJK = X{\Griso} \text { for } X \not \cong \two \ \text{ and } \ \XGrJK[\two] \cong \two {\Griso} \amalg \{ (\shortmid, id), (\shortmid, \tau)\}.\] 

The endofunctor $ {\TJK}\colon \GS \to \GS$ is defined pointwise by \[\begin{array} {lll}
\TJK S_\S &\defeq S_\S, &\\ \TJK S_{X} &\defeq \mathrm{colim}_{\X \in \XGrJK} S(\X) & \text { for all finite sets } X, \end{array} \] together with the obvious extension of $T$ on morphisms in $\fisinv$.

Since $TS \subset \TJK S$ (for all $S$), $\eta^\TT$ induces a unit $\etaJK$ for the endofunctor $\TJK$ (see \cref{defn: pointed endo}). 

\begin{prop}\label{prop: TJK endo}
Algebras for the pointed endofunctor $(\TJK, \etaJK)$ on $ \GS$ are modular operads.
\end{prop}

\begin{proof}
Since $T \subset \TJK$, algebras for $\TJK$ have the structure of non-unital modular operads by \cref{unpointed modular operad}. If $(A, h)$ is an algebra for $( \TJK, \etaJK)$, then each $c \in A_\S$ defines an element $(\shortmid, c) \in \TJK A_\two $, and $h(\shortmid, c) \in A_\two$ provides a $c$-coloured unit for the induced multiplication.
\end{proof}

However $(\TJK, \etaJK)$  cannot be extended to a monad on $\GS$:

For all graphical species $S$, an element of $ {\TJK}^2 S_X$ is represented by an $X$-graph $\X$ and a (possibly degenerate) $\X$-shaped graph of $S$-structured connected graphs $\Gg_S\colon \elG[\X] \to \Grs $. In particular, if $\TJK$ admits a monad multiplication $\muJK\colon {\TJK}^2 \Rightarrow \TJK$, then $\muJK S $ restricts to $ \mu^\TT S$ on $T^2 S$.
But this cannot be well-defined, as the following example shows: 

\begin{ex}\label{ex: deg wheel} As usual, let $\W$ be the wheel graph with one vertex $v$ and edges $ \{a, \tau a \}$. 
	Its category of elements $\elG[\W]$ has skeletal subcategory 
	\begin{equation}
	\label{eq: el W}
		\xymatrix{
		(\shortmid) \ar[dr]_{ch_a} \ar@<2pt>[rr]^{ch_1^{\C_\two}} \ar@<-2pt>[rr]_{ ch^{\C_\two}_{1} \circ \tau} && \C_\two \ar[dl]^{ (1_{\C_\two} \mapsto a)} \\
		&\W& }	
	\end{equation}
	
So, if $S$ is a $(\CCC, \omega)$-coloured graphical species and $c \in \CCC$, then there is a (degenerate) $\W$-shaped graph of $S$-structured graphs $\Gdg_{S,c}$ given by
	\begin{equation}
\label{eq: deg graph of graphs}
\xymatrix{
\elG[\W]\ar[dd]_{\Gdg_{S,c}} &&	(\shortmid) \ar@{|->}[dd]_{\Gdg_{S,c}(ch_a)} \ar@<4pt>[rrr]^{ch_1^{\C_\two}} \ar@<-4pt>[rrr]_{ ch^{\C_\two}_{1} \circ \tau} &&& \C_\two \ar@{|->}[dd]^{\Gdg_{S,c}(1_{\C_\two} \mapsto a) } \\\\
\Grs &&(\shortmid, c) \ar@<4pt>[rrr]^{\Gdg_{S,c}(ch_1^{\C_\two})} \ar@<-4pt>[rrr]_{ \Gdg_{S,c}(ch^{\C_\two}_{1} \circ \tau)} &&& (\shortmid, c).}
\end{equation}
In particular, $ \Gdg_{S,c}(ch_1^{\C_\two}) = id_{(\shortmid,c)} = \Gdg_{S,c}(ch^{\C_\two}_{1} \circ \tau)$ and hence 
$\Gdg_{S,c}$ has colimit 
 $id_c\colon (\shortmid,c) \to (\shortmid,c)$ in $\Grs $. 	

We observe immediately that $E_0 (\shortmid) \neq E_0(\W)$ so \cref{graph of graphs edges} does not hold for ${\Gdg_{S,c}}$.

Moreover, if 
$\Gdg_{S,\omega c}$ is the $\W$-shaped graph of $S$-structured graphs given by 
\[ \left ( \shortmid, ch_a \mapsto (\shortmid, \omega c) \right ) \ \text {and } \left ( \C_\two, (1_{\C_\two} \mapsto a) \right ) \longmapsto (\shortmid, \omega c), \] and if $\tauG[\mathcal W]\colon \W \to \W$ is the unique non-trivial (but trivially boundary-preserving) automorphism, then ${\Gdg_{S,\omega c}}( \C,b) = {\Gdg_{S, c}}( \C,\tauG[\mathcal W]\circ b )$ for all $(\C,b)\in \elG[\W]$. 

Hence ${\Gdg_{S,c}}$ and ${\Gdg_{S, \omega c}}$
represent the same element of ${\TJK}^2S_{\nul}$. But ${\Gdg_{S,c}}$ has colimit $(\shortmid, c)$ in $\Grs $ while ${\Gdg_{S, \omega c}}$ has colimit $(\shortmid, \omega c) \in \Grs $, and these are distinct if $c \neq \omega c$.
 
\end{ex}

As \cref{ex: deg wheel} shows, taking colimits in $\XGrJK$ of degenerate graphs of $S$-structured graphs does not always lead to a well-defined class of $S$-structured graphs, let alone one in the correct arity. The issue arises because the coequaliser in $\GrShape$ of the parallel morphisms $id_{(\shortmid)}, \tau\colon (\shortmid) \rightrightarrows (\shortmid)$ is the exceptional loop $\bigcirc$, which is not a graph (\cref{deg loop}). 
 
An obvious first attempt at resolving the problem outlined in \cref{ex: deg wheel}, in order to extend $\mu^\TT S$ to a well-defined multiplication $ \muJK S\colon {\TJK}^2 S \rightarrow \TJK S$, is to enlarge $\Gr$ to include the exceptional loop $\bigcirc$: 

%
%


Let $\Gr^\bigcirc$ be the category of \emph{fully generalised Feynman graphs} and \'etale morphisms, obtained from $\Gr$ by adding the object $\bigcirc$ and a unique morphism $(\shortmid) \to \bigcirc.$ 

%

By definition, $\bigcirc$ is the formal coequaliser of the diagram $id, \tau\colon (\shortmid) \rightrightarrows (\shortmid)$ in $\fisinv \subset \Gr$. So, we may define its category of elements 
$\elG[\bigcirc] \defeq \fisinv \ov \bigcirc$, and thereby extend any graphical species $S\colon {\fisinv}^\mathrm{op} \to \Set$ to a presheaf on $\Gr^\bigcirc$ according to $ \G \longmapsto \mathrm{lim}_{\elG} S $. But this would imply that $S(\bigcirc) \cong S(\shortmid)$ for all graphical species $S$.

 It follows that $\Gr^\bigcirc $ does not embed densely or fully in $ \GS$. 
  In particular, there is no monad $\MM$ on $\GS$ with arities $\Gr^\bigcirc$ (see \cref{sec: Weber}).
  
In particular, let $X\mathdash\mathsf{Gr}_{\mathsf{iso}}^\bigcirc$ be the groupoid of \textit{fully generalised connected $X$-graphs} defined by 
\[ X\mathdash\mathsf{Gr}_{\mathsf{iso}}^\bigcirc \defeq  \XGrJK \text { for } X \not \cong \nul \ \text{ and } \ \nul\mathdash\mathsf{Gr}_{\mathsf{iso}}^\bigcirc \defeq  \XGrJK[\nul]  \amalg \{ \bigcirc\},\] 
and let the endofunctor $T^\bigcirc\colon \GS \to \GS$, such that $T^\bigcirc S_X = \TJK S_X$ for $X \not \cong \nul$, be given by \[\label{Tfullmoon defn}\begin{array}{ll} T^\bigcirc S_\S&\defeq S_\S,\\
T^\bigcirc S_{X} &\defeq \mathrm{colim}_{(\G, \rho) \in X\mathdash\mathsf{Gr}_{\mathsf{iso}}^\bigcirc} S(\G).\end{array}\]

Then, the $\W$-shaped graphs of graphs $\Gdg_{S,c}$ and $\Gdg_{S, \omega c}$ described in \cref{ex: deg wheel}
represent the same element $[\W, \beta] \in {T^\bigcirc}^2S_{\nul}$.
\label{no degenerate multiplication}
But, $S(\bigcirc) \cong S(\shortmid) = S_\S$ and so \[ [\bigcirc, c]\neq [\bigcirc, \omega c] \in T^\bigcirc S_{\nul} \text { whenever } c \neq \omega c \in S_\S\] 
It follows that $\mu^\TT $ cannot be extended to a multiplication $\mu^\bigcirc\colon {T^\bigcirc}^2 \Rightarrow T^\bigcirc $.

	Indeed, this is not surprising: For all $c \in S_\S$, the contraction $\zeta_c \colon S_{c,\omega c} \to S_\nul$ factors through the quotient $\widetilde {S_\two}$ of $ S_\two $ under the action of $Aut (\two)$. Hence, $\zeta (\phi)$ loses data relative to $\phi \in S_{c, \omega c}$, and the morphism $ (\shortmid) \to \bigcirc$ in $\Gr^\bigcirc$ -- that collapses a two-element set to a point -- would seem to be \textit{in the wrong direction}! 
	
	\begin{rmk}\label{rmk: standard solution}

		In the graphical formalism of \cite{GK98,BM08} described in \cref{ex: BM graphs} (as well as in, for example \cite{HRY15, MMS09, JY15}), where graph ports are defined to be the fixed points of edge involution, the graph substitution is not defined in terms of a functorial construction, but by `removing neighbourhoods of vertices and gluing in graphs'. Therefore, the exceptional loop arises from substitution as in \cref{construct fullmoon}.
		\begin{figure}[htb!] 
			\begin{tikzpicture}[scale = 0.3]
				
				\draw (0,0) circle (2 cm);
				\draw [dashed] (2,0) circle (1 cm);
				\filldraw (2,0) circle (3pt);
				
				
				\node at (5,-2.5) { 
					\begin{tikzpicture}[scale = 0.5]
						\draw [dashed] (0,0) circle (1 cm);
						\draw [ultra thick] (0,-1)--(0,1);\end{tikzpicture}
				};
				\draw [thick, ->]
				(4,-1.5) --(3,-0.8);
				
				\node at (8,0) {\Large $\longrightarrow$};
				
				\node at (14,0) {\begin{tikzpicture}[scale = 0.5]
						\draw (0,0) circle (2 cm);
						\draw [dashed] (2,0) circle (1 cm);
						
						\draw[ultra thick] ([shift=(-30:2cm)](0,0)arc (-30:30:2cm); \end{tikzpicture}};
				
			\end{tikzpicture}
			\caption{Constructing an exceptional loop by removing a vertex and substituting the stick graph.}
			\label{construct fullmoon}
		\end{figure}
	\end{rmk}

	
	 \begin{rmk}\label{rmk: HRY solution}
		Hackney, Robertson and Yau \cite[Definition~1.1]{HRY19b} are able to construct the modular operad monad on $\GS$, within the framework of Feynman graphs, by including extra \textit{boundary} data in their definition of graphs. For them, a graph is a pair {$(\G, \eth(\G))$} of a Feynman graph $\G$, and subset $\eth(\G) \subset E_0$ of ports, that satisfies certain conditions. In their formalism, $(\shortmid, \two) $ and $ (\shortmid, \nul)$ are different graphs:  $(\shortmid, \two) $ is the stick, and $ (\shortmid, \nul)$ plays the the role of the exceptional loop $\bigcirc$.
		
		However, this approach does not result in a dense functor from the graph category to $\GS$. And, though they construct a fully faithful nerve for modular operads in terms of a dense subcategory $U\hookrightarrow \CSM$ of graphs, the inclusion is not fully faithful, and so $U$ does not fully describe the graphical combinatorics of modular operads. 
		
	\end{rmk}
 
 The combinatorics of contracted units are examined more closely in \cref{s. Unital} where the \textit{problem of loops} discussed in this section will be resolved by adjoining a map that acts as a formal \textit{equaliser}, rather than a coequaliser, of $id, \tau\colon (\shortmid) \rightrightarrows (\shortmid)$ (see also Figures \ref{fig. lim and colim} and \ref{fig: contracting units}).

 \begin{rmk}\label{rmk: construction unique}
 	To my knowledge, the construction that I present in \cref{s. Unital} is unique among graphical descriptions of unital modular operads (or wheeled prop(erad)s), in that all others include some version of the exceptional loop as a graph. (See e.g.\ \cite{MMS09,Mer10, HRY15, JY15, BB17}.) 
 \end{rmk}
	

	\section{Modular operads with unit}\label{s. Unital}

\cref{unpointed modular operad} identifies the category of non-unital modular operads with the EM category of algebras for the monad $\TT$ on $\GS$. The goal of this section is to extend this in order to obtain (unital) modular operads. Some potential obstacles have been discussed in \cref{degenerate}, where it was also explained why the `obvious' modification of the operad monad (Examples \ref{ex: operad units} and \ref{ex: operad endo}) does not work for unital modular operads.

This section begins by returning to the definition of modular operads in \cref{defn: Modular operad} and looking in more detail at the combinatorics of (contracted) units. This combinatorial information can be encoded in a monad $\DD = (D, \mu^\DD, \eta^\DD)$ on $\GS$. 

Once $\DD$ is defined, it is a small step to obtaining the distributive law $\lambda\colon TD \Rightarrow DT$, whose construction provides us 
%
with an explicit description of the modular operad monad $\DD\TT$ in terms of equivalence classes of graphs structured by graphical species. Moreover, as discussed in \cref{dist for CSM}, the construction of $\DD\TT$ is such that it is always possible to work with nice (non-degenerate) representatives of these classes, thereby avoiding the problem of loops described in \cref{degenerate}.
%

\subsection{{Pointed} graphical species} \label{pointed graphical species}
By definition, if $(S, \diamond, \zeta, \epsilon)$ is a 
modular operad, then the unit $ \epsilon\colon S_\S \to S_\two$ is an injective map such that 
\begin{equation} \label{eq. unit 1} \epsilon \circ \ S_\tau \ = \ S(\sigma_\two) \circ \epsilon. \end{equation} 

The key observation is that the combination of a unit and a contraction implies that, as well as the unit elements in arity $\two$ provided by $\epsilon \colon S_\S \to S_\two$, modular operads also have distinguished elements 
in arity $\nul$. 
Namely, as in (\ref{eq: unit contraction}), there is a contracted unit map $o = \zeta \circ \epsilon \colon S_\S \to S_\nul$ that satisfies 
\begin{equation} \label{eq. unit 2} o = o \circ S_\tau \colon S_\S \to S_\nul. \end{equation} 

\begin{defn}\label{GSp}
	Objects of the category $\GSp$ of \emph{pointed graphical species} are triples $S_* = (S, \epsilon, o)$ (or $ (S, \epsilon^S, o^S)$) where $S$ is a graphical species and 
	$\epsilon\colon S_\S \to S_\two$, and 
	$o\colon S_\S \to S_\nul$ are maps satisfying conditions (\ref{eq. unit 1}) and (\ref{eq. unit 2}) above. Morphisms in $\GSp$ are morphisms in $\GS$ that preserve the additional structure.
	
\end{defn}

\begin{ex}\label{terminal species pointed}
For any palette $(\CCC, \omega)$, the terminal $(\CCC, \omega)$-coloured graphical species $\CComm{}$ is trivially pointed and hence terminal in the category 
	of \emph{$(\CCC, \omega)$-coloured pointed graphical species} and palette-preserving morphisms. 
\end{ex}

The category $\GSp$ is also a presheaf category:
%
Let $\fisinvp$ be the category obtained from $\fisinv$ by formally adjoining morphisms
$u \colon \two \to \S$ and $ z\colon \nul \to \S$, subject to the relations
\begin{itemize}
	\item 
	$u \circ ch_1 = id \in\fisinv (\S, \S)$ and $ u \circ ch_2= \tau \in \fisinv (\S, \S)$,
	\item $\tau \circ u = u \circ \sigma_{\mathbf 2}\in \fisinv (\two, \S),$ 
	\item 
	$z = \tau \circ z \in \fisinv (\nul, \S)$.
	\label{relations}
\end{itemize}

\begin{lem}\label{lem: elp presheaves}
	The following are equivalent:
	\begin{enumerate}
		\item $S_*$ is a presheaf on $ \fisinvp$ that restricts to a graphical species $S$ on $ \fisinv$; 
		\item $(S, \epsilon, o)$, with $\epsilon = S_*(u)$ and $o = S_*(z)$, is a pointed graphical species.
	\end{enumerate}
	\label{CGSp}
\end{lem}
\begin{proof}
It is easy to check directly that $\fisinvp$ is completely described by
\begin{itemize}
	\item $\fisinvp(\S, \S) = \fisinv (\S, \S)$ and $\fisinvp (Y,X) = \fisinv(Y,X)$ whenever $Y \not \cong \nul$ and $ Y \not \cong \two$,
	\item $\fisinvp (\nul, \S) = \{z\}$, and $\fisinvp (\nul, X) = \fisinv (\nul, X) \amalg \{ ch_x\circ z\}_{x \in X}$,
	\item $\fisinvp (\two, \S) = \{u, \tau \circ u\}$, and $\fisinvp (\two, X) = \fisinv (\nul, X) \amalg \{ ch_x\circ u, ch_x \circ \tau \circ u\}_{x \in X}$ for all finite sets $X$,
\end{itemize}
and the lemma follows immediately.
\end{proof}

By \cref{lem: elp presheaves}, a pointed graphical species $(S, \epsilon, o)$ may also be denoted by $ S_*$, and these forms will be used interchangeably. The category of elements of a pointed graphical species $S_*$ will be denoted by $\elpG[S_*] \defeq \ElP{S_*}{\fisinvp}$.


\begin{lem}\label{forget monadic} The forgetful functor $ \GSp \to \GS$ is strictly monadic: it has a left adjoint $\GS \to \GSp$, and $\GSp$ is the EM category of algebras for the induced monad $\DD = (D ,\mu^\DD,\eta^\DD)$ on $\GS$.
\end{lem}

\begin{proof}
The left adjoint $(\cdot)^+$ to the forgetful functor $ \GSp \to \GS$ takes a graphical species $S$ to its left Kan extension $S^+ $ along the inclusion $({\fisinv})^\mathrm{op} \hookrightarrow({\fisinvp})^\mathrm{op}$. 
This does nothing more than formally adjoin elements $\{\epsilon^{+}_c \} _{c \in S_\S}$ to $S_ \two$ and $\{o^{+}_{\tilde c}\}_{\tilde c \in \widetilde {S_\S}}$ to $S_\nul$ according to the combinatorics of contracted units (\ref{eq. unit 1}), (\ref{eq. unit 2}). Hence, $S^+$ is described by $(DS, \epsilon^+, o^+) = (DS \epsilon^{DS}, o^{DS}) $ where $DS_\two = S_\two \amalg \{\epsilon^{+}_c \} _{c \in S_\S}$, $DS_\nul = S_\nul \amalg \{o^{+}_{\tilde c}\}_{\tilde c \in \widetilde {S_\S}}$, and $DS_X  = S_X$ for $X \not \cong \two, X \not \cong \nul$.  

The monadic unit $\eta^\DD$ is provided by the inclusion $S \hookrightarrow DS$, and the multiplication $\mu^\DD$ is induced by the canonical projections $D^2S_\two \to DS_\two$. \label{+ functor}
\end{proof}


%
%
%

\subsection{Pointed graphs}\label{grp cat}

Let $ \Grp$ be the category obtained in the bo-ff factorisation of $(\yet -)^+\colon \Gr \hookrightarrow \GS \to \GSp$, so that the following diagram commutes:
 
\begin{equation} \label{defining Grp}
\xymatrix{ 
	\fisinvp \ar[rr]^-{\text{ dense}}_-{\text{ f.f.} }	\ar@/^2.0pc/[rrrr]_-{\text{ f.f.} }			&& \Grp \ar [rr]_-{\text{ f.f.} }^-{\yetp}				&& \GSp \ar@<2pt>[d]^-{\text{forget}}\\ 
	\fisinv \ar@{^{(}->} [rr]\ar[u]^{\text{b.o.}}	&& \Gr \ar@{^{(}->} [rr]_-{\yet} \ar[u]^{\text{b.o.}}	&& \GS. \ar@<2pt>[u]^-{(\cdot)^+	}}\end{equation}

The inclusion $\fisinvp \to \Grp$ is fully faithful (by uniqueness of bo-ff factorisation), and also dense, since the induced nerve $\yetp\colon \Grp \to \GSp$ is fully faithful by construction.

Let $\G\in \Gr$ be a graph. By \cref{lem: elp presheaves}, for each edge $e \in E$, $ \epsilon^\G_e  =  ch_e \circ u \in \Grp (\C_\two, \G)$ is the $ch_e$-coloured unit for $ \yetp \G$, and the corresponding contracted unit is given by $o^\G_{\tilde e} = ch_e \circ z \in \Grp (\C_\nul, \G)$. 


Since the functor $\yetp$ embeds $\Grp$ as a full subcategory of $\GSp$, I will denote $\yetp \G \in \GSp$ simply by $\G$ where there is no risk of confusion. In particular, the element category $\elpG[\yetp \G]$ is denoted by $\elpG$ and called the 
\emph{category of pointed elements of a graph $\G$}.

For all pointed graphical species $S_*$, the forgetful functor $\GSp \to \GS$ induces injective-on-objects inclusions $ \ElS[S]\to \elpG[S_*]$.   

Recall \cite[Section~IX.3]{Mac98} that a functor $\Psi\colon \CCat \to \DCat$ is \emph{final} if the slice category 
$d \ov \Psi\defeq \Psi^{\mathrm{op}}\ov d$ is non-empty and connected for all $d \in \DCat$, and that this is the case if and only if, for any functor $\Phi\colon \DCat \to \DCat[E]$ such that $\mathrm{colim}_{\CCat}(\Phi \circ\Psi)$ exists in $\DCat[E]$, $\mathrm{colim}_{\DCat}\Phi $ also exists in $\DCat[E]$ and the two colimits agree. 

\begin{lem}\label{lem: final}
	For all graphs $\G$, the inclusion $\elG \hookrightarrow \elpG$ is final. 
Therefore, $\fisinv$ is dense in $\Grp$ and, 
for all pointed graphical species $S_* = (S, \epsilon, o)$,
\[ S(\G) = \mathrm{lim}_{(\C, b) \in \elG}S(\C) = \mathrm{lim}_{(\C',b') \in \elpG}S_*(\C').\]
	
\end{lem}
\begin{proof} By definition, $\elpG$ is obtained from $ \elG$ by adjoining, for each $e \in E$, the objects $(\two, ch_e \circ u)$ and $ (\nul, ch_e \circ z) = (\nul, ch_{\tau e} \circ z)$ and the morphisms 
	\[ \xymatrix{(\two, ch_e \circ u) \ar[rr]^-{u}&& (\S, ch_e)&& \ar[ll]_-z (\nul, ch_e \circ z).
	} \]
Hence, for all $(\C, b) \in \elpG$, the slice category $ b \ov \elG$ is connected and non-empty. 
\end{proof}

By \cref{lem: final}, 
a morphism $f \in \Grp (\G, \G')$ is described by a functor $\elG \to \elpG[\G']$ such that, for each $(\C,b) \in \elG$, $(\C,b)\mapsto (\C, f (b))$, and there is a commuting diagram
\[ \xymatrix{ \C \ar[rr]^-{g_{f(b)}} \ar[dr]_{f (b)} &&\C'\ar[dl]^{b'} \\
	&\G'&}\] where $g_{f(b)} \in \fisinvp(\C,\C')$ and $(\C', b') \in \elG[\G']$ is an (unpointed) element of $\G'$.

\begin{ex}\label{loop collapse} (Compare \cref{ex: deg wheel}.) A surprising consequence of the definitions is that the morphism set $\Grp(\W, \shortmid)$ is non-empty. 
	There are two morphisms $\kappa, \tau \circ \kappa \in \Grp (\W, \shortmid)$: 
	
	\begin{minipage}[t]{0.5\textwidth}
		\begin{equation} \label{epsilon}
		\xymatrix{
			& \W & \\
			(\shortmid) \ar[ur]^{ch_a} \ar@<-2pt>[rr]_{ch_2\circ \tau} \ar@<2pt>[rr]^{ ch_{1} } \ar@{=}[dr] && \C_\two \ar[ul]_{ 1_{\C_\two} \mapsto a} \ar[dl]^{u}\\
			& (\shortmid)& }\end{equation}
		
	\end{minipage}
	\begin{minipage}[t]{0.5\textwidth}
		\begin{equation} \label{tau epsilon}
		\xymatrix{
			& \W & \\
			(\shortmid) \ar[ur]^{ch_a} \ar@<-2pt>[rr]_{ch_2\circ \tau} \ar@<2pt>[rr]^{ ch_1 } \ar@{=}[d] && \C_\two \ar[ul]_{ 1_{\C_\two} \mapsto a} \ar[d]^{\sigma_\two}\\
			(\shortmid) \ar@<-2pt>[rr]_{ch_1\circ \tau} \ar@<2pt>[rr]^{ ch_2} \ar@{=}[dr] && \C_\two \ar[dl]^{u}\\
			& (\shortmid).& }\end{equation}
		
	\end{minipage}
	
	
	Hence, $\Grp(\W, \shortmid) \cong \Grp(\shortmid, \shortmid) \cong \Grp(\W, \W)$. In particular, for all graphs $\G \not \cong \W$, 
	\[\Grp (\W, \G) \cong E(\G) \ \text{ by } \ ch_e \circ \kappa \mapsto e.\]
	
	These morphisms play a crucial role in the proof of the nerve theorem, \cref{nerve theorem}.
\end{ex}


Now, let $W \subset V_2$ be a subset of bivalent vertices of a connected graph $\G$. 

\begin{defn}\label{def: vertex deletion}
	A \emph{vertex deletion functor (for $W$)} is a $ \G$-shaped graph of graphs ${\Gdg}^\G_{\setminus W}\colon \elG \to \Grp $ such that for $(\CX, b) \in \elG$, 
	\[{\Gdg}^\G_{\setminus W} (b) = \left \{ \begin{array}{ll} 
	(\shortmid)& \text{ if } (\CX, b) \text{ is a neighbourhood of } v \in W,\\
	\CX &\text{ otherwise. } 
	\end{array} \right. \] 
	
	\label{vertex deletion morphism}
	If ${\Gdg}^\G_{\setminus W}$ admits a colimit $\Gnov[W]$ in $\Grp$, then the induced morphism 
	$\delW \in \Grp(\G, \Gnov[W])$ is called the \emph{vertex deletion morphism corresponding to $W$.}
\end{defn}
Note, in particular, that a vertex deletion functor ${\Gdg}^\G_{\setminus W}$ is non-degenerate if and only if $ W = \emptyset$ in which case ${\Gdg}^\G_{\setminus W}$ is the identity graph of graphs $\Gid \colon (\C,b) \mapsto \C$ (\cref{subs: gluing}).

Moreover, if $W = W_1 \amalg W_2$ and $ \delW =\delW^\G  \colon \G \to \Gnov$ exists in $\Grp$, then so do 
\[\delW[W_1]^\G\colon \G\to \Gnov[W_1] \ \text{ and } \ \delW[W_2]^{\Gnov[W_1]}\colon \Gnov[W_1]^\G \to (\Gnov[W_1])_{\setminus W_2} = \Gnov[W]^\G\]
and $\delW = \delW[W_2]^{\Gnov[W_1]}\circ \delW[W_1]^\G$.

%

\begin{ex}\label{line deletion} For $\G = \C_\two$ and $ W = V = \{*\}$, ${\Gdg}^\G_{\setminus W}$ is the constant functor induced by the cocone of $ \elG[\C_\two]$ over $(\shortmid)$ in $\Grp$: 
	\begin{equation}
	\qquad \xymatrix{ \elpG[\C_\two] \ar[d]& \simeq &
		(\shortmid) \ar[rr]^-{ch_1}\ar[drr]_-{id_{(\shortmid)}} && \C_\two \ar[d]^{u} && (\shortmid) \ar[ll]_-{ch_2 \circ \tau} \ar[dll]^{id_{(\shortmid)}} \\
		\elG[\shortmid] && &&(\shortmid)&&.}
	\end{equation}
	
	So, ${\Gdg}^\G_{\setminus W}$ has colimit $(\shortmid) $ in $\Grp$ and $\delW = u \in \Grp (\C_\two, \shortmid)$.
	
	In fact, for all $k \geq 0$, if $\G = {\Lk}$ and $ W = V$, then ${\Gdg}^\G_{\setminus W}$ is also the constant functor to $\Gnov = (\shortmid)$, and $u^k \defeq \delW\colon \Lk \to \Grp$ is induced by the $\Grp$-cocone under $\elpG[\Lk]$:
	\[
	\xymatrix{ \elpG[\Lk] \ar[d] & \simeq &
		(\shortmid) \ar[rr]^-{ch_1}\ar[drr]_-{id_{(\shortmid)}} && \C_\two \ar[d]^{u} &&{ \text{ \dots }} \ar[ll]_-{ch_2 \circ \tau} \ar[rr]^-{ch_1}
		&& \C_\two \ar[d]^-{u}&& \ar[ll]_-{ch_2 \circ \tau}\ar[dll]^{id_{(\shortmid)}} (\shortmid)\\
		\elG[\shortmid] && &&(\shortmid)\ar@{=}[r]&&{ \text{ \dots }} &&\ar@{=}[l](\shortmid).&&}\]
	(So $u^1 = u \colon \C_\two \to (\shortmid)$ and $u^0$ is just the identity on $(\shortmid)$.)

\end{ex}

For any graph $\G$, a pointwise \'etale injection $\iota \in \Gr (\Lk, \G)$ describes a subset $V(\Lk) = W \subset V_2(\G)$ of bivalent vertices of $\G$. Hence, $ \delW \in \Grp(\G, \Gnov)$ exists in $\Grp$ and there is an edge $ e_{\setminus W} = \delW(\iota(1_{\Lk})) \in E(\Gnov)$ so that the following diagram commutes:
\begin{equation}\label{eq: delete lines} \xymatrix{ \Lk \ar[rr]^-{\iota}\ar[d]_{u^k} && \G \ar[d]^-{\delW }\\
	(\shortmid) \ar[rr]_-{ch_{e_{\setminus W}}} && \Gnov,
}\end{equation}

\begin{ex}\label{wheel deletion}
	Let $*$ be the unique vertex of the wheel graph $\W$. By \cref{loop collapse}, $\W_{\setminus \{*\}} = \mathrm{colim}_{ \elG[\W]} {\Gdg}^\W_{\setminus \{*\}} $ exists and is isomorphic to $(\shortmid)$ in $ \Gr$. (See also \cref{degenerate}.) The induced morphism $\delW[\{*\}]$ is precisely $\kappa\colon \W \to (\shortmid)$.

For $m \geq 1$, let $W \subset V(\Wm)$ be the image of $V(\Lk[m-1])$ under an \'etale morphism $\iota \in \Gr(\Lk[m-1], \Wl)$. So $V(\Wm) = W \amalg \{*\}$, and by (\ref{eq: delete lines}), $\iota$ induces a vertex deletion morphism $\delW \in \Grp (\Wm, \W)$. Therefore $ \kappa^m \defeq \delW[V(\Wl)]$ is given by the composite $\kappa ^m =  \kappa \circ \delW \colon \Wl \to \W \to (\shortmid)$.

	In particular, for all $m \geq 1$, there are precisely two distinct morphisms, $\kappa^m $ and $\tau \circ \kappa^m$, in $\Grp (\Wm, \shortmid)$. Hence, for all graphs $\G$, 
	\[\Grp (\Wl, \G) = \Gr(\Wl, \G)\amalg \{ch_{e} \circ \kappa^m\}_{ e\in E(\G)} .\] \end{ex}

\begin{prop}\label{Gnov construction}
	For all graphs $\G$ and all $W \subset V_2$, the colimit $\Gnov$ of ${\Gdg}^\G_{\setminus W}$ exists in $ \Grp$.
	
	Moreover, $E_0(\G) =E_0(\Gnov)$ unless $\G = \Wl$ and $W = V$ for some $m \geq 1$.
\end{prop}

\begin{proof}
	
	If $W$ is empty, then $ \Gnov = \G$ and $\delW $ is the identity on $\G$. On the other hand, if $W = V$ then, by \cref{bivalent graphs}, $ \G = {\Lk} $ or $\G = \Wl$ for some $k \geq 0$ or $m \geq 1$, and so $\Gnov = (\shortmid) $ by  Examples \ref{line deletion} and \ref{wheel deletion}. For $\G = \Lk$, the vertex deletion morphism $u^k\colon \Lk \to (\shortmid)$ induces a bijection on boundaries, so the proposition is proved when $W = V$ or $W = \emptyset$.
	
	Assume therefore, that $\emptyset \neq W \subsetneq V$ is a proper, non-empty subset of (bivalent) vertices of $\G$.
	
	Let $ \G^W \subset \G$ be the subgraph with vertices $ V(\G^W) = W$, half-edges $H(\G^W) = \coprod_{v \in W}\vH[v]$ and whose edge set $E(\G^W) $ is the $\tau$-closure of $\coprod_{v \in W}\vE[v]$. (See \cref{fig: vertex deletion}.)
	
	By \cref{bivalent graphs}, $\G^W\cong \coprod_{i = 1}^m \Lk[k_i]$ is a disjoint union of line graphs, with $k_i \geq 1$ for all $i$: 

	\begin{equation}\label{eq: GW in G} 
	E_0(\G^W) = \coprod_{i = 1}^m \{1_{\Lk[k_i]}, 2_{\Lk[k_i]}\}, \quad \text{ and }\quad \left(\coprod_{i = 1}^m \{1_{\Lk[k_i]}\}\right)\cap \left(\coprod_{i = 1}^m \{2_{\Lk[k_i]}\}\right) = \emptyset \ \text{ in } E(\G).
	\end{equation} 
	
	The graph $\Gnov$ is obtained by applying $u^{k_i}\colon \Lk[k_i] \to (\shortmid)$ on each component $\Lk[k_i] \hookrightarrow \G$ in turn. Since $W \neq V$ and components of $\G^W$ are disjoint in $\G$, this is independent of the order of $\{ \Lk[k_i] \}_{i = 1}^m$, and hence $\delW \colon \G \to \Gnov$ exists in $\Grp$. The construction is summarised in the diagram
	\[ \xymatrix{\G^W \ar@{<->}[r]^-{\cong} &\coprod_{i = 1}^m \Lk[k_i] \ar[rr]^-{\iota}\ar[d]_{\coprod_{i = 1}^m u^{k_i} }&& \G \ar[d]^{\delW}\\
	&	\coprod_{i = 1}^m (\shortmid) \ar[rr]_-{ \coprod_i ch_{e_{\setminus W_i}} }&& \Gnov,}\]
	where, for $1 \leq i \leq m$, $e_i  = \delW (\iota (1_{\Lk[k_i]}))$ in $ E(\Gnov) $. 
%
	
When $W \neq V$, the graph $\Gnov$ (see \cref{fig: vertex deletion}) is described explicitly by:
	\[\ \Gnov = \ \xymatrix{
		*[r] { \Enov}\ar@(ul,dl)[]_{\taunov} && { \Hnov} \ar[ll]_-{\snov} \ar[rr]^-{\tnov}&& \Vnov,}\] where
	\[ \begin{array}{lll}\Vnov&=&V \setminus W,\\
	\Hnov &= &H \setminus H(\G^W) = H \setminus \left (\coprod_{v \in W} \vH \right),\\
	\Enov &
= & E\setminus \left (\coprod_{ v \in W} \vE \right).	\end{array}\]
The maps $\snov, \tnov$ are just the restrictions of $s$ and $t$ and the involution $\taunov\colon \Enov \to \Enov$ is given by 
	\[ \begin{array}{llll} \taunov (e) & = &\tau e & \text{ for } e \in E \setminus E(\G^W), \\
	\taunov (1_{\Lk[k_i]}) & = &2_{\Lk[k_i]} & \text{ for } 1 \leq i \leq m.\end{array}\] 
By (\ref{eq: GW in G}), this is fixed point free and induces an identity $E_0 (\G) = E_0 (\Gnov)$ on boundaries.
%
\end{proof}
	
	%
	
	\begin{figure}[htb!]
		\begin{tikzpicture}[scale = .3]

		\draw[->] (-15,0)--(-11,0);
		\node at (-13,.5){\scriptsize{$\delW$}};

		\node at (-25,0){\begin{tikzpicture}[scale = .25]
			
			\draw [thick, cyan]
			(-3,2).. controls (-2.5,1.5) and (-2.5,0.5).. (-3,0)
			(-3,0).. controls (-3.5,-1).. (-5,-1)
			(-5,3) --(-2,4)--(1,4)--(3,3)
			
			(1,-4)--(4, -3.5)--(5,-1)--(5,2)
			(-1,-5) -- (-3,-5)--(-5, -6)--(-6, -9)
			;
			
			\draw[blue, fill = cyan]
			(-3,0) circle (8pt)
			(4, -3.5) circle (8pt)
			(5,-1)circle (8pt)
			(5,2)circle (8pt)
			(-2,4)circle (8pt)
			(1,4)circle (8pt)
			(-3,-5)circle (8pt)(-5, -6)circle (8pt)(-5.5,-7.5)circle (8pt)
			;

			\draw[ thick] 	
			(-6,1).. controls (-6,2).. (-5,3)
			(-5,3).. controls (-4.5,3) and (-3.5,3).. (-3,2)
			(-5,-1).. controls (-6,-0) .. (-6,1)
			(-5,-1)-- (-3,2)
			(-3,2)--(-6,1)
			(-6,1)..controls (-7.5,1) and (-9, 0.5) ..(-11,0.5)
			(-1,-5)..controls (-0.5,-3.5) and (0.5, -3.5)..(1,-4)
			(-1,-5)..controls (-0.5,-5) and (0.5, -5)..(1,-4)
			(4,1)--(5,2)
			(5,2)--(3,3)
			(3,3)--(4,1)
			(3,3).. controls (4,3.5) and (5,2.5)..(5,2)
			(3,3).. controls (3,4) and(4,6) ..(8,6)
			(-3,2).. controls (1,-1.5) and (5,-1).. (5,2)
			;
			
			\filldraw [black]	
			(-6,1) circle (8pt)
			(-5,3) circle (8pt)
			(-3,2) circle (8pt)
			(-5,-1) circle (8pt)
			(1,-4) circle (8pt)
			(-1,-5) circle (8pt)
			(4,1) circle (8pt)
			(5,2) circle (8pt)
			(3,3) circle (8pt);
			
			\draw[black, fill = white]
			(-4,.5)circle (8pt)
			(4,1) circle (8pt);
			

			(0,6) .. controls (2,6) and (7,4) ..(7,0)
			(7,0) .. controls (7,-2) and (2,-8) ..(0,-8)--(0,0)--(-7.5,0);
			\end{tikzpicture}};
		
		\node at (-2,0){\begin{tikzpicture}[scale = .25]
			
			\draw [thick, red]
			(-3,2).. controls (-2.5,1.5) and (-2.5,0.5).. (-3,0)
			(-3,0).. controls (-3.5,-1).. (-5,-1)
			(-5,3)..controls (-2,4) and (1,4) .. (3,3)
			
			(1,-4)..controls (2, -4) and (6,-2).. (5,2)
			(-1,-5) .. controls (-2.5,-5) and (-5, -4.5) ..(-6, -9)
			;
			

			\draw[ thick] 	
			(-6,1).. controls (-6,2).. (-5,3)
			(-5,3).. controls (-4.5,3) and (-3.5,3).. (-3,2)
			(-5,-1).. controls (-6,-0) .. (-6,1)
			(-5,-1)-- (-3,2)
			(-3,2)--(-6,1)
			(-6,1)..controls (-7.5,1) and (-9, 0.5) ..(-11,0.5)
			(-1,-5)..controls (-0.5,-3.5) and (0.5, -3.5)..(1,-4)
			(-1,-5)..controls (-0.5,-5) and (0.5, -5)..(1,-4)
			(4,1)--(5,2)
			(5,2)--(3,3)
			(3,3)--(4,1)
			(3,3).. controls (4,3.5) and (5,2.5)..(5,2)
			(3,3).. controls (3,4) and(4,6) ..(8,6)
			(-3,2).. controls (1,-1.5) and (5,-1).. (5,2)
			;
			
			\filldraw [black]	
			(-6,1) circle (8pt)
			(-5,3) circle (8pt)
			(-3,2) circle (8pt)
			(-5,-1) circle (8pt)
			(1,-4) circle (8pt)
			(-1,-5) circle (8pt)
			(4,1) circle (8pt)
			(5,2) circle (8pt)
			(3,3) circle (8pt);
			
			\draw[black, fill = white]
			(-4,.5)circle (8pt)
			(4,1) circle (8pt);
			

			(0,6) .. controls (2,6) and (7,4) ..(7,0)
			(7,0) .. controls (7,-2) and (2,-8) ..(0,-8)--(0,0)--(-7.5,0);
			\end{tikzpicture}};

	\end{tikzpicture}
	\caption{Vertex deletion $\delW\colon \G \to \Gnov$, with colours indicating $\G^W \subset  \G$ and $W \subset V_2$, and $\coprod_{i = 1}^3 u^{k_i} (\G^W) \subset  \Gnov$. }
	\label{fig: vertex deletion}
\end{figure}

%
%
%
%


\begin{defn}
	\label{def: similar}
	The \emph{similarity category} $\Grsimp\hookrightarrow \Grp$ is the identity on objects subcategory of $\Grp$ whose morphisms are generated under composition by $z\colon \C_\nul \to (\shortmid)$, the vertex deletion morphisms, and graph isomorphisms. Morphisms in $\Grsimp$ are called \emph{similarity morphisms}, and connected components of $\Grsimp$ are \emph{similarity classes}. Graphs in the same connected component of $\Grsimp$ are \emph{similar}. 
\end{defn}

\begin{ex}
	\label{ex: morphisms to stick}\label{morphism sets *}
	Up to isomorphism, the only morphisms in $\Grp$ with codomain $(\shortmid)$ are similarity morphisms of the form $z \colon \C_\nul \to (\shortmid)$, 
	$\kappa^m  \colon \Wm \to ( \shortmid)$ ($m \geq 1$), and $ u^k \colon \Lk \to ( \shortmid)$ ($k \geq 0$). 
\end{ex}


%

\begin{cor}[Corollary to \cref{Gnov construction}]\label{pp* sim}
The pair $(\Grsimp, \Gr)$ of subcategories of $\Grp$ defines a weak 
	factorisation system on $\Grp$. 
	

In particular, if $E_0(\G) \neq \emptyset$ and $f \in \Grp (\G, \G')$ is boundary-preserving, then $f =  f_{\setminus {W_f}} \circ \delW[W_f]$ where $f_{\setminus {W_f}} \in \Grp ( \G_{\setminus {W_f}}, \G')$ is an isomorphism.

\end{cor}

\begin{proof}
	The only non-identity morphisms in $\Grp$ with (co)domain $\C_\nul$ are of the form $ ch_e \circ z   = ch_{\tau e} \circ z \colon \C_\nul \to \G$ for some graph $\G$ with edge $e $, and $z$ has the left lifting property with respect to morphisms in $\Grp$. Moreover, any morphism $f \in \Grp (\G, \G')$ between connected graphs $\G \not \cong \C_\nul$ and $\G' \not \cong \C_\nul$ factors uniquely as $f =  f_{\setminus {W_f}} \circ \delW[W_f]$, where $W_f$ is the set of bivalent vertices $w$ of $\G$ such that, if $(\C_\two, b)$ is a minimal neighbourhood of $w$, then $f \circ b = ch_{e'} \circ u\colon \C_\two \to \G'$ for some (necessarily unique) edge $e'$ of $\G'$. Hence $(\Grsimp, \Gr)$ describes a weak 
	factorisation system on $\Grp$.

The second statement follows immediately from \cref{pp iso}.
\end{proof}


\begin{ex}\label{ex. DS}
	For all graphical species $S$ and all graphs $\G$ with no isolated vertices, by \cref{pp* sim}, 
	\begin{equation} \label{eq. D on graphs}
		DS(\G) = \coprod_{W \subset V_2} S(\Gnov).
	\end{equation} 
\end{ex}

\begin{ex}
	\label{Lk morphisms}
	In particular, for all graphs $\G $ and all $k \in \N$, (\ref{eq. D on graphs}) gives
	\[\Gr_*({\Lk}, \G)\cong \coprod_{j =0}^k \left (k \atop j \right)\Gr(\Lk[j], \G).\]
\end{ex}




By 
 \cref{pp* sim}, a morphism 
%
$f \in \Grp(\G, \G')$ is uniquely characterised by a commuting diagram of the form (\ref{starmorphism}), and such that $ {\mathfrak f}_V^{-1}(\widetilde{E'}) \subset V_0\amalg V_2$ is either a single isolated vertex or a (possibly empty) subset of bivalent vertices, and the induced square \ref{starpullback} is a pullback.

\begin{minipage}{.6\textwidth}
	
	\begin{equation}
	\label{starmorphism}
	\xymatrix{
		E \ar@{<->}[r]^-\tau\ar[d]_{{\mathfrak f}_E}&E \ar[d]_{{\mathfrak f}_E}& H \ar[l]_-s\ar[d]_{{\mathfrak f}_H} \ar[r]^-t& V\ar[d]^{{\mathfrak f}_V}\\
		E' \ar@{<->}[r]_-{\tau'}&E'& H' \amalg {E'} \ar[l]^-{s'\amalg id' } \ar[r]_-{t' \amalg {q'}}& V' \amalg \widetilde{E'} }\end{equation}
\end{minipage}
\begin{minipage}{.4\textwidth}
	
	\begin{equation}
	\label{starpullback}
	\xymatrix{
		H \ar[d]_{{\mathfrak f}_H} \ar[r]^-t& (V \setminus V_0)\ar[d]^{{\mathfrak f}_V}\\
		H' \amalg {E'} \ar[r]_-{t' \amalg {q'}}& (V \setminus V_0) \amalg \widetilde{E'} }\end{equation} \end{minipage}

If $\G \not \cong \C_\nul$, and $f = f_{\setminus W_f}\circ \delW[W_f]$ where $f_{\setminus W_f}$ is a morphism in $\Gr$, then $ W_f = \mathfrak f_V^{-1}(\widetilde {E'}) \subset V_2$.
\begin{ex}\label{ex: starmorphism} The morphisms $z \in \fisinvp(\nul, \S) = \Grp(\C_\nul, \shortmid)$ and $u \in \fisinvp(\two, \S) = \Grp(\C_\two, \shortmid)$ are described by commuting diagrams (\ref{z diagram}) and (\ref{u diagram}):
	

\begin{equation} \label{z diagram}
\xymatrix{ \C_\nul \ar[d]_{z} &&
		*[r] {\emptyset \ }
		\ar[d]&
		{\emptyset }\ar[l]\ar[d] \ar[r] &\{*\}\ar[d]\\
		(\shortmid)	&& 
	*[r] { \{1,2\} \ }
		\ar@(lu,ld)[]_{\tau} &
		{\{1,2\} }\ar[l]^-{id} \ar[r]_-{q_\shortmid}&{\{\tilde 1\}}}\end{equation}

\begin{equation} \label{u diagram}
	\xymatrix {C_\two \ar[d]_u &&	*[r]{ \tiny{ 
						\left \{ 
						\begin{array}{ll}
							1_{\C_\two}, & 2_{\C_\two},\\
							1^\dagger_{\C_\two}, & 2^\dagger_{\C_\two} \end{array}\right \} }  }
		\ar[d]&&
	\{	{1^\dagger_{\C_\two },2^\dagger_{\C_\two } }\}\ar[ll]\ar[d] \ar[rr] &&\{*\}\ar[d]\\
			(\shortmid)	&& 
		*[r] { \{1,2\} \ }
		\ar@(lu,ld)[]_{\tau} &&
		{\{1,2\} }\ar[ll]^-{id} \ar[rr]_-{q_\shortmid}&&{\{\tilde 1\}}	
}
\end{equation}

%

\end{ex}

The following extension of \cref{all fE} says that most morphisms in $\Grp$ are completely determined by their action on edges: 

\begin{lem}
\label{allEV}
If $\G \not \cong \C_\nul$ and $\G' \not \cong \W$, then $\mathfrak f_E$ is sufficient to define $f \in \Grp(\G, \G')$.
\end{lem}

\begin{proof} 
	Let $v \in V_2$ be a bivalent vertex of $\G$ with incident edges $\vE = \{e_1, e_2\} \subset E_2$. If $ \mathfrak f_E(e_1) \neq \mathfrak f_E(\tau e_2)$, then ${\mathfrak f}_V(v) = t' {s'}^{-1}({\mathfrak f}_E (e_1))\in V'.$ Otherwise, $ \mathfrak f_E(e_1) = \mathfrak f_E(\tau e_2)$. Then, either 
	\[\mathfrak f_V(v) = q'(\mathfrak f_E(e_1)) = q'(\mathfrak f_E(\tau e_2)) \in \widetilde{E'},\] or, there is a vertex $v'$ of $\G'$ with $\vE[v'] = \{ {\mathfrak f}_E (e_1), {\mathfrak f}_E (e_2)\}, $ in which case $\G' = \W$. 
\end{proof}

\begin{ex} \cref{allEV} does not hold if $\G' = \W$. For example, there are only two maps of edges $ E(\Wl[2]) \to E(\W) $ that are compatible with the involution, and these correspond to the two morphisms in $\Gr(\Wl[2], \W)$. However, there are six distinct morphisms in $\Grp(\Wl[2], \W)$.
\end{ex}

\subsection{$S_*$-structured graphs}

The \'etale topology on $\Gr$ extends to a topology on $\Grp$ whose covers at $\G$ are jointly surjective collections $ \mathfrak U \subset \Grp \ov \G$. By \cref{lem: final}, a presheaf $P \colon \Grp^{\mathrm{op}} \to \Set$ is a sheaf for this topology if and only if, for all graphs $\G$,
$P(\G) \cong \mathrm{lim}_{ (\C,b) \in \elG}P(\C)$. 
In particular, there is a canonical equivalence $\sh{\Grp, J_*} \simeq \GSp$ (compare \cref{subs. sheaves}). 
%
	
	Let $S_* $ be a pointed graphical species. 


\begin{defn}\label{S*-graph} (Compare \cref{S-graph}.) 
An \emph{$S_*$ structure on a connected graph $\G$} is an element $\alpha\in S_*(\G) \cong \GSp (\G, S_*)$. The category of (connected) $S_*$-structured graphs is denoted by $\ovP{S_*}{\Grp}$. 

An $S_*$-structured graph $(\G, \alpha)$ is called \emph{admissible} if $\G \not \cong (\shortmid)$ is not a stick graph.
\end{defn} 

\begin{ex}\label{identity graphs}
\label{wheel and line structure notation} For all pointed graphical species $S_* = (S, \epsilon, o)$, all $k \geq 0$, $ m \geq 1$, the vertex deletion morphisms $u^k \in \Grp ( \Lk ,\shortmid)$ and $\kappa^m \in \Grp( \Wm ,\shortmid)$ induce injective maps 
\[S_*(u^k)\colon S_\S = S(\shortmid) \to S(\Lk), \ \text{ and } \ 
S_*(\kappa ^m)\colon S_\S\to S(\Wm).\]

For each $c \in S_\S$, there are \textit{$c$-coloured unit structures} on $\Lk$ and $\Wm$ (as pictured in \cref{fig: unit graphs}):
\[ \Lk(\epsilon_c)\defeq S_*(u^k)(c) \in S_*(\Lk) \ \text{ and } \ \Wm(\epsilon_c)\defeq S_*(\kappa ^m)(c) \in S_*(\Wm).\] 

\end{ex}

\begin{figure}[htb!] 
\begin{tikzpicture}
\node at (0,0){
	\begin{tikzpicture}[scale = 0.6]
	\draw (0,0) -- (10,0);
	
	\filldraw[white] (2,0) circle (9pt);
	\filldraw [white] (4,0) circle (9pt);
	\filldraw	[white]	(6,0) circle (9pt);
	\filldraw	[white]	(8,0) circle (9pt);
	\draw[red] (2,0) circle (9pt);
	\draw[red] (4,0) circle (9pt);
	\draw[red]	(6,0) circle (9pt);
	\draw[red]	(8,0) circle (9pt);
	\node at (2,0) {\scriptsize{ $\epsilon_c$}};
	\node at (4,0) {\scriptsize{ $\epsilon_c$}};
	\node at (6,0) {\scriptsize{ $\epsilon_c$}};
	\node at (8,0) {\scriptsize{ $\epsilon_c$}};
	
	\node at (1.6,-.4) {\tiny {$\omega c$}};
	\node at (3.6,-.4) {\tiny {$\omega c$}};
	\node at (5.6,-.4) {\tiny {$\omega c$}};
	\node at (7.6,-.4) {\tiny {$\omega c$}};
	\node at (9.6,-.4) {\tiny {$\omega c$}};
	\node at (2.4,-.4) {\tiny $c$};
	\node at (4.4,-.4) {\tiny $c$};
	\node at (6.4,-.4) {\tiny $c$};
	\node at (8.4,-.4) {\tiny $c$};
	\node at (.4,-.4) {\tiny $c$};
	\node at (.1,.4){\tiny{$(1_{\Lk[4]})$}};
	\node at (9.9,.4){\tiny{$(2_{\Lk[4]})$}};
	
	\end{tikzpicture}};

\node at (8,0){
	\ \begin{tikzpicture}[scale = 0.6]
	\draw (0,0) circle (2cm);
	\filldraw[white] (2,0) circle (9pt);
	\filldraw [white] (0,2) circle (9pt);
	\filldraw	[white]	(-2,0) circle (9pt);
	\filldraw	[white]	(0,-2) circle (9pt);
	\draw[red] (2,0) circle (9pt);
	\draw[red] (0,2) circle (9pt);
	\draw[red]	(-2,0) circle (9pt);
	\draw[red]		(0,-2) circle (9pt);
	
	\node at (2,-.1){\rotatebox{90}{\scriptsize{ $\epsilon_c$}}};
	\node at (-2,.1){\rotatebox{-90}{\scriptsize{ $\epsilon_c$}}};
	\node at (0,2){\rotatebox{180}{\scriptsize{ $\epsilon_c$}}};
	\node at (0,-2){\scriptsize{ $\epsilon_c$}};
	\node at (1.6,.4) {\tiny {$\omega c$}};
	\node at (1.6,-.4) {\tiny {$ c$}};
	\node at (-1.6,.4) {\tiny {$c$}};
	\node at (-1.6,-.4) {\tiny {$\omega c$}};
	\node at (-.6, 1.6) {\tiny {$\omega c$}};
	\node at (.6, 1.6) {\tiny {$c$}};
	\node at (-.6, -1.6) {\tiny {$ c$}};
	\node at (.6, -1.6) {\tiny {$\omega c$}};
	
	\end{tikzpicture}};
\end{tikzpicture}
\caption
{The $c$-coloured unit structures $\Lk[4](\epsilon_c)$ and $\Wl[4](\epsilon_c)$.} \label{Lk fig}
\label{fig: unit graphs} 
\end{figure}
 For any subset $W $ of bivalent vertices of a graph $\G$, and any 
$S_*$-structure $\alpha_{\setminus W} \in S(\Gnov)$, there is a unique $S_*$-structure $\alpha \in S (\G)$ such that $\delW \in \ovP{S_*}{\Grp}(\alpha, \alpha_{\setminus W})$: If $(\C_\two, b) \in \elG$ is a neighbourhood of $ v \in W$, then $\delW \circ b = ch_e \circ u$ for some $ e \in E(\Gnov)$. Hence $ \alpha \in S(\G)$ is determined by 
\[ S_*(b)(\alpha) = S_*(\delW \circ b)(\alpha_{\setminus W}) = S_* (u) S(ch_e)(\alpha_{\setminus W}) = \epsilon\left(S(ch_e)(\alpha_{\setminus W})\right).\] 

\begin{defn}\label{unit vertices}\label{iota vertex deletion}
Let $ (\G, \alpha)$ be an $S_*$-structured graph. Then
\[ W_\alpha \defeq \{ v \ | \text{ there is a neighbourhood } (\C,b) \text { of } v \text{ such that } S(b)(\alpha) \in im(\epsilon) \cup im(o)\} \subset V_0 \amalg V_2,\] 
is the set of \emph{vertices $\alpha$-decorated by (contracted) units}. 

An $S_*$-structure $(\G, \alpha)$ is called \emph{reduced} if $W_\alpha = \emptyset$. 

\end{defn}


Let $\G \not \cong \C_\nul$ and $W \subset V_2(\G)$. There exists an $S_*$-structure $\alpha_W \in S(\Gnov)$ such that $\delW \in \Grp(\G, \Gnov)$ describes a morphism in $ \ovP{S_*}{\Grp}((\G, \alpha), (\Gnov, \alpha_{\setminus W}))$ if and only if $W \subset W_\alpha$. By definition, $(\Gnov, \alpha_{\setminus W})$ is reduced if and only if $W = W_\alpha$.


%

\subsection{Similar structures} \label{subs. similar}
\label{rmk: good representatives}

The issues that can arise from trying to build degenerate substitution by the stick graph into the definition of the modular operad monad have been outlined in \cref{degenerate}. Degenerate substitutions, and therefore exceptional loops, can be avoided if there is a suitable notion of equivalence of $S_*$-structured graphs, for which all constructions may be obtained in terms of admissible representatives.

This principle informs the construction of the distributive law $\lambda\colon TD \Rightarrow DT$. 


By \cref{Gnov construction}, any similarity morphism that is not of the form $z \colon \C_\nul \to (\shortmid)$ or $\kappa^m \colon \Wm \to (\shortmid)$ preserves boundaries. Let $\G$ be a connected graph with non-empty vertex set. A boundary-preserving similarity morphism $f \in \Grsimp (\G, \G')$, together with an $X$-labelling $\rho \colon E_0 \to X$ of $\G$, induces an $X$-labelling on $\G'$. 

Recall that $X\Griso$ is the category of (admissible) $X$-graphs and boundary-preserving isomorphisms. The category $X\XGrsimp$ is obtained from $X \Griso$ by adjoining all similarity morphisms from objects of $X\Griso$ and, where necessary, their $X$-labelled codomains: 
\begin{itemize}
	\item If $X \not \cong \nul$, $X \not \cong \two$, then $ X \Griso$ is a bijective on objects subcategory of $X \XGrsimp$ whose morphisms are similarity morphisms 
that preserve the labelling of the ports. 

\item For $X = \two $, $\two\XGrsimp$ is obtained from $\two \Griso$ by adjoining the morphisms $ \delW[V] \colon \Lk \to (\shortmid)$, and hence also the labelled stick graphs $(\shortmid, id)$ and $(\shortmid, \tau)$. An admissible $\two$-graph $\X \in \two \Griso \subset \two \XGrsimp$ is in the same connected component as $(\shortmid, id)$ if and only if $\X = (\Lk, id_{\Lk})$ for some $k \geq 1$. In particular, $\tau \colon (\shortmid) \to (\shortmid)$ does not induce a morphism in $\two \XGrsimp$.

 \item When $X = \nul$, the morphisms $\delW[V] \colon \Wm \to (\shortmid)$, and $z \colon \C_\nul \to (\shortmid)$ are not boundary-preserving, and do not induce any labelling on the ports of $(\shortmid)$. So, the objects of $\nul\XGrsimp$ are the admissible $\nul$-graphs in $\nul \Griso$, together with $(\shortmid)$. In particular, $\Wm, \C_\nul$ and $(\shortmid)$ are in the same connected component of $\nul\XGrsimp$, and $(\shortmid)$ is terminal in this component. 

\end{itemize}
To simplify notation in what follows, let ${\C_{\nul}}_{\setminus V} \defeq (\shortmid)$ and $ \delW[V] = z \colon \C_\nul \to (\shortmid)$ in $\Grp$.
\begin{defn} \label{Xgrsim} For all $S_* \in \GSp$, the slice category $\ovP{S_*}{X\XGrsimp}$ induced by the canonical functor $X\XGrsimp \to \Grp \hookrightarrow \GSp$ is the category of \emph{similar $S_*$-structured $X$-graphs}. 
%
	
	Admissible $S_*$-structured $X$-graphs $(\X^1, \alpha^1) , (\X^2, \alpha^2)$ are called \emph{similar}, written $(\X^1, \alpha^1) \sim (\X^2, \alpha^2) $ (or $\alpha^1 \sim \alpha^2$), if they are in the same connected component of $ \ovP{S_*}{X\XGrsimp}$. 
\end{defn}

Let $(\X, \alpha)$ be an admissible $S_*$-structured $X$-graph such that $(\X, \alpha)\neq (\C_\nul, o_{\tilde c})$. Then 
\[ (\X^\bot_\alpha, \alpha^\bot) \defeq  (\X_{\setminus W_\alpha}, \alpha_{\setminus W_\alpha})\] is reduced and terminal in the connected component of $(\X, \alpha)$ in $\ovP{S_*}{X\XGrsimp}$. It is admissible unless $(\X, \alpha) =  \Lk(\epsilon_c)$, or $ (\X, \alpha )\Wm(\epsilon_c)$ for some $ c \in S_\S$ and $k, m \geq 1$. 

If $(\X^\bot_\alpha, \alpha^\bot)$ is not admissible, then it is represented by $(\shortmid, c)$ for some $c \in S_\S$. If $ c \neq \omega c \in S_\S$, then $(\shortmid, c) $ and $(\shortmid, \omega c)$  are in disjoint components of $\ovP{S_*}{ \two\XGrsimp}$. By contrast, 
 $z \colon \C_\nul \to (\shortmid)$ induces morphisms $ (\C_\nul, o_{\tilde c}) \to (\shortmid, c)$ {and} $(\C_\nul, o_{\tilde c}) \to (\shortmid, \omega c)$ in $\ovP{S_*}{\Grp}$. So, for all $c \in S_\S$, $\tilde c$ defines a unique element $[\shortmid, c] = [\shortmid, \omega c]$ in $\ovP{S_*}{\nul\XGrsimp}$. 

Moreover, the similarity maps $\kappa\colon \W \rightarrow (\shortmid) \leftarrow \C_\nul\colon z$ in $\Grp$ induce morphisms of $S_*$-structured graphs: 
\begin{equation}\label{eq. cone } \xymatrix{\W(\epsilon_c) \ar[rr] &&(\shortmid, c) && \ar[ll] (\C_\nul, o_{\tilde c}) \ar[rr]&& (\shortmid, \omega c)&& \ar[ll] \W(\epsilon_{\omega c})}. \end{equation}
So, $\W(\epsilon_c) \sim (\C_\nul, o_{\tilde c})$ and there is a double-cone shaped diagram in $\ovP{S_*}{\nul\XGrsimp}$ (\cref{fig: contracting units}).
			\begin{figure}
	[htb!]
	
	\begin{tikzpicture}[scale = .45]
	\draw[ thin, dashed , red] (-12,4)--(12,-4)
	(-12,-4)--(12,4);
	
	\draw [thick](3,1) arc [start angle=90, 
	end angle=450,
	y radius=1cm, 
	x radius=.5cm]
	node [pos=.25] {\begin{tikzpicture}\filldraw (0,0) circle (2pt);\end{tikzpicture}} ;;
	\draw [thin,gray, ->] (.2,3.1) arc [start angle=-70, 
	end angle= 250,
	y radius=.5cm, 
	x radius=.5cm];
	\draw[ thin,gray, ->]
	(-2.8,1.2) -- (-1, 2.4);
		\node at (-1.6,1.3){\tiny{$\kappa$}};
	\draw[thin,gray, ->]
	(-5.8,2.2) -- (-1,2.8);
	\draw[thin,gray, ->]
	(-8.8,3.2) -- (-1,3.2);

	\draw[thick, gray, ->]
	(0,.3) -- (0,2);
	\draw[thin,gray, ->]
	(2.8,1.2) -- (1, 2.4);
		\node at (1.8,1.3){\tiny{$\tau\kappa$}};
	\draw[thin,gray, ->]
	(5.8,2.2) -- (1,2.8);
	\draw[thin,gray, ->]
	(8.8,3.2) -- (1,3.2);
	
	\draw[ thin, gray, ->]
	(-11,0) -- (-9.4,0);
	\draw[ thin,gray, ->]
	(-8.7,0) -- (-6.3,0);
	\draw[ thin, gray, ->]
	(-5.5,0) -- (-3.1,0);
	\draw[thick, gray, dotted]
	(-12.5,0)--(-11,0);
	\draw[ thin, gray, ->]
	(11,0) -- (9.4,0);
	\draw[ thin,gray, ->]
	(8.7,0) -- (6.3,0);
	\draw[ thin, gray, ->]
	(5.5,0) -- (3.1,0);
	\draw[thick, gray, dotted]
	(12.5,0)--(11,0);
	
	\node at (0.2, 1){\tiny{$z$}};
	
	
	\draw[thin,gray, <->]
	(-2.6,-1.2) -- (2.6, -1.2);
	\node at (0,-.9){\tiny{$\tauG[\W]$}};
	\draw[thin, gray, <->]
	(-5.6,-2.2) -- (5.6, -2.2);
	\draw[thin,gray, <->]
	(-8.6,-3.2) -- (8.6, -3.2);
	\draw [ thick]
	(0,2.5)--(0,3.5);
	
	\filldraw (0,0) circle (4pt);
	\draw [thick](-3,1) arc [start angle=90, 
	end angle=450,
	y radius=1cm, 
	x radius=.5cm]
	node [pos=.75] {\begin{tikzpicture}\filldraw (0,0) circle (2pt);\end{tikzpicture}} ;;
	\draw [thick](6,2) arc [start angle=90, 
	end angle=450,
	y radius=2cm, 
	x radius=.75cm]
	node [pos=0] {\begin{tikzpicture}\filldraw (0,0) circle (2pt);\end{tikzpicture}}
	node [pos=.5] {\begin{tikzpicture}\filldraw (0,0) circle (2pt);\end{tikzpicture}} ;;
	\draw[thick] (-6,2) arc [start angle=90, 
	end angle=450,
	y radius=2cm, 
	x radius=.75cm]
	node [pos=0] {\begin{tikzpicture}\filldraw (0,0) circle (2pt);\end{tikzpicture}}
	node [pos=.5] {\begin{tikzpicture}\filldraw (0,0) circle (2pt);\end{tikzpicture}} ;;
	\draw [thick](9,3) arc [start angle=90, 
	end angle=450,
	y radius=3cm, 
	x radius=1.2cm]
	node [pos=.25] {\begin{tikzpicture}\filldraw (0,0) circle (2pt);\end{tikzpicture}}
	node [pos=.6] {\begin{tikzpicture}\filldraw (0,0) circle (2pt);\end{tikzpicture}} 
	node [pos=-.1] {\begin{tikzpicture}\filldraw (0,0) circle (2pt);\end{tikzpicture}} ;;
	\draw [thick] (-9,3) arc [start angle=90, 
	end angle=450,
	y radius=3cm, 
	x radius=1.2 cm]
	node [pos=-.25] {\begin{tikzpicture}\filldraw (0,0) circle (2pt);\end{tikzpicture}}
	node [pos=-.6] {\begin{tikzpicture}\filldraw (0,0) circle (2pt);\end{tikzpicture}} 
	node [pos=.1] {\begin{tikzpicture}\filldraw (0,0) circle (2pt);\end{tikzpicture}} ;;
	%
	
\end{tikzpicture}
%
\caption{
The contraction of units is described in terms of a commuting diagram in $\nul\XGrsimp$ that is strongly suggestive of a conical singularity. }
\label{fig: contracting units}
\end{figure}




\begin{lem}\label{lem: pi0}
For all pointed graphical species $S_*$ and all finite sets $X$, there is a canonical bijection 
\begin{equation}\label{eq: sim colim} \mathrm{colim}_{\X \in X\XGrsimp} S_*(\X) \cong \pi_0 (\ovP{S_*}{X\XGrsimp}).\end{equation}\end{lem} \begin{proof}
Since $\ovP{S}{X\Griso}\subset \ovP{S_*}{X\XGrsimp}$, there is a surjection of connected components $\pi_0(\ovP{S}{X\Griso})\to \pi_0(\ovP{S_*}{X\XGrsimp})$. Every component of $\ovP{S_*}{X\XGrsimp}$ has a representative in $\ovP{S}{X\Griso}$, and the result follows from  (\ref{eq: component identity}) and \cref{Xgrsim}.
\end{proof}

\subsection{A distributive law for modular operads}\label{dist for CSM}

Let $S$ be a graphical species and $X$ a finite set. An element of $TDS_X$ is represented by an $ X$-graph $\X$, with a $DS$-structure $\alpha \in DS(\X) = S^+(\X)$. The idea is to construct a distributive law $\lambda\colon TD \Rightarrow DT$ such that $\lambda S$ is invariant under similarity morphisms. 

%
%

\begin{prop}\label{monads distribute} There is a distributive law $\lambda\colon TD \Rightarrow DT$ such that for all graphical species $S$ and finite sets $X$, and all elements $[\X, \alpha]$, $[\X', \alpha'] $ of $TDS_X$, 
\[ \lambda S [\X, \alpha] = \lambda S [\X', \alpha'] \text{ in } DTS_X \text{ if and only if } [\X, \alpha]\sim [\X', \alpha'] \in \ovP{S^+}{X\XGrsimp}.\]
\end{prop}

\begin{proof}
	

Since the endofunctor $D$ just adjoins elements, there are canonical inclusions $ TS \hookrightarrow DTS$ and $TS\hookrightarrow TDS$. The natural transformation $\lambda\colon TD \Rightarrow DT$ will restrict to the identity on $TS$. 

For a finite set $X$, elements of $ TDS_X$ are represented by $DS$-structured $X$-graphs $(\X, \alpha) $, whereas elements of $DTS_X$ are either of the form $\epsilon^{DTS}_c, o^{DTS}_{\tilde c}$ for $c \in S_\S$, or are represented by $S$-structured $X$-graphs $(\X', \alpha')$. Observe also that an object $(\X, \alpha) \in \ovP{S^+}{X\XGrsimp}$ is reduced and admissible 
if and only if $(\X, \alpha) \in \ovP{S}{X\Griso}$, and hence $[\X, \alpha] \in TS_X$. 

Let $(\X, \alpha)\in \ovP{DS}{X\Griso}$. If $X = \two$, and $ (\X, \alpha)$ has the form $\Lk(\epsilon_c)$, set 
\[\lambda S [\X, \alpha] = \epsilon^{DS}_c \in DTS_\two. \] 
And, if $X = \nul$, and $(\X, \alpha) = \Wl(\epsilon_c)$ or $(\X, \alpha) = (\C_\nul, o_{\tilde c})$, set 
\[ \lambda  S [\X, \alpha] =o^{DTS}_{\tilde c} \in DTS_{\nul}. \]

Otherwise, the component of $(\X, \alpha)$ in $ \ovP{S^+}{X\XGrsimp}$ has an admissible and reduced (hence terminal) object $(\X^\bot_\alpha, \alpha^\bot)$, so we can set
\[ \lambda S [\X, \alpha] = [\X^\bot_\alpha, \alpha^\bot] \in TS_X \subset DTS_X.\] 

The assignment so defined clearly extends to a natural transformation $\lambda \colon TD \Rightarrow DT$.

The verification that $\lambda$ satisfies the four axioms \cite[Section~1]{Bec69} for a distributive law is straightforward but unenlightening, so I prove just one here, that the following diagram of natural transformations commutes: 

\begin{equation}\label{eq: dist axiom}
\xymatrix{
TD^2 \ar@{=>}[rr]^-{\lambda D} \ar@{=>}[d]_{T (\mu^\DD)} && DTD \ar@{=>}[rr]^-{D \lambda } && D^2 T \ar@{=>}[d]^{(\mu^\DD) T}\\
TD \ar@{=>}[rrrr]_-\lambda &&&& DT } \end{equation} 

Let $[\X, \alpha] \in TD^2 S_X$. The result is immediate when $\X = \C_\nul$. Moreover, all the maps in (\ref{eq: dist axiom}) restrict to the identity on $TS\subset TD^2S$.

Therefore, we may assume that $[\X, \alpha] \not\in TS_X$ and that $\X \not \cong \C_\nul$. For $j = 1, 2$, define the sets $W^j$, 
of {vertices decorated by distinguished elements} adjoined in the $j^{th}$ application of $D$:
\[ W^j \defeq \{ v|v \text{ has a minimal neighbourhood } (\CX, b) \text{ with } \ D^2S(b)(\alpha) \in im(\epsilon^{D^jS}) \} \subset V_2.\]
Then $T\mu^\DD S [ \X, \alpha] = [\X,  \alpha^{\mu^\DD}]  \in TDS_X$, where $\alpha^{\mu^\DD} \in DS(\X)$ is given by
\[ {DS(b)(\alpha^{\mu^\DD})} =\left \{ \begin{array}{ll} \epsilon^{DS}_c &\text{ if } (\C , b) \text{ is a minimal neighbourhood of } v \in W^1 \amalg W^2,\\
 {D^2S(b)(\alpha) \in S(\C)}& \text{ otherwise.}
\end{array}\right .\]

If $W^1 \cup W^2 \neq V$, then(\ref{eq: dist axiom}) gives 
\[ \xymatrix{
[\X, \alpha] \ar@{|->}[rr]^-{\lambda DS} &&[\X_{\setminus W^2}, \alpha_{\setminus W^2}] \ar@{|->}[rr]^-{D \lambda S } &&[(\X_{\setminus W^2})_{\setminus W^1},(\alpha_{\setminus W^2})_{\setminus W^1}] \ar@{=}[d]\\
[\X, \alpha] \ar@{|->}[rr]_-{T (\mu^\DD)} &&[\X, \mu^\DD \alpha] 
\ar@{|->}[rr]_-{\lambda S} &&[\X_{\setminus (W^1 \amalg W^2)}, \alpha_{\setminus (W^1 \amalg W^2)}] \in TS_X}. \]

If $W^1 \amalg W^2 = V$, then $T(\mu^\DD)[\X, \alpha] $ has the form $ \Lk(\epsilon^{DS}_c)$ or $ \Wm(\epsilon^{DS}_c)$ and both paths map to the corresponding (contracted) unit in $DTS$. \end{proof}

It follows that there is a composite monad $\DD\TT$ on $\GS$, induced by $\lambda$. 
Moreover, by \cite[Section~3]{Bec69}, 
$\lambda\colon TD\Rightarrow DT$ induces a lift $\TTp$ of $\TT$ to $\GSp$, such that the EM categories $\GS^{\DD\TT}$ and $\GSp^{\TTp}$ are canonically isomorphic. (See also \cref{subs: dist}.)


\begin{cor}
\label{prop: TpX} The monad $\TTp = (\Tp, \mup, \etap)$ on $\GSp$ is given by 
\[ \Tp S_\S = S_\S, \text { and } \Tp S_X = \mathrm{colim}_{\X\in X\XGrsimp} S(\X). \]
The unit $\etap\colon 1_{\GSp} \Rightarrow \Tp$ and multiplication $\mup\colon \Tp^2 \Rightarrow \Tp$ are induced by the unit $\eta^\TT \colon 1_{\GS} \Rightarrow T$ and multiplication $\mu^\TT \colon T^2 \Rightarrow T$ for $\TT$. In other words, if $[\X, \alpha]_*$ denotes the class of $[\X, \alpha] \in TS_X$ in $\Tp S_X$, then
\[ \etap (\phi) = [\eta^\TT \phi]_* \text { and } \mup [\X, \beta] = [\mu^\TT[\X, \beta]]_*. \]

\end{cor}
\begin{proof}
Let $S_* = (S, \epsilon, o)$ be a pointed graphical species, and let $h_\DD\colon DS \to S$ denote the $\DD$-algebra structure on $S$, so $ \epsilon =h_{\DD}(\epsilon^{+})$ and $o = h_\DD(o^+)$.

Observe first that \[(\mu^\DD \mu^\TT)\circ (D \lambda T) \circ (DTD \eta^\TT)= (\mu^\DD T)\circ (D \lambda) \colon DTD\Rightarrow DT. \] So, by \cite[Section~3]{Bec69}, $\Tp(S_*)$ is described by the coequaliser 
\begin{equation}\label{T* equiv}\xymatrix{ DTD S \ar[rrrr]^-{DT h_\DD } \ar[drr]_-{ D \lambda S} &&&& DTS \ar [rr]^-{\pi} && \Tp S_*.\\
&& D^2 TS \ar[urr]_{ \mu^\DD TS}. &&&&} \end{equation}

The occurrence of $\lambda $ means that the lower path of (\ref{T* equiv}) identifies similar elements of $TDS \hookrightarrow DTDS$ and, since $ \epsilon =h_{\DD}(\epsilon^{+})$, it follows that similar elements of $TS \subset DTS$ are identified by the quotient $\pi$. Reduced elements of $TS_X \subset DTDS_X \to DTS_X \supset TS_X$ are unchanged under both paths in (\ref{T* equiv}). Therefore $\Tp S_X = \mathrm{colim}_{\X \in X\XGrsimp} S_*(\X)$ as required. Clearly $[\Lk(\epsilon)]_*$ and $[\Wm(\epsilon)]_* = [\C_\nul, o]_*$ provide (contracted) units for $\Tp S_*$. 

%
%

Let $(\X, \beta)$ represent an element of $ T^2S_X$ such that $(\X, \beta) \not \cong (\C_\nul, (\C_\nul, o_{\tilde c}))$, $(\X, \beta) \not \cong (\C_\nul,[ \Wm(\epsilon_c)])$. And let $\delW \in X\XGrsimp((\X, \beta), (\X_{\setminus W}, \beta_{\setminus W}))$, so $W \subset W_{\setminus \alpha}$. Then $(\C_\two,b) \in \elG[\X]$ is a neighbourhood of $v \in W$ if and only if there is a $k \geq 1$, and a $c \in S_\S$, such that $S(b)(\beta) =[\Lk(\epsilon_c)]$ in $ TS_\two.$ In particular, $\mu^\TT [\X, \beta]$ and $\mu^\TT[\X_{\setminus W}, \beta_{\setminus W}]$ are represented by similar objects of $\ovP{S_*}{X\XGrsimp}$. 

Finally, let $(\X, \beta) \in TS (\X)$ be similar to $(\C_\nul, (\C_\nul, o_{\tilde c}))$, and hence also $(\X, \beta) \sim  (\C_\nul,[ \Wm(\epsilon_c)])$ in $\ovP{\Tp S_*}{\nul\XGrsimp}$. Then, either $\mu^\TT[\X, \beta] = (\C_\nul, o_{\tilde c})$ or $\mu^\TT[\X, \beta] = [\Wm[M](\epsilon_c)]$ for some $M \geq 1$. It follows that $\mu^\TT$ preserves similarity classes of $(\C_\nul,[ \Wm(\epsilon_c)])$ in $TS (\C_\nul) \to T^2 S_\nul$. 

Hence,  $\mup = [\mu^\TT(-)]_*$ is well-defined and provides the multiplication for $\TTp$. Clearly $\etap = [\eta^{\TT}(-)]_*$, and the corollary follows.
\end{proof}


 In particular, to compute the image of $[\X, \beta] \in \Tp^2 S_*$ under the multiplication $\mup (S_*) \colon \Tp^2 S_* \to \Tp S_*$, it is sufficient to chose a non-degenerate representative of $\beta$ (i.e.~a non-degenerate $\X$-shaped graph of $S$-structured graphs) and quotient by similarity at the end. 
 
 \begin{rmk}\label{rmk: TD} There is also a distributive law $DT \Rightarrow TD$. Algebras for the composite monad $\TT\DD$ are just the cofibred coproducts of algebras for $\DD$ and $\TT$. There is no further relationship between the two structures. (See also \cref{ex: category composite}.)
 	
 \end{rmk}
 \subsection{$\DD\TT$-algebras are modular operads.}

At last we are ready to prove the first main theorem -- that modular operads are $\DD\TT$-algebras in $\GS$. 

Let $(S, \diamond, \zeta, \epsilon)$ be a modular operad. Since $(S, \diamond, \zeta)$ is a non-unital modular operad, it is equipped with a $\TT$-algebra structure $p_\TT = p^{\diamond, \epsilon}_\TT\colon TS \to S$, by \cref{unpointed modular operad}. Moreover $S_* = (S, \epsilon, \zeta \epsilon)$ is a pointed graphical species, and hence also a $\DD$-algebra. 

\begin{lem}\label{lem: deletion lemma}
	The defining functor $\ovP{S}{X\Griso} \twoheadrightarrow TS_X \xrightarrow{p_\TT} S_X$ factors through $\pi_0(\ovP{S_*}{X\XGrsimp})$.

\end{lem}
\begin{proof}
Since $(S, \diamond, \zeta, \epsilon)$ is a modular operad, $p_\TT$ satisfies
\begin{equation}\label{eq. mult invariant} p_\TT[ \mathcal M_c(\phi, \epsilon_c)] = \phi \diamond_c \epsilon_c = \phi = p_\TT(\eta^\TT \phi) \text{ wherever defined.}\end{equation} 
In particular, let $\alpha \in S(\X)$ for $\X \not \cong \C_\nul$ and let $W \subset W_\alpha$ be such that $W \neq V(\X)$. Then 
\[ p_\TT[\X, \alpha] = p_\TT[\X_{\setminus W}, \alpha_{\setminus W}]\] by the proof of \cref{unpointed modular operad}.

Therefore, it remains to check that $p_\TT[\C_\nul, o_{\tilde c}] = p_\TT[\W(\epsilon_c)]$ for all $c \in S_\S$. This is immediate since $W \cong \mathcal N^\nul_{1,2}$ (\cref{ex: N graph}), and hence 
 \[ p_\TT[\C_\nul, o_{\tilde c}] = o_{\tilde c}  \defeq \zeta (\epsilon_c) = p_\TT[\mathcal N(\epsilon_c)] =  p_\TT[\Wm(\epsilon_c)]\]
 by construction, since $(S, p_\TT)$ is a $\TT$-algebra.
%
%
%
%
\end{proof}

It is now straightforward to prove that $\DD\TT$ is the desired modular operad monad on $\GS$.
\begin{thm}\label{CSM monad DT}
The EM category $\GS^{\DD\TT}$ of algebras for $\DD\TT$ is canonically isomorphic to $\CSM$.
\end{thm}

\begin{proof} 

Let $(A, h)$ be a $\DD\TT$-algebra with corresponding $\DD$ and $\TT$ structure morphisms
\[h_\DD \defeq h \circ (D\eta^\TT A)\colon DA \to A, \qquad \text{ and } \qquad h_\TT \defeq h \circ (D \eta^\DD A )\colon TA \to A.\] 
Since $\eta^\DD$ is just an inclusion, 
$ h_\TT = h|_{TA}\colon TA \to A$ is the restriction of $h$ to $TA \subset DTA$. By \cref{unpointed modular operad}, $A$ is equipped with a multiplication $\diamond = h \circ [\mathcal M( - , -)] $ and contraction $\zeta = h \circ [\mathcal N( -)] $, as described in the proof of \cref{lem: algebra operations}, so that $(A, \diamond, \zeta)$ is a non-unital modular operad.


It remains to show that $\epsilon$ provides a unit for the multiplication $\diamond$. By the monad algebra axioms, 
there are commuting diagrams in $\GS$:

\begin{minipage}[t]{0.5\textwidth}
\begin{equation} \label{etaDT}
\xymatrix{ A \ar[rr]^{ \eta^\DD \eta^\TT A}\ar@{=}[drr] && DTA \ar[d]^h \\	
	&& A,													}\end{equation}

\end{minipage}
\begin{minipage}[t]{0.5\textwidth}
\begin{equation} \label{muDT}
\xymatrix{ (DT)^2 A \ar[r]^-{D \lambda T A} \ar[d] _{DT h} & D^2 T^2 A \ar[r]^-{\mu ^\DD \mu^\TT A} & DTA \ar[d]^h \\
	DTA \ar[rr]_-h && A. }\end{equation}
\end{minipage} 

Let $X$ be a finite set, $\underline c \in (A_\S)^X$, and let $\phi \in A_{\underline c}$. By definition of $\lambda$, 
\[D \lambda TA [\mathcal M_{c_x}(\eta^{\TT}A (\phi) , \epsilon^{DTA}_{c_x})] = [\CX, \eta^\TT A(\phi)] = (\eta^\TT)^2 A(\phi)\] for all $x \in X$. So $ [\mathcal M_{c_x}(\eta^{\TT}A (\phi) , \epsilon^{DTA}_{c_x})] \mapsto \phi$ under the top-right path in (\ref{muDT}).


Now, $\phi\ \diamond_{c_x} \ \epsilon_{c_x}= h[\mathcal M_{c_x}(\phi, \epsilon_{c_x})] $ by definition, and $[\mathcal M_{c_x}(\phi, \epsilon_{c_x})]= DTh[\mathcal M_{c_x}(\eta^{\TT} (\phi) , \epsilon^{DTA}_{c_x})]$ by the monad algebra axioms. Then, since (\ref{muDT}) commutes,
\[ \phi\ \diamond_{c_x} \ \epsilon_{c_x}= h[\mathcal M_{c_x}(\phi, \epsilon_{c_x})] = hDTh[\mathcal M_{c_x}(\eta^{\TT} (\phi) , \epsilon^{DTA}_{c_x})] = \phi, \]  
and $\epsilon$ is a unit for $\diamond $. 

Conversely, a modular operad $ (S, \diamond, \zeta, \epsilon)$ induces a pointed graphical species $S_* = (S, \epsilon, o =\zeta \epsilon)$. By \cref{unpointed modular operad}, $(S, \diamond, \zeta)$ has a $\TT$-algebra structure 
$ p_\TT\colon TS \to S$ such that 
\begin{equation}\label{eq. T alg}
	 \diamond = p_\TT \ \circ [\mathcal M (\cdot, \cdot)] \ \text { and }\ \zeta = p _\TT \ \circ [\mathcal N(\cdot)].\,
\end{equation}
and, for all $c \in S_\S$ and all $m \geq 1$, since $\epsilon $ is a unit for $\diamond$,
\begin{equation}
	\label{eq cont unit}
	o_{\tilde c} \defeq \zeta \epsilon_c =p_\TT [\mathcal N(\epsilon_c)] = p_\TT[ \W(\epsilon_c)]  =  p_\TT[\Wl (\epsilon_c)].
\end{equation}

Let $p \colon DTS \to S$ be defined by 
\begin{equation}\label{eq structure map}
	p (\epsilon^{DTS}) = \epsilon\colon S_\S \to S_\two, \ p (o^{DTS}) = o = \zeta \epsilon\colon S_\S\to S_\nul, \text{ and }  p = p_\TT\colon TS \to S \text { on } TS\subset DTS.
\end{equation}
Then (\ref{etaDT}) commutes for $(A,h) = (S, p)$ since $ (S, p_\TT)$ is a $\TT$-algebra.

%
%

It remains to check that (\ref{muDT}) commutes for $(A,h) = (S, p)$. This is clear for the adjoined (contracted) units $\epsilon ^{(DT)^2S}$, and $ o^{(DT)^2S}$ in $(DT)^2S$. Otherwise, if $[\X, \beta] \in TDTS_X$, then 
%
exactly one of the following four conditions holds:
\begin{enumerate}[(i)]
\item $X = \nul$ and $[\X, \beta] = [\C_\nul, o^{DTS}_{\tilde c}]$. In this case, it is immediate from the definitions of $p$ and $\lambda$ that the image of $[\X, \beta]$ under both paths in (\ref{muDT}) is $o_{\tilde c}$;
\item $X = \nul$ and $[\X, \beta] = [\Wl(\epsilon^{DTS}_c)]$ for some $m \geq 1$, and $c \in S_\S$. The application of $\lambda TS$ means that the top-right path takes $[\Wl(\epsilon^{DTS}_c)] $ to $o_{\tilde c} \in S_\nul$.

Since $ p (\epsilon^{DTS}) = \epsilon$, the bottom-left path takes $[\Wl (\epsilon^{DTS}_c)]$ first to $[\Wl(\epsilon_c)] \in TS_\nul$ and then, by (\ref{eq cont unit}), to $o_{\tilde c}$. 
\item $X = \two$ and $[\X, \beta] = [\Lk(\epsilon^{DTS}_c)]$ for some $k\geq 1$, and $c \in S_\S$. Once again, the application of $\lambda TS$ means that the top-right path takes $[\Lk(\epsilon^{DTS}_c)]$ to $\epsilon_{ c} \in S_\two$.

The bottom-left path takes $[\Lk(\epsilon^{DTS}_c)]$ first to $[\Lk (\epsilon_c)] \in TS_\two$ by applying $p$ inside and then, by \cref{lem: deletion lemma}, 
\[ p[\Lk(\epsilon_c)] = p_\TT[\Lk(\epsilon_c)] = \epsilon_c \in S_\two ; \]	
\item Otherwise, $\lambda TS[\X, \beta] = [\X^\bot_\beta, \beta^\bot] $ where $(\X^\bot_\beta, \beta^\bot) $ is the unique reduced and admissible $S$-structured graph that is similar to $(\X, \beta)$ in $\ovP{S_*}{X\XGrsimp}$, and hence $[\X^\bot_\beta, \beta^\bot] \in T^2S_X$. 

So, for the top-right path we have 
\[ [\X, \beta] \mapsto p \circ \mu^\TT S [\X^\bot_\beta, \beta^\bot] = p_\TT \circ \mu^\TT S [\X^\bot_\beta, \beta^\bot] \in S_X.\]
 Since $(S, p_{\TT})$ is a $\TT$-algebra $p_\TT \circ \mu^\TT S [\X^\bot_\beta, \beta^\bot] = p_\TT \circ T p_\TT S [\X^\bot_\beta, \beta^\bot]$ and, by \cref{lem: deletion lemma}, this is precisely $ p \circ Tp[\X, \beta]$.  

\end{enumerate}
Therefore (\ref{muDT}) commutes, and $(S, p)$ has the structure of a $\DD \TT$-algebra.

It is straightforward to verify that the assignment $(S, \diamond, \zeta, \epsilon) \mapsto ( S, p)$ is natural and that the functors $ \CSM \leftrightarrows \GS^{\DD\TT}$ so defined are each others' inverses.
\end{proof}
\begin{rmk}
	\label{rmk. alg free} In particular, $\DD\TT$ is the algebraically free monad \cite{Kel80} on the endofunctor $\TJK$ from \cref{degenerate}.
\end{rmk}

%

\section{A nerve theorem for modular operads}\label{s. nerve}

By \cref{CSM monad DT} and \cite{Bec69}, 
 there is a diagram of functors
\begin{equation} \label{eq: nerve big picture}
\xymatrix{ 
	&&\Klgr\ar@{^{(}->} [rr]_-{\text{f.f.}}						&& \CSM\ar@<2pt>[d]^-{\text{forget}^{\TTp}} \ar [rr]^-N				&&\pr{\Klgr}\ar[d]^{j^*}	\\
	\fisinvp \ar@{^{(}->} [rr]	^-{\text{dense}}_-{\text{ f.f.}}		&& \Grp \ar@{^{(}->} [rr] ^-{\text{dense}}_-{\text{ f.f.}} \ar[u]^{j}_{\text{b.o.}}			&& \GSp \ar@<2pt>[d]^-{\text{forget}^\DD} \ar@{^{(}->} [rr]^-{\text{ }}_-{\text{f.f.}}	 \ar@<2pt>[u]^-{\text{free}^{\TTp}}	&& \pr{\Grp}\ar[d] \\
	\fisinv\ar@{^{(}->} [rr]^-{\text{dense}}_-{\text{ f.f.}}\ar[urr]^-{\text{dense}}\ar[u]^{\text{b.o.}}	&& \Gr \ar@{^{(}->} [rr]^-{\text{dense}}_-{\text{ f.f.}}	 \ar[u]_{\text{b.o.}}	&& \GS \ar@{^{(}->} [rr]_-{\text{ }}_-{\text{f.f.}}	 \ar@<2pt>[u]^-{\text{free}^\DD}		&& \pr{\Gr}.}
\end{equation}
 where $\Klgr$ is the category obtained in the bo-ff factorisation of $ \Gr \to \GS \to \CSM$, \textit{and} also in the bo-ff factorisation of $\Grp \to \GSp \to \CSM$. 
 
 The goal of this section is to prove the following nerve theorem for modular operads using the abstract machinery described in \cref{sec: Weber}.

\begin{thm}\label{nerve theorem}The functor $N\colon\CSM \to\pr{\Klgr}$ is full and faithful. Its essential image consists of precisely those presheaves $P$ on $\Klgr$ whose restriction to $\pr{\Klgr}$ are graphical species. In other words, 	
	\begin{equation}
	\label{eq. Segal nerve sec}P(\G) = \mathrm{lim}_{(\C,b) \in \elG} P(\C) \text{ for all graphs } \G.
	\end{equation}
\end{thm}

\begin{rmk}
	{A version of this theorem was stated in \cite{JK11}, and another version was proved, by different methods from those presented here, in \cite[Theorem~3.8]{HRY19a}. In \cite{Ray18}, I proved a version of \cref{nerve theorem} by almost the same methods as presented here, but without the use of the distributive law. In all these versions, the statement of the Segal condition (\ref{eq. Segal nerve sec}) is the same.} 
\end{rmk}

An overview (following \cite[Sections~1-3]{BMW12}) of monads with arities was given in \cref{sec: Weber}. If the monad $\DD\TT$ on $\GS$ had arities $\Gr$ \cref{nerve theorem}, then \cref{nerve theorem} would follow immediately from \cite[Section~1]{BMW12}. Unfortunately, this is not the case. The obstruction, unsurprisingly, relates to the contracted units (see \cref{rmk: Gr no arities}).
%

The remainder of this work is devoted to showing instead that $\TTp$ {has arities $\Grp \subset \GSp$}. In this case, the nerve $N\colon\CSM \to \pr{\Klgr}$ is fully faithful. Moreover, because $\fisinv$ is dense in $\Grp$, the essential image of $N$ is characterised by the $\Klgr$-presheaves $P$ that satisfy the Segal condition (\ref{eq. Segal nerve sec}). 

By construction, $\Klgr \subset \CSM$ is the full subcategory on the modular operads $\Klgr(\G)$, free on connected graphs $\G \in \Gr$. So, the first step is to study these in more detail.


\subsection{The free modular operad on a graph}\label{T*H}

Fix a connected graph $\H = (E, H,V,s,t, \tau) $. To streamline the notation, let $T\H \defeq T\yet\H$ denote the free non-unital modular operad on $\H$, and $ \Tp \H \defeq \Tp \yetp \H$ the corresponding free unital modular operad on $\H$. 

	Of course, $\Tp \H (\shortmid) = \{ ch_e\}_{e \in E} = \Grp(\shortmid, \H)$. 
Recall that the unit for $\yetp \H$ is given by $ch_e \mapsto \epsilon^{\H}_e \defeq ch_e \circ u \in \Grp(\C_\two, \H) $, and the contracted unit for $\yetp \H$ is given by $ch_e \mapsto o^{\H}_{\tilde e} \defeq ch_e \circ z \in \Grp(\C_\nul, \H)$.
	
	So, by \cref{monads distribute}, $\Tp \H$ has units
	\[ ch_e \mapsto \epsilon^{\Tp \H}_e \defeq [\eta^\TT\epsilon_e^{\H}]_* = [\Lk, ch_e \circ u^k]_*, \] and contracted units \[ ch_e \mapsto o^{\Tp \H}_{\tilde e} \defeq [\eta^\TT o^\H_e]_* = [\C_\nul,ch_e \circ z]_* = [\Wm,ch_e \circ \kappa ^m]_*. \]
	
	
Let $X$ be a finite set. By \cref{prop: TpX}, elements of $ \Tp \H_X$ are represented by pairs $(\X, f)$ where $\X $ is an admissible $X$-graph and $f \in \Grp(\X, \H)$. Pairs $(\X^1, f^1)$ and $(\X^2, f^2)$ represent the same element $[\X,f]_* \in \Tp \H_X$ if and only if there is a commuting diagram 
	\begin{equation}\label{connected graph factors}\xymatrix{ 
		\X^1 \ar[rr]^-{g^1} \ar[drr]_{f^1}& & \X^\bot \ar[d]^{f^\bot }&&\ar[ll]_-{g^2} \X^2 \ar[dll]^{f^2} \\&&\H&&}\end{equation} in $\Grp$ such that, for $j = 1,2$, $g^j$ is a morphism in $X\XGrsimp$, and $f^\bot\colon \X^\bot \to \H$ is an (unpointed) \'etale morphism in $\Gr$. 
	
Outside the (contracted) units, $\X^\bot$ is admissible. Otherwise $f^\bot = ch_e \in \Gr(\shortmid, \H)$ for some $e \in E$. 
In particular, the following special case of (\ref{connected graph factors}) commutes in $\Grp$ for all $ e \in E$ and all $m \geq 1$:

	\begin{equation}\label{0 connected} \xymatrix{ \C_\nul \ar[rr]^-{ z} \ar[drr]_{ch_e \circ z } & &(\shortmid)\ar[d]^{ch_e}&& \Wl \ar[ll]_{\kappa^m } \ar[dll]^-{ch_e \circ \kappa^m}\\
		&&\H.&&}\end{equation} 
	This will be essential in the proof of \cref{nerve theorem}. 
	
\subsection{ The category $\Klgr$}\label{Klgr}
By (\ref{eq: nerve big picture}), $\Klgr$ is the restriction to $\Grp$ of the Kleisli category of $\TTp$. So, for all pairs $(\G, \H)$ of graphs \[\Klgr(\G, \H) = \GSp (\G, \Tp \H) \cong \Tp \H (\G).\] 
In particular, for $\G \cong \CX$ or $\G \cong (\shortmid)$, $ \Klgr(\G, \H) \cong \Tp \H (\G)$ has been described in \cref{T*H}. 

For the general case, it follows from \cref{T*H} that a morphism $\alpha \in \Klgr(\G, \H)$ is represented by a non-degenerate $ \G$-shaped graph of graphs $\Gg$ with colimit $\Gg(\G)$, together with a morphism $f\in \Grp(\Gg(\G),\H)$. 

%
Let $\G\not \cong \C_\nul$ and $\H$ be graphs. For $i = 1,2$, let $\Gg^i$ be a non-degenerate $\G$-shaped graph of graphs with colimit $\Gg^i(\G)$, and let $f^i \in \Grp(\Gg^i(\G), \H)$. For each $(\C, b) \in \elG$, let $\iota^i_b \colon \Gg^i(b) \to \Gg^i(\G)$ denote the defining embedding.
 \begin{lem}\label{lem: Klgr representatives} 
The pairs $(\Gg^1, f^1), (\Gg^2, f^2)$ represent the same morphism $\alpha \in \Klgr(\G, \H)$ if and only if there is a non-degenerate $\G$-shaped graph of graphs $\Gg$ with colimit $\Gg(\G)$, and a morphism $f \in \Grp(\Gg(\G), \H)$ 
such that there is a commuting diagram 
\begin{equation}\label{eq: well-defined kleisli}\xymatrix{\Gg^1(\G) \ar[rr] \ar[drr]_{f^1} && \Gg(\G)\ar[d]^{f} &&\ar[ll] \Gg^2(\G) \ar[dll]^{f^2}\\&&\H.&&}\end{equation}
 \end{lem}
 in $\Grp$ where the morphisms in the top row are vertex deletion morphisms.
 \begin{proof}
 If $(\Gg^1, f^1)$ and $ (\Gg^2,f^2)$ represent the same morphisms $\alpha \in \Klgr(\G, \H)$, then, for all $(\CX[X_b],b) \in \elG$, $(\Gg^1(b), f^1 \circ \iota^1_b)$ and $(\Gg^2(b), f^2 \circ \iota^2_b)$ are similar in $\ovP{\yetp \H}{X_b \XGrsimp}$ by definition. 
Therefore, by \cref{T*H}, since 
$\G \not \cong \C_\nul$, there is an admissible graph $\Gg(b)$ and a morphism $f_b \in \Grp(\Gg(b), \H)$ such that the following diagram -- in which the horizontal morphisms are vertex deletion morphisms between graphs with non-empty boundaries -- commutes in $\Grp$: 
\[ \xymatrix{\Gg^1(b) \ar[rr] \ar[drr]_{f^1 \circ \iota^1_b} && \Gg(b) \ar[d]^{f_b} && \ar[ll] \Gg^2(b)\ar[dll]^{f^2 \circ \iota^2_b}\\
&& \H.&&} \]

If $(\Gg(\G), f)$ is the colimit of the non-degenerate $\G$-shaped graph of $\yetp \H$-structured graphs defined by $(\CX[X_b], b) \mapsto (\Gg(b), f_b)$ for all $(\CX[X_b], b) \in \elG$, then 
(\ref{eq: well-defined kleisli}) commutes by construction. The converse follows immediately from the definitions.	
 \end{proof}

Since every graph $\G$ is trivially the colimit of the identity $\G$-shaped graph of graphs $\Gid \colon (\C, b) \mapsto \C$ (\cref{subs: gluing}), the assignment $f \mapsto [\Gid, f] \in \Klgr(\G, \H)$ induces an inclusion of categories $\Grp \hookrightarrow \Klgr$. 

It follows that there is weak ternary factorisation system on $\Klgr$: Morphisms in $\Klgr$ factor as boundary-preserving morphisms $[\Gg] \colon \G \to \Gg(\G)$ represented by non-degenerate graphs of graphs $\Gg$, followed by morphisms in $\Grp$, which themselves factor as $(\Grsimp, \Gr)$ by \cref{pp* sim}. 

\subsection{Factorisation categories}\label{subs: factorisation categories}

More generally, let ${\GSp}_{\TTp}$ be the Kleisli category of $\Tp$ given by ${\GSp}_{\TTp}(S_*, S_*') = \GSp(S_*, \Tp S_*')$ for all $S_*, S_*' \in \GSp$. In particular, the graphical category $\Klgr \subset {\GSp}_{\TTp}$ is the full subcategory whose objects are graphs $\G \in \Gr$. 

Let $S_*$ be a pointed graphical species. Elements of $\Tp S_X$ correspond to similarity classes of $ S_* $-structured $X$-graphs $(\X, \gamma)$. So, for any graph $ \G$, a morphism $\beta \in \GSp(\G, \Tp S_*)$ 
is represented by a non-degenerate $\G$-shaped graph of $S$-structured graphs $\Gg_S $. The colimit of $\Gg_S$ describes an $S_*$ structured graph $(\Gg(\G), \alpha)$, where $\Gg(\G)$ is the colimit of the underlying $\G$-shaped graph of graphs $\Gg\colon \elG[\G] \to \Grs \to \Gr$, which represents a morphism $[\Gg]\in {\GSp}_{\TTp}(\yetp \G, \yetp \Gg(\G))$, as in \cref{Klgr}.


So, let $S_*$ be a pointed graphical species, $\G$ a graph, and let $ \beta \in \GSp(\G, \Tp S)$. The following definition is from \cite[Section~2.4]{BMW12}:
\begin{defn}\label{factorisation category}

The \emph{factorisation category $\factcat$ of $\beta$} is the category whose objects are pairs $(\Gg, \alpha)$, where $ \Gg $ is a non-degenerate $\G$-shaped graph of graphs with colimit $\Gg(\G)$ and $\alpha \in \GSp (\Gg(\G), S)\cong S(\Gg(\G))$ is such that $\beta $ is given by the composition of morphisms in ${\GSp}_{\TTp}$:
\[ \xymatrix{ \G \ar[rr]^-{[\Gg]} &&\Gg(\G)\ar[rr] ^-{ \alpha} && S_*.}\] 
 Morphisms in $\factcat((\Gg^1, \alpha^1), (\Gg^2, \alpha^2))$ are commuting diagrams in ${\GSp}_{\TTp}$
 \begin{equation}\label{factcat mor}
 \xymatrix{&& {\Gg^1}(\G)\ar[dd]_{g} \ar[drr]^{ \alpha^1}&&\\
 \G \ar[urr]^-{[\Gg^1]}\ar[drr]_-{[\Gg^2]}&&&& S_*\\
 &&{\Gg^2}(\G) \ar[urr]_{ \alpha^2}&&}
 \end{equation} such that $g $ is a morphism in $ \Grp{\subset}{\GSp}_{\TTp}$.

 \end{defn}

By \cite[Proposition~2.5]{BMW12}, the monad $\TTp$ has arities $\Grp$ if the following lemma holds for all pointed graphical species $S_*$, all graphs $\G \in \Grp$ and all $\beta \in \GSp(\G, \Tp S_*)$:


\begin{lem}\label{connected} 
	The category $\factcat$ is connected.
 \end{lem}

\begin{proof} This follows easily from the discussion above, and in particular \cref{T*H}:

Let $S_*$ be a pointed graphical species. For $X$ a finite set, $S_*$-structured $X$-graphs $(\X^1, \alpha^1), (\X^2, \alpha^2)$ represent the same element of $\Tp S_X$ if and only if they are similar in $\ovP{S_*}{X\XGrsimp} \cong \GSp(\CX, \Tp S_*)$. So, by \cref{Klgr}, the lemma holds whenever $\G = (\shortmid)$ or $\G = \CX$ is a corolla (including $\C_\nul$, by (\ref{0 connected})).

Now, let $\G \not \cong \C_\nul$ be any connected graph. Elements of $\GSp(\G, \Tp S) \cong \Tp S(\G)$ are represented by non-degenerate $\G$-shaped graphs of $S_*$-structured graphs. Since there is no object of the form $(\C_\nul, b)$ in $\elG$, two such non-degenerate $S_*$-structured graphs of graphs, $\Gg^1_{S_*}, \Gg^2_{S_*}$ represent the same element of $\Tp S(\G)$ if and only if  for all $(\CX[X_b], b) \in \elG$, 
$\Gg^1_{S_*} (\CX[X_b],b) \sim \Gg^2_{S_*} (\CX[X_b],b)$ in $ \ovP{S_*}{X_b \XGrsimp }$, whereby the colimits $\Gg^1_{S_*} (\G) $ and $\Gg^2_{S_*}(\G)$ are also similar in $\ovP{S_*}{\Grp}$. Hence, $\factcat$ is connected by \cref{prop: TpX}. 
\end{proof}

 \cref{nerve theorem} now follows from \cite[Sections~1~\&~2]{BMW12}.
 
 \begin{proof}[Proof of \cref{nerve theorem}]The category $\Grp$ is dense in $ \GSp$. By \cite[Proposition~2.5]{BMW12}, the monad $\TTp$ on $\GSp$ has arities $\Grp$ if and only if $\factcat$ is connected for all $S_*$, $\G$ and $\beta \in \GSp (\G, \Tp S_*)$. 
 	
 Hence, by \cref{connected}, $\TTp$ has arities $\Grp \subset \GSp$ and the induced nerve functor $N\colon \CSM \to \pr{\Klgr}$ is fully faithful by \cite[Propositions~1.5~\&~1.9]{BMW12}.
 	
 	 Moreover, by \cite[Theorem~1.10]{BMW12}, its essential image is the subcategory of those presheaves on $\Klgr$ whose restriction to $\Grp$ are in the image of the fully faithful embedding $\GS_* \hookrightarrow \pr{\Grp}$. 
 	
In other words, a presheaf $P$ on $\Klgr$ is in the essential image of $N$ if and only if, for all $\G$, $P(\G) = \mathrm{lim}_{(\C,b) \in \elpG} P(\C)$. By finality of $\elG \subset \elpG$, this is the case precisely when $ P(\G) = \mathrm{lim}_{(\C,b) \in \elG} P(\C)$. 
 \end{proof}

\begin{rmk}\label{rmk: Gr no arities} To see that the modular operad monad $\DD\TT$ on $\GS$ does not have arities, let us use the method of \cite[Section~2]{BMW12} to construct its \textit{unpointed} factorisation categories.

For any graphical species $S$ and graph $\G$, 
$\GS_{\DD\TT}(\yet \G, S) 
\cong {\GSp} (\yetp \G, \Tp S^+) \cong \Tp S^+(\G)$ canonically.

So, a morphism $\beta \colon \yet \G \to S$ in the Kleisli category $\GS_{\DD\TT}$ is represented by a $\G$-shaped graph of graphs $\Gg$ with colimit $\Gg(\G)$, and a $DS$-structure $\alpha\in \GS(\Gg(\G), DS) \cong DS (\Gg(\G))$. 

Such pairs $(\Gg, \alpha)$ are the objects of the unpointed factorisation category $\factcatup$. Morphisms in $ \factcatup((\Gg, \alpha), (\Gg', \alpha'))$ are morphisms in $\Gr(\Gg(\G), \Gg(\G'))$ making the diagram (\ref{factcat mor}) commute. 

By \cite[Proposition~2.5]{BMW12}, $\DD\TT$ has arities $\Gr$ if and only if $\factcatup$ is connected for all $S$, $ \G$, and $\beta$.

To see that this is not the case, let $S = \yet(\shortmid)$, and so $TS \cong S$. Let $\G = \C_\nul$ and let $\beta = o= z \colon \C_\nul \to (\shortmid)$. Then the diagrams $\C_\nul \to \C_\nul \xrightarrow {z} (\shortmid)$, and $\C_\nul \to \W \xrightarrow {\kappa}(\shortmid)$ describe objects in $\factcatup$. Since there are no non-trivial morphisms in $\Gr$ with domain or codomain $\C_\nul$, these objects are in disjoint components of 
$\factcatup$. Therefore, $\factcatup$ is not connected and $\DD\TT$ does not have arities $\Gr$.

	\end{rmk}

\subsection{ Weak modular operads} 

In \cite{HRY19a, HRY19b}, Hackney, Robertson and Yau have proved a version of \cref{nerve theorem} in terms of a bijective on objects subcategory $U $ of $\Klgr$ that was constructed precisely so as to have a generalised Reedy structure. The inclusion $U \hookrightarrow \CSM$ is not fully faithful since the category $U$ does not contain those morphisms in $\Grp \hookrightarrow \Klgr$ that factor through $z\colon \C_\nul \to (\shortmid)$ or $\kappa^m \colon \to (\shortmid)$, $m \geq 1$, nor does it contain any morphisms of $\Gr$ that are not embeddings. 
 However, by \cite[Theorem~3.6]{HRY19b}, $U$ is dense in $\CSM$ and hence induces a fully faithful nerve. 

Furthermore, by \cite[Theorem~3.8]{HRY19a}, the category $\prs{U}$ of $\sSet$-valued presheaves on $U$ admits a  cofibrantly generated model structure, obtained by localising the Reedy model structure at the \textit{Segal morphisms}\[ \mathrm{lim}_{(\C, b) \in \elG} P(\C) \longrightarrow P(\G),\] and the fibrant objects for this model structure are those simplicial presheaves on $U$ that satisfy the \textit{weak Segal condition}
\begin{equation}
\label{eq: weak segal}P(\G) \simeq \mathrm{lim}_{(\C,b) \in \elG} P(\C), \ \text { for all graphs } \G \in U. 
\end{equation}

The method of \cite{HRY19a, HRY19b} cannot be applied in the current case since there is no (obvious) generalised Reedy structure on $\Klgr$. However, in \cite{CH15}, Caviglia and Horel describe a general class of rigidification results whereby, given a dense inclusion $\DCat \hookrightarrow \CCat$ satisfying certain conditions, an equivalence is established between $\sSet$-valued presheaves on $\DCat$ that satisfy a weak Segal condition, and $\CCat$ objects internal to $\sSet$ that satisfy the Segal condition on the nose. In \cite[Section~7]{CH15}, this result is applied to a certain class of monads with arities. 
This leads directly to the following corollary of \cref{nerve theorem}:

\begin{cor}\label{cor: weak}
There is a model category structure on the category $\prs{\Klgr}$ 
of functors $P\colon \Klgr^\mathrm{op} \to \sSet$ whose fibrant objects are those $P$ that satisfy the weak Segal condition:
\begin{equation}
\label{eq: weak segal 1} P(\G) \simeq \mathrm{lim}_{ (\C, b) \in \elG}P(\C) \ \text{ for all graphs } \G \in \Gr.
\end{equation}
\end{cor}

\begin{proof}
The monad $\TTp$ has arities $\Grp$ and $\elG$ is connected and essentially small for all connected graphs $\G$. Therefore the assumptions of \cite[Assumptions~7.9]{CH15} are satisfied. By \cite[Section~7.5]{CH15}, $\CSM$ is equivalent to the category of models in $\Set$ of the limit sketch $L = (\Grp,\{(\G \ov {\fisinv}^\mathrm{op})_{\G\in \Gr} \})$. 

Moreover, there is a Segal model structure on the category of $\sSet$ valued models for $L$ and, by \cite[Proposition~7.1]{CH15}, this can be transferred to a model structure on $\prs{\Klgr}$ whose fibrant objects are those presheaves that satisfy the weak Segal condition.
\end{proof}

In current work with M.\ Robertson, we are comparing the existing models for weak modular operads. We expect that there is a direct Quillen equivalence between the model structure on $\prs{\Klgr}$ of \cref{cor: weak} and the model structure on $\prs{U}$ of \cite{HRY19a}. 

 \begin{rmk} \label{comment on JK HRY}\label{more on JK HRY} \cref{nerve theorem} was originally formulated in \cite{JK11}, in terms of the graphical category $\overline{ Gr }$, whose morphisms are described in \cite[Section~6]{JK11}. This is the bijective on objects subcategory of $\Klgr$ that does not contain any morphisms in $\Grp$ that factor through $z\colon \C_\nul \to (\shortmid)$ or $\kappa \colon \W \to (\shortmid)$. 
In particular $\overline{ Gr }$ does not embed fully in $\CSM$. 
 
There are bijective on objects inclusions $U \subset \overline {Gr} \subset \Klgr$. Hence, since $\Klgr$ and $U$ are both dense in $\CSM$, so is $ \overline {Gr}$, and the inclusion yields a fully faithful nerve functor $\CSM \to \pr{ \overline {Gr}}$ whose essential image satisfies the same Segal condition (\ref{eq. Segal nerve sec}). (See also \cite[Theorem~3.6~\&~Section~4]{HRY19b} for more details.)
 
 
 
 
 \end{rmk}

	\bibliography{CSMbib}{}
	\bibliographystyle{plain}

\end{document}